\documentclass[a4paper,11pt]{article} 

\usepackage[utf8]{inputenc} 

\usepackage[margin=1in]{geometry} 
\usepackage[english]{babel}
\usepackage{csquotes}
\usepackage{graphicx} 

\usepackage{amssymb}
\usepackage{stmaryrd}
\usepackage{xcolor} 
\usepackage[bookmarks=false]{hyperref} 

\usepackage[symbol]{footmisc}

\usepackage{booktabs} 
\usepackage{array} 
\usepackage{paralist} 
\usepackage{verbatim} 
\usepackage{subfig} 

\usepackage{fancyhdr} 
\usepackage{sectsty}
\allsectionsfont{\sffamily\mdseries\upshape} 

\usepackage[nottoc,notlof,notlot]{tocbibind} 
\usepackage[titles,subfigure]{tocloft} 

\usepackage{times}
\usepackage{amsmath, amsfonts, bbm, dsfont, mathrsfs}

\DeclareUnicodeCharacter{03A6}{~}
\DeclareUnicodeCharacter{0393}{~}

\usepackage[
backend=biber,
style=alphabetic
]{biblatex}

\usepackage{amsthm}
\usepackage{mathtools}
\usepackage{tikz-cd} 
\tikzset{
	labl/.style={anchor=south, rotate=90, inner sep=.5mm}
}
\usepackage{xfrac} 

\newcommand{\xrightarrowdbl}[2][]{%
	\xrightarrow[#1]{#2}\mathrel{\mkern-14mu}\rightarrow
}

\usepackage{enumitem}

\newtheorem{theo}{Theorem}[section]
\newtheorem{prop}[theo]{Proposition}
\newtheorem{lemma}[theo]{Lemma}
\newtheorem{coro}[theo]{Corollary}

\newtheorem{defiprop}[theo]{Definition / Proposition}

\theoremstyle{definition}
\newtheorem{defi}[theo]{Definition}
\newtheorem{rem}[theo]{Remark}
\newtheorem{ex}[theo]{Example}

\newcommand{\Z}{\mathbb{Z}}
\newcommand{\N}{\mathbb{N}}
\newcommand{\R}{\mathbb{R}}

\newcommand{\sep}{\hrm{sep}}

\newcommand{\coindu}[3]{\hrm{Coind}_{#1}^{#2}(#3)}
\newcommand{\bigcoindu}[3]{\hrm{Coind}_{#1}^{#2}\left( #3\right)}
\newcommand{\qp}[0]{\mathbb{Q}_p}
\newcommand{\zp}{\mathbb{Z}_p}

\newcommand{\zptimes}{\mathbb{Z}^{\times}_p}

\newcommand{\cg}{\hcal{G}}
\newcommand{\ch}{\hcal{H}}

\newcommand{\hrm}[1]{\mathrm{#1}}
\newcommand{\hcal}[1]{\mathcal{#1}}

\newcommand{\Oe}{\hcal{O}_{\hcal{E}}}
\newcommand{\Oeplus}{\hcal{O}^+_{\hcal{E}}}
\newcommand{\Oed}{\hcal{O}_{\hcal{E}_{\Delta}}}

\newcommand{\Oehat}{\hcal{O}_{\widehat{\hcal{E}^{\hrm{ur}}}}}

\newcommand{\detale}[2]{\hrm{Mod}^{\hrm{\acute{e}t}}\left(#1,#2\right)}
\newcommand{\detaleproj}[2]{\hrm{Mod}_{\hrm{prj}}^{\hrm{\acute{e}t}}\left(#1,#2\right)}
\newcommand{\detaledvproj}[3]{\hrm{Mod}_{#3 \hrm{\text{-}prjdv}}^{\hrm{\acute{e}t}}(#1,#2)}
\newcommand{\dmod}[2]{\hrm{Mod}\left(#1,#2\right)}

\newcommand{\cdetale}[2]{\mathscr{M}\hrm{od}^{\hrm{\acute{e}t}}(#1,#2)}
\newcommand{\cdetaleproj}[2]{\mathscr{M}\hrm{od}_{\hrm{prj}}^{\hrm{\acute{e}t}}\left(#1,#2\right)}
\newcommand{\cdetaledvproj}[3]{\mathscr{M}\hrm{od}_{#3 \hrm{\text{-}prjdv}}^{\hrm{\acute{e}t}}(#1,#2)}

\newcommand{\quot}[2]{\raisebox{.15em}{$#1$}/\raisebox{-.15em}{$#2$}}

\newcommand{\modulo}[1]{\,\,\, \hrm{mod} \,\, #1}

\newcommand{\colim}[1]{\mathop{\hrm{colim}}}

\newcommand{\gal}[2]{\hrm{Gal}\left(#1|#2\right)}

\newcommand{\myrightleftarrow}[2]{\mathrel{\substack{\xrightarrow{#1} \\[0.15ex] \xleftarrow[#2]{}}}}

\graphicspath{./Visuels} 

\title{\textsc{Study of various categories graviting around $(\varphi,\Gamma)$-modules}}
\author{Nataniel Marquis}
		
\usepackage[T1]{fontenc}

\begin{document}

	\begin{titlepage}
	\parindent=0pt

	\phantom{ }

	\vspace{20ex}
	
	
	\hrulefill
	\begin{center}\bfseries \Huge \textbf{Study of various categories gravitating around $(\varphi,\Gamma)$-modules}
	\end{center}
	\hrulefill

	\vspace*{0.5cm}
	
	\renewcommand*{\thefootnote}{\fnsymbol{footnote}}
	
	\begin{center}\Large
		by Nataniel Marquis\footnote[2]{\underline{n.marquis@uni-muenster.de}. Universität Münster, Mathematisches Institut, Orléans-Ring 12, 48149 Münster, Germany.}

	\end{center}

	\vspace{3 cm}
	
	\textit{\textbf{Abstract.}} \textemdash \,\, Functors involved in Fontaine equivalences decompose as extension of scalars and taking of invariants between full subcategories of modules over a topological ring equipped with semi-linear continuous action of a topological monoid. We give a general framework for these categories and the functors between them. We define the categories of étale projective $\hcal{S}$-modules over $R$ to englobe categories that will correspond by Fontaine-type equivalences to finite free representations of a group. We study their preservation by base change, taking of invariants by a normal submonoid of $\hcal{S}$ and coinduction to a bigger monoid. We define and study categories corresponding to finite type continuous representations over $\zp$ through the notions of finite projective $(r,\mu)$-dévissage and of topological étale $\hcal{S}$-modules over $R$. 
	
	
	\vspace*{\stretch{3}}
	
	\begin{flushright}
	\end{flushright}

\end{titlepage}
	
\renewcommand*{\thefootnote}{\arabic{footnote}}

	\tableofcontents
	
	\vspace{3 cm}
	
	\section{Introduction}

	The starting point of this article is Fontaine equivalence of categories \cite[Theorem 3.4.3]{Fontaine_equiv}. Let $\qp$ be the field of $p$-adic numbers, let $\overline{\qp}$ be a fixed Galois closure and let $\cg_{\qp}:=\gal{\overline{\qp}}{\qp}$ be its absolute Galois group. Fontaine considers the $\zp$-algebra $\Oe:=\left(\zp\llbracket X \rrbracket [X^{-1}]\right)^{\wedge p}$, with the weak topology for which a basis of neighborhood of zero  is given by $\left(p^n \Oe + X^m \zp\llbracket X \rrbracket\right)_{n,m\geq 0}$, and with a continuous lift of Frobenius $\varphi$ and a continuous action of $\Gamma:=\zptimes$.  He also constructs a $\zp$-algebra $\Oehat$ equipped with a topology also called the weak topology and with commuting Frobenius $\varphi$ and action of $\cg_{\qp}$. The action of $\cg_{\qp}$ is continuous for the weak topology and the action of its subgroup $\ch_{\qp} := \gal{\qp(\mu_{p^{\infty}})}{\qp}$ is continuous for the $p$-adic topology. 
	
	From these definitions, Fontaine constructs functors between three categories. First, the category $\hrm{Rep}_{\zp} \cg_{\qp}$ of finite type $\zp$-linear representations $V$ of $\cg_{\qp}$ which are continuous for the $p$-adic topology on $V$. Second, the category $\cdetale{\varphi^{\N}\times \Gamma}{\Oe}$ of étale $(\varphi,\Gamma)$-modules. Its objects are the finite type $\Oe$-modules $D$ equipped with a $\varphi$-semilinear Frobenius $\varphi_D$ and a semilinear $\Gamma$-action, commuting with each other and such that:
	
	
	
	
	\begin{enumerate}[label=\arabic*.,itemsep=0mm]
		\item The image $\varphi_D(D)$ generates the $\Oe$-module $D$ (shortened as $D$ being étale).
	
		\item The $\Gamma$-action is continuous for the topology on $D$ corresponding to the weak topology on $\Oe$.
	\end{enumerate} Finally, Fontaine implicitly uses the category $\cdetale{\varphi^{\N}\times\cg_{\qp},\, \ch_{\qp}}{\Oehat}$ as an intermediate. Its objects are the finite type $\Oehat$-modules $D$ equipped with a $\varphi$-semilinear Frobenius $\varphi_D$ and a semilinear action of $\cg_{\qp}$ commuting with each other and such that:
	
	
	
	
	
	\begin{enumerate}[label=\arabic*.,itemsep=0mm]
		\item The image $\varphi_D(D)$ generates the $\Oehat$-module $D$.
	
		\item The $\cg_{\qp}$-action is continuous for the topology on $D$ corresponding to the weak topology on $\Oehat$. 
	
		\item The $\ch_{\qp}$-action is continuous for the $p$-adic topology on $D$.
	\end{enumerate} Fontaine's article proves the following equivalences of categories $$\mathbb{D} \, : \,  \hrm{Rep}_{\zp} \cg_{\qp} \myrightleftarrow{\Oehat\otimes_{\zp} -}{D\mapsto D^{\varphi=\hrm{Id}}} \cdetale{\varphi^{\N}\times \cg_{\qp},\,\ch_{\qp}}{\Oehat} \myrightleftarrow{D\mapsto D^{\ch_{\qp}}}{\Oehat\otimes_{\Oe} -} \cdetale{\varphi^{\N}\times \Gamma}{\Oe}\, : \, \mathbb{V},$$ which describes more explicitly the representations of $\cg_{\qp}$.
	
	The three categories involved are categories of semilinear representations of monoids (respectively the monoids $\cg_{\qp}$, $(\varphi^{\N}\times \Gamma)$ and $(\varphi^{\N}\times \cg_{\qp})$). Moreover, the proof reduces to the case of $p$-torsion modules by dévissage. Recent generalisations by \cite{zabradi_equiv} and \cite{zabradi_kedlaya_carter} link representations of a finite product of $\cg_{K}$, where $K$ is a $p$-adic local field, to multivariable cyclotomic $(\varphi,\Gamma)$-modules. In these cases, the dévissage step contains hidden algebraic and topological subtleties that are not always fully detailed. While trying to generalise \cite{zabradi_kedlaya_carter} to a Lubin-Tate setting, these difficulties subtleties became more precise and I wanted to produce an automatisation of the dévissage step in a Fontaine-like equivalence that would take into account these subtleties. Therefore, this article proposes a formalism that encapsulates the whole panel of categories appearing in such equivalences and the functors between these categories. In this language, we can rewrite Fontaine's original equivalences, treat the subtleties \cite{zabradi_kedlaya_carter} and we use it in \cite{nataniel_fontaine} to prove Lubin-Tate and plectic generalisations. This formalism also allows to consider families of Galois representations : namely, we can recover\footnote{Note that the definition of \cite[p. 655]{dee_families} seems to miss a topological compatibility about the action of $\Gamma$ and the topology induced by the $\mathfrak{m}_R$-topology on $\Phi \Gamma$-modules.} \cite[Theorem 2.2.1]{dee_families} for represensations $V$ over a complete regular unramified local ring $R$ of characteristic zero and finite residue field , provided that $p^n V/p^{n+1} V$ are finite projective over $R/p$.

	\vspace{0.2 cm}
	
	
	The setting of this article is the following.
	
	
	\begin{defi} Let $\hcal{S}$ be a monoid acting on a commutative ring $R$. We define \textit{the category of $\hcal{S}$-modules over $R$}, denoted by $\dmod{\hcal{S}}{R}$. Its objects are the $R$-modules equipped with a semilinear action of $\hcal{S}$.
	\end{defi}  
	
	We study a number of full subcategories of $\dmod{\hcal{S}}{R}$, for which the three categories of Fontaine are examples, as well as their preservation by various operations (scalar extension, taking invariants, coinduction). We now introduce our finest full subcategory of $\dmod{\hcal{S}}{R}$, which is suitable for both dévissage and topological considerations.
	
	\begin{defi}\label{intro_defdvproj}
		Let $R$ be a ring and $r\in R$. An $R$-module $M$ is said to have \textit{projective $(r,\mu)$-dévissage}\footnote{The letter $\mu$ stands for multiplicative. Another dévissage by torsion part will be used in the fourth section and called the $(r,\tau)$-dévissage.} if each subquotient $r^n M/r^{n+1} M$ is a finite projective $R/r$-module of constant rank over $\hrm{Spec}(R/r)$.
	\end{defi}
	
	\begin{defi}\label{intro_topfaible}
	Let $R$ be a topological ring. Let $M$ be a finite type $R$-module and $R^k \twoheadrightarrow M$ be a quotient map. The quotient topology on $M$ is called \textit{the initial topology}. It does not depend on the chosen quotient map.
	\end{defi}
	
	\begin{defi}\label{def_intro}
		Let $\hcal{S}$ be a topological monoid, let $R$ be a ring equipped with a ring topology $\mathscr{T}$ and with an $\hcal{S}$-action continuous for $\mathscr{T}$. Let $\hcal{S}'\triangleleft \hcal{S}$ be a normal submonoid and $\mathscr{T}'$ be a ring topology on $R$ for which the $\hcal{S}'$-action is continuous.
		
		Let $r\in R^{\hcal{S}'}$ be such that $R$ is $r$-adically complete and separated, $r$-torsion-free, and such that
		$$\forall s\in \hcal{S}, \,\,\, \varphi_s(r)R=rR.$$
		
		
		The category $\cdetaledvproj{\hcal{S},\hcal{S}'}{R}{r}$ is the full subcategory of $\dmod{\hcal{S}}{R}$ whose objects are the $\hcal{S}$-modules $D$ over $R$ such that:
		\begin{enumerate}[itemsep=0mm]
			\item For every $s\in \hcal{S}$, the image of $D$ by the action of $s$ generates $D$ as an $R$-module\footnote{Its actually a simplification to avoid introducing here the linearisations.}.
			
			\item The $R$-module $D$ is of finite presentation with bounded $r^{\infty}$-torsion, i.e. $D[r^n]=D[r^{\infty}]$ for some $n\geq 1$.
			
			\item The module $R$ has finite projective $(r,\mu)$-dévissage.
			
			\item The $\hcal{S}$-action is continuous for the initial topology on $D$ corresponding to $\mathscr{T}$, i.e. that the action map $\hcal{S}\times D \rightarrow D$ is continuous.
			
			\item The $\hcal{S}'$-action is continuous for the initial topology on $D$ corresponding to $\mathscr{T}'$.
		\end{enumerate}
		
		The category $\cdetaleproj{\hcal{S},\hcal{S}'}{\sfrac{R}{r}}$ is defined in a similar way, replacing the third condition with "the $R/r$-module $D$ is projective of constant rank" and forgetting the condition about $r^{\infty}$-torsion.
		
		We may omit $\hcal{S}'$ if $\mathscr{T}'$ is the trivial topology; in this case, the fifth condition is automatic no matter the other data.
	\end{defi}
	
	Our main results automate the dévissage steps while proving that a Fontaine-type functor produces finitely presented modules.
	
	\begin{prop}\label{intro_prop_1}[See Proposition \ref{ex_cdetaledvproj}]
		Let $\hcal{S}$ be a topological monoid, let $A$ and $R$ be topological rings equipped with continuous $\hcal{S}$-actions and let $f\, : \, A\rightarrow R$ be an $\hcal{S}$-equivariant continuous ring morphism. Let $a\in A$ be such that:
		
		\begin{itemize}[itemsep=0mm]
			\item The ring $A$ is $a$-adically complete, separated and $a$-torsion-free.
			
			\item The ring $A$ verifies $$\forall s\in \hcal{S}, \,\,\, \varphi_s(a)A=aA.$$
			
			\item The ring $R$ is $f(a)$-adically complete, separated and $f(a)$-torsion-free.
		\end{itemize} Then, the functor $$D\mapsto R\otimes_A D$$ sends $\cdetaledvproj{\hcal{S}}{A}{a}$ to $\cdetaledvproj{\hcal{S}}{R}{f(a)}$.
	\end{prop}
	
	\vspace{0.25 cm}
	\begin{theo}\label{intro_theo_1}[See Theorem \ref{inv_cdvdetaleproj_dévissage}]
		Fix the same setup as in Definition \ref{def_intro}. Suppose that:
		
		\begin{itemize}[itemsep=0mm]
			\item The inclusion $R^{\hcal{S}'}/r\subset R/r$ is fully faithful.
			
			\item We have $\hrm{K}_0(R^{\hcal{S}'}/r)=\Z$.
			
			\item The topology $\mathscr{T}'$ is coarser than the $r$-adic topology and for every $R$-module $D$ with finite projective $(r,\mu)$-dévissage, the initial topology on $D$ induces the initial topology on $rD$ and $D[r]$.
			
			\item We have $\hrm{H}^1_{\hrm{cont}}(\hcal{S}',R/r)=\{0\}$ for $\mathscr{T}'$.
			
			\item For every $D$ in $\cdetaleproj{\hcal{S},\, \hcal{S}'}{\sfrac{R}{r}}$, the comparison morphism $$ R\, \mathop{\otimes}_{R^{\hcal{S}'}}  D^{\hcal{S}'}\rightarrow D, \,\,\, t\otimes d \mapsto td $$ is an isomorphism.
			
		\end{itemize}
		
		
		Then, the comparison morphism is an isomorphism for every object of $\cdetaledvproj{\hcal{S},\, \hcal{S}'}{R}{r}$ and the functor $D\mapsto D^{\hcal{S}'}$ sends $\cdetaledvproj{\hcal{S},\, \hcal{S}'}{R}{r}$ to $\cdetaledvproj{\sfrac{\hcal{S}}{\hcal{S}'}}{R^{\hcal{S}'}}{r}$.
	\end{theo}

	 The $\hrm{K}$-theory condition is often verified in our context : if $R^{\hcal{S}'}$ is the $p$-completion of the localisation of an unramified complete local $\zp$-algebra and $r=p$, the condition $\hrm{K}_0(R^{\hcal{S}'/r})=\Z$ is verified. The rings appearing in multivariable Fontaine equivalences are of this form.
	 
	 Note that applying these result to Fontaine's equivalence introduces a condition on the $(p,\mu)$-dévissage at the level of $(\varphi,\Gamma)$-modules. These conditions are often automatic for a finitely presented étale module : in \cite{Fontaine_equiv}, Fontaine worked with discrete valuation rings and \cite[Proposition 2.2]{zabradi_equiv} proves that the action of $\Gamma_{\Delta}$ makes the condition automatic\footnote{In its notation, the $\varphi_{\Delta}$-stable ideals of $E_{\Delta}$ are trivial, which allows to drop the $\Gamma_{\Delta}$-action in the proof.}. They might not be automatic in a multivariable perfectoid setting\footnote{The proof of \cite[Proposition 2.2]{zabradi_equiv} uses noetherianity.},which might suggest that the $(r,\mu)$-dévissage condition is a necessary condition on the image of Fontaine-type functors. This condition is usually sufficient. For example, it gives the correct description of the essential image of an analogue of \cite[Theorem 4.30]{zabradi_kedlaya_carter}, for finite type representations and imperfect coefficient ring rather than finite free representations and perfect coefficient ring. 
	 
	 \vspace{0.5cm}
	
	Our study of the $(r,\mu)$-dévissage is carried out through the following structure theorem:
	
	\begin{theo}\label{intro_complet_pf}[See Theorem \ref{complet_pf}]
		Let $R$ be a ring and $r\in R$ be such that $R$ is $r$-adically complete and separated, $r$-torsion-free. For every $R$-module $M$, the following are equivalent:
		
		\begin{enumerate}[label=\roman*),itemsep=0mm]
			\item $M$ is $r$-adically complete and separated with finite projective $(r,\mu)$-dévissage.

			\item $M$ is finitely presented with finite projective $(r,\mu)$-dévissage and bounded $r^{\infty}$-torsion.
			
			\item There exists $N\geq 1$, a finite projective $R$-module of constant rank $M_{\infty}$ and an $r^N$-torsion $R$-module with finite projective $(r,\mu)$-dévissage $M_{\hrm{tors}}$ such that 
			
			$$M\cong M_{\infty} \oplus M_{\hrm{tors}}.$$
		\end{enumerate}
		
			
			

		
		Suppose in addition that $\hrm{K}_0(\sfrac{R}{r})=\Z$, i.e.  every finite projective $R/r$-module is stably free. Then the above three conditions on an $R$-module $M$ are also equivalent to
		
		\begin{enumerate}[label=\roman*),itemsep=0mm]
			\setcounter{enumi}{3}
			\item There exists $N\geq 1$ and an isomorphism $$M\cong M_{\infty} \oplus \bigoplus_{1\leq n \leq N} M_n$$ where $M_{\infty}$ is a finite projective $R$-module of constant rank and each $M_n$ is a finite projective $R/r^n$-module of constant rank.
		\end{enumerate}
		\end{theo} Condition $iv)$ is an extension of the structure theorem for finite type modules over a discrete valuation ring. This is a powerful tool for studying these modules: for instance, we make extensive use of the equivalence $i)\Leftrightarrow ii)$ in a non topological version of Theorem \ref{intro_theo_1} and we crucially use condition $iv)$ to simplify what "continuous" means for modules with finite projective $(r,\mu)$-dévissage.
	
	We also obtain a structure theorem for the modules obtained by Fontaine equivalences. Recall that the structure of $\Oe$-modules underlying univariable $(\varphi,\Gamma)$-modules does not come from the functor applied to representations but from the fact that $\Oe$ is principal. It may happen that no decomposition of a considered $\zp$-representation $V$ is Galois invariant. In this case, the condition $iv)$ is not automatically preserved by descent and we need to descend the $(r,\mu)$-dévissage condition then use Theorem \ref{intro_complet_pf} in order to recover condition $iv)$.
	
	
	\vspace{1 cm}
	
	
	In \textit{section 2} of this article, we define the category $\dmod{\hcal{S}}{R}$, its full subcategories $\detale{\hcal{S}}{R}$ and $\detaleproj{\hcal{S}}{R}$ then endow them, when possible, with symmetric closed monoidal structures. In \textit{section 3}, we study three operations on these categories: scalar extension, taking invariants by a normal submonoid of $\hcal{S}$ and coinducting to a monoid containing $\hcal{S}$. In \textit{section 4}, we introduce the notion of finite projective $(r,\mu)$-dévissage and carve the subcategory $\detaledvproj{\hcal{S}}{R}{r}$ of modules with finite projective $(r,\mu)$-dévissage. We prove that this category is stable under tensor product, study when it is stable by internal $\hrm{Hom}$, and give results on its stability by scalar extension and taking of invariants. In \textit{section 5}, we consider topological rings and topological monoids acting continuously on them. We first study the initial topology on finitely generated modules then recover the results of previous sections for the continuous versions $\cdetale{\hcal{S}}{R}$, $\cdetaleproj{\hcal{S}}{R}$ and $\cdetaledvproj{\hcal{S}}{R}{r}$ of our subcategories. Finally, in \textit{section 6}, we prove Fontaine equivalence for $K=\qp$ in our language.
	
	\vspace{ 1 cm}
	
	\textbf{\underline{Acknowledgments:}} this article is extracted from my thesis, supervised by Pierre Colmez and Antoine Ducros. At the time, I was a member of both the IMJ-PRG at Sorbonne Université and the DMA at the École Normale Supérieure-PSL. I would like to thank Pierre Colmez for his patience and careful reading of the most tedious sections, as well as for inviting me to Bonn in spring 2023, where I began to write this text. I am also grateful to Gergely Zábr\'adi and Benjamin Schraen, who agreed to act as referees for my thesis. I especially thank Benjamin Schraen who pointed a small gap in an homological argument and Antoine Ducros who double-checked the changes I had to make. I would also like to thank all the mathematicians with whom I've had informal chats related to this work : Muriel Livernet, Olivier Benoist, Kęstutis Česnavičius, Gaëtan Chenevier, Ariane Mézard, Arnaud Vanhaecke and Paul Wang. This article has certainly improved thanks to your time.

	
	
	
	
	
	\color{black}

	\clearpage
	
	\section{The category of $\hcal{S}$-modules over $R$ and some subcategories}\label{section_dmod}
	
	Our aim is to give a setup that captures both representations of Galois groups and the various $(\varphi,\Gamma)$-modules discussed in the introduction. To allow a non invertible Frobenius, we consider monoid actions; to allow lifts to representations over $\zp$, we consider coefficient rings rather than coefficient fields. Finally, we take into account the potential semilinearity of the actions.
	
	Unless explicitely stated, our rings are unital and commutative.

	\begin{defi}
		Let $\hcal{S}$ be a monoid. An \textit{$\hcal{S}$-ring $R$} is a pair formed by a ring, which we will also denote by $R$, and a morphism of monoids from $\hcal{S}$ to $\hrm{End}_{\hrm{Ring}}(R)$, denoted by $s\mapsto \varphi_s$.
	\end{defi}
	
	\begin{ex}
		\begin{enumerate}
			\item For any monoid $\hcal{S}$ and any ring $R$, the trivial action equips $R$ with the structure of a $\hcal{S}$-ring.
			
			\item Any ring $R$ in which $p=0$ with action of the absolute Frobenius is a $\varphi^{\N}$-ring.
			
			\item The ring $\Oe:= (\zp\llbracket X \rrbracket [X^{-1}])^{\wedge p}$ has a $\zp$-linear Frobenius verifying\footnote{It is characterised by its continuity with respect to the weak topology.} that $\varphi(X)=(1+X)^p-1$. It also has a $\zp$-linear action of $\Gamma:=\zptimes$ verifying that $\gamma \cdot X = (1+X)^{\gamma}-1$. This equips $\Oe$ with the structure of an $(\varphi^{\N}\times \Gamma)$-ring.
		\end{enumerate}
	\end{ex}

		Without additional precision, $\hcal{S}$ will always be a monoid and $R$ will be an $\hcal{S}$-ring.
		
		\hspace{1 cm}
	
	\subsection{The $\hcal{S}$-modules over $R$}
	
We define our most general category.
	
\begin{defi}\label{def_sring}
	Define the category $\dmod{\hcal{S}}{R}$ of \textit{$\hcal{S}$-modules over $R$}. Its objects are the pairs $(D,\varphi_{\text{-},D})$ where $D$ is an $R$-module and $$\varphi_{\text{-},D}\, : \,\hcal{S}\rightarrow \hrm{End}_{\hrm{Ab}}(D), \,\,\, s\mapsto \varphi_{s,D}$$ is a monoid morphism such that each $\varphi_{s,D}$ is $\varphi_s$-semilinear, i.e. $$\forall r \in R, \,\, d\in D, \,\, \varphi_{s,D} (rd)=\varphi_s(r) \varphi_{s,D}(d).$$ Its morphisms are the $R$-linear morphisms $f\, : \, D_1 \rightarrow D_2$ such that $$\forall s\in \hcal{S}, \,\, f\circ \varphi_{s,D_1}=\varphi_{s,D_2} \circ f.$$
\end{defi}

\begin{rem}\label{traduction_linearise}
	We can give an equivalent definition using only linear algebra.

For any $R$-module $D$ and any $s\in S$, define the \textit{$\varphi_s$-linearisation of $D$} as $$\varphi_s^* D= R\, \mathop{\otimes}_{\varphi_s,R} D$$ seen as an $R$-module via the left factor. For any $R$-linear morphism $f$, we write $\varphi_s^*f$ for its base change along $\varphi_s$. For any $\varphi_s$-semilinear endomorphism of $R$-modules $f_s \, : \, D_1 \rightarrow D_2$, the following map is a correctly defined morphism of $R$-modules:

$$f_s^* \,: \, \varphi_s^*D_1 \mapsto D_2, \,\,\,\, r\otimes d\mapsto r f_s(d).$$ Any such linear map $f_s^*$ is obtained this way from $f_s\, : \, d \mapsto f_s^*(1\otimes d)$.


Since $s\mapsto \varphi_s$ is a morphism of monoids, $$\forall (s,s')\in \hcal{S}^2, \forall D\in R\text{-}\hrm{Mod}, \,\, \exists \,\,\varphi_s^*(\varphi_{s'}^*D)\cong\varphi_{ss'}^*D$$ natural in $D$. We let the reader check that $(D,\varphi_{\text{-},D})\mapsto (D,\varphi_{\text{-},D}^*)$ is an equivalence of categories between $\dmod{\hcal{S}}{R}$ and the category of pairs $(D,\varphi_{\text{-},D}^*)$ where $D$ is an $R$-module and $\varphi_{\text{-},D}^*$ is a family of $R$-linear maps $\varphi_{s,D}^* \,: \, \varphi_s^*D \rightarrow D$ such that $$\forall (s,s')\in \hcal{S}^2, \,\,\,\begin{tikzcd}
			\varphi_{ss'}^*D \ar[rrrr,"\varphi_{ss',D}^*"] \arrow{dr}[']{}[rotate=-27,yshift=-2ex,xshift=-1.25ex]{\sim} & & & & D \\
			& \varphi_s^*(\varphi_{s'}^*D) \ar[rr,"\varphi_s^*\left(\varphi_{s',D}^*\right)"'] & & \varphi_s^*D \ar[ur,"\varphi_{s,D}^*"']
		\end{tikzcd} \,\,\, \text{commutes}.$$
\end{rem}

\begin{ex}
	The $R$-module $R$ with $\varphi_{s,R}:=\varphi_s$ belongs to $\dmod{\hcal{S}}{R}$. 
\end{ex}

\begin{rem}
	For any group $\cg$ and ring $R$, the category of $R$-linear representations of $\cg$ is precisely $\dmod{\cg}{R}$, for $\cg$ acting trivially on $R$.

	In \cite{Fontaine_equiv}, keeping the notations of the introduction, the category of étale $(\varphi,\Gamma)$-modules is a full subcategory of $\dmod{\varphi^{\N}\times \Gamma}{\Oe}$.
\end{rem}

\begin{lemma}
	The category $\dmod{\hcal{S}}{R}$ is abelian.
\end{lemma}

\begin{proof}
	It is the category of left modules over the non commutative $R$-algebra $R[\varphi_s \, |\, s\in \hcal{S}]$, where $$\forall s\in \hcal{S},\, r\in R, \,\,\,\varphi_s \times r= \varphi_s(r) \times \varphi_s. \hspace{0.5 cm}  \qedhere $$
\end{proof}

Unfortunately, this well-behaved category usually has far too many objects to be equivalent to a category of group representations. For instance, for any $\hcal{S}$-ring $R$ where $\hcal{S}$ is cancellative, then there is a fully faithful embedding of $\dmod{\hcal{S}^{\times}}{R}$ in $\dmod{\hcal{S}}{R}$ by extension of the action via zero endomorphisms.
	
	\hspace{1 cm}
	
	\subsection{\'Etale and étale projective modules}

	A Fontaine-type functor is expressed as a scalar extension to a bigger ring followed by taking Galois invariants. Let's take a closer look at finite type $\cg_{\qp}$-representations over $\zp$ inside of $\dmod{\cg_{\qp}}{\zp}$. Since $\cg_{\qp}$ is a group, linearisations of the action maps are isomorphisms\footnote{The fact that the linearisation of the action of an element is an isomorphism is true in $\dmod{\cg}{R}$, even if the action of the group $\cg$ on $R$ is not trivial. This would later be translated by saying that the category of (possibly semilinear) finite type representations of $\cg$ over a $\cg$-ring $R$ is equivalent to $\detale{\cg}{R}$.}. The underlying $\zp$-modules are also finitely presented. These two conditions are preserved by base change and taking of invariants (see Propositions \ref{constr_ext_etale} and \ref{inv_detaleproj}) for precise statements) and cut a natural subcategory of $\dmod{\varphi^{\N}\times \Gamma}{\Oe}$ in which the essential image of Fontaine-type functors must be contained. J.-M. Fontaine call them étale $(\varphi,\Gamma)$-modules. We give a general définition.

\begin{defi}\label{def_detale}
	The category of \textit{étale $\hcal{S}$-modules over $R$}, denoted by $\detale{\hcal{S}}{R}$, is the full subcategory of $\dmod{\hcal{S}}{R}$ whose objects are the finitely presented $R$-modules $D$ such that $\varphi_{s,D}^*$ is a $R$-linear isomorphism for all $s\in \hcal{S}$.
\end{defi}

\begin{rem}
	Although Fontaine's definition only requires the modules to be of finite type, he works with discrete valuation rings. For these, finite type modules and finitely presented modules coincide. Our work suggests that the right property is finite presentation, especially in cases where the base ring is not noetherian (cf. \cite{zabradi_kedlaya_carter}).
\end{rem}

Let's mention that, with noetherian and flatness properties, the category $\detale{\hcal{S}}{R}$ is again abelian.

\begin{prop}[Propositions 1.1.5 and 1.1.6 in \cite{Fontaine_equiv}]\label{detale_ab}
	
	Suppose that $R$ is noetherian and that the endomorphisms $\varphi_s$ are flat.
	\begin{enumerate}[itemsep=0mm]
		\item The category $\detale{\hcal{S}}{R}$ is abelian. More precisely, kernels and cokernels in $\dmod{\hcal{S}}{R}$ of morphisms between objects of $\detale{\hcal{S}}{R}$ are objects of $\detale{\hcal{S}}{R}$.
		
		\item Suppose in addition that $R$ is a domain of Krull dimension $\leq 1$ and that for every pair $(s,\mathfrak{m})\in \hcal{S}\times \hrm{Spm}(R)$, the ideal $\varphi_s(\mathfrak{m})$ is maximal. Then a finitely presented object of $\dmod{\hcal{S}}{R}$ lies in $\detale{\hcal{S}}{R}$ if and only if the linearisations are surjective, i.e. every image of $\varphi_{s,D}$ generates $D$ as an $R$-module.
		
		Moreover, an extension of two $\hcal{S}$-modules over $R$ is étale if and only if the corresponding subobject and quotient are étales.
	\end{enumerate}
\end{prop}

\begin{rem} To deal with mod $p$ representations, Fontaine only needs the coefficient rings $\mathbb{F}_p$, $E\!:=\!\mathbb{F}_p(\!(X)\!)$ and $E^{\sep}$, all of which happen to be fields. The theories of modules over them are thus greatly simplified. Recent litterature, on the other hand, is packed with $(\varphi,\Gamma)$-modules variants: see \cite{breuil2023multivariable}, \cite{breuil2023multivariable}, \cite{zabradi_equiv}, \cite{zabradi_gln}, \cite{zabradi_coh} and \cite{zabradi_kedlaya_carter} for multivariable variants, see \cite{emerton2022moduli} for families. Z\'abr\'adi's multivariable $(\varphi,\Gamma)$-modules in characteristic $p$ have $$E_{\Delta}:= \mathbb{F}_p\llbracket X_1,\ldots,X_{n-1}\rrbracket[X_1^{-1},\ldots, X_{n-1}^{-1}]$$ as underlying ring. It is not a field anymore, leading to a projectivity condition on the essential image of Fontaine-type functors (see \cite[Proposition 2.2]{zabradi_equiv} and \cite[Theorem 4.6]{zabradi_kedlaya_carter}). The article \cite{zabradi_kedlaya_carter} even consider perfectoid rings which are neither noetherian, nor domains. Keep in mind that for these perfectoid coefficient rings, the abelianity of $\detale{\hcal{S}}{R}$ does not hold a priori. \end{rem}
 
For finite dimensional $\mathbb{F}_p$-representations of $\cg_{\qp}$, the underlying modules are finite projective over $\mathbb{F}_p$. This is preserved by scalar extension and taking of invariants (see Propositions \ref{constr_ext_etale} and \ref{inv_detaleproj}).

\begin{defi}
	The category of \textit{étale projective $\hcal{S}$-modules over $R$}, denoted by $\detaleproj{\hcal{S}}{R}$, is the full subcategory of $\detale{\hcal{S}}{R}$ whose objects have a finite projective $R$-module of constant rank as underlying $R$-module.
\end{defi}

\begin{ex}
The $R$-module $R$ equipped with $\varphi_{s,R}$ is an object of $\detaleproj{\hcal{S}}{R}$. Notice that $\varphi_s$ is not required to be injective : even if $\varphi_s^*R$ seen as an $R$-algebra from the right is has structural morphism $\varphi_s$ and can have torsion, the map $\varphi_{s,R}^*$ only sees its $R$-algebra structure from the left.
\end{ex}

\begin{rem}
	For a group $\cg$ acting on a ring $R$, the category $\detaleproj{\cg}{R}$ is the category of  $R$-semilinear representations of $\cg$ on finite projective modules of constant rank.
	
	 If $R$ is a field, then $\detaleproj{\hcal{S}}{R}$ and $\detale{\hcal{S}}{R}$ coincide.
\end{rem}

Unfortunately, there is no general reason for $\detaleproj{\hcal{S}}{R}$ to be abelian without some Fontaine-type equivalence.

\hspace{1 cm}

	\subsection{Closed symmetric monoidal structure}
	
	Our categories can be endowed with a closed symmetric monoidal structure. The corresponding internal $\hrm{Hom}$ recovers the usual $\hrm{Hom}$ set by taking invariants.
		
	\begin{prop}\label{monoidal_structure}
		Let $D_1,D_2$ be two objets of $\dmod{\hcal{S}}{R}$. 
		
		\begin{enumerate}[itemsep=0mm]
			\item For each $s\in \hcal{S}$ the application $$\varphi_{s,D_1}\times \varphi_{s,D_2} \, : \, D_1 \times D_2 \rightarrow D_1\, \mathop{\otimes}_R D_2, \,\,\,\, (d_1,d_2)\mapsto \varphi_{s,D_1}(d_1)\otimes \varphi_{s, D_2}(d_2)$$ factors through $D_1\otimes D_2$ as a $\varphi_s$-semilinear morphism, which we call $\varphi_{s,(D_1\otimes_R D_2)}$.
			
			\item The map $[s\mapsto \varphi_{s,D_1\otimes_R D_2}]$ endows $\left(D_1\otimes_R D_2\right)$ with a structure of an $\hcal{S}$-module over $R$. It represents the functor $$\dmod{\hcal{S}}{R} \rightarrow \hrm{Set}, \,\,\, D\mapsto \left\{ f\, : \, D_1\times D_2 \rightarrow D \, \bigg|\, \substack{f \,\, \text{is }R\text{-bilinear} \\ \forall s, \,\,\, f\,\circ\, \left(\varphi_{s,D_1} \times \varphi_{s,D_2}\right)=\varphi_{s,D}\,\circ\, f}\right\}.$$
			
			\item If $D_1$ and $D_2$ are objects of $\detale{\hcal{S}}{R}$, then $\left(D_1\otimes_R D_2\right)$ is also étale and represents the same functor from $\detale{\hcal{S}}{R}$.
			
			\item If $D_1$ and $D_2$ are objects of $\detaleproj{\hcal{S}}{R}$, then $\left(D_1\otimes_R D_2\right)$ is also étale projective and represents the same functor from $\detaleproj{\hcal{S}}{R}$.
		\end{enumerate}
	For any of these categories, $-\otimes_R -$ is a bifunctor.
	\end{prop}
	
	\begin{proof}
		\begin{enumerate}[itemsep=0mm]
		\item We give a proof highlighting the back and forth between semilinear morphisms and their linearisations. Since $\varphi_{s,D_1}^*$ and $\varphi_{s,D_2}^*$ are $R$-linear, their tensor product is well defined. So we can consider the $R$-linear composition
		
		$$f_s^*\, : \,\varphi_s^*\left(D_1 \, \mathop{\otimes}_R D_2\right) \xrightarrow[]{\sim} \varphi_s^* D_1 \, \mathop{\otimes}_R \varphi_s^* D_2 \xrightarrow[]{\varphi_{s,D_1}^*\otimes \varphi_{s,D_2}^*} D_1 \, \mathop{\otimes}_R D_2,$$ where the first isomorphism is a tensor product identity. The image of $d_1\otimes d_2$ under its delinearisation can be computed as $$f_s^*(1\otimes(d_1\otimes d_2))=\left[\varphi_{s,D_1}^* \otimes \varphi_{s,D_2}^*\right] \left( (1\otimes d_1)\otimes(1\otimes d_2)\right)= \varphi_{s,D_1}(d_1)\otimes \varphi_{s,D_2}(d_2).$$ This proves that $\varphi_{s,D_1}\!\times \!\varphi_{s,D_2}$ factors as a $\varphi_s$-semilinear morphism.
		
		\item It remains to check that $[s\mapsto \varphi_{s,(D_1\otimes_R D_2)}]$ is a morphism of monoid. For this, the delinearised setup is convenient. It is obvious using the property for $D_1$ and $D_2$ for pure tensors, then true by semilinearity.	
		
		For the universal property, show first that the usual bijection between $\{f \,: \, D_1\times D_2 \rightarrow D \, |\, f \,\, \text{is }R\text{-bilinear}\}$ and $\hrm{Hom}_R(D_1 \otimes_R D_2, D)$, natural in $D$, naturally restricts-corestricts to a bijection between the functor we want to represent evaluated at $D$ and $\hrm{Hom}_{\dmod{\hcal{S}}{R}}(D_1\otimes_R D_2,D)$. Then, the commutating squares witnessing the naturality for morphisms in $\dmod{\hcal{S}}{R}$ restrict-corestrict to the previous subsets, giving naturality of our transformation (then at the end naturality of the tensor product in $D_1$ and $D_2$).

		\item Being of finite presentation is preserved by tensor product. In the first point, we had an explicit description of $\varphi_{s,(D_1\otimes_R D_2)}^*$, which exhibits that it is an isomorphism as soon as $\varphi_{s,D_1}^*$ and $\varphi_{s,D_2}^*$ are. It proves that $\left(D_1\otimes_R D_2\right)$ is étale as soon as $D_1$ and $D_2$ are.
		
		The representability follows from the fact that $\detale{\hcal{S}}{R}$ is a full subcategory of $\dmod{\hcal{S}}{R}$.
		
		\item It remains to show that $D_1 \otimes_R D_2$ is finite projective of constant rank if $D_1$ and $D_2$ are. This follows from the fact that finite projectivity is equivalent to being finite free locally on $\hrm{Spec}(R)$ (cf. \cite[\href{https://stacks.math.columbia. edu/tag/00NX}{Tag 00NX}]{stacks-project}).
		
		The representability follows from the fact that $\detaleproj{\hcal{S}}{R}$ is a full subcategory of $\dmod{\hcal{S}}{R}$.
	\end{enumerate}
	\end{proof}

Now that the tensor product is constructed, we can move on to constructing the internal $\hrm{Hom}$. While the tensor product was already defined on $\dmod{\hcal{S}}{R}$, the internal $\hrm{Hom}$ only exists at the level of étale projective modules, occasionally for étale modules.

\begin{lemma}\label{hom_fin_proj}
	Let $A$ be a ring, let $M_1$ and $M_2$ be two $A$-modules. Let $f \, : \, A\rightarrow B$
	 be a ring morphism.
	
	\begin{enumerate}[itemsep=0mm]
		\item  There exists a morphism of $B$-modules natural in both $M_1$ and $M_2$
	\begin{align*}
		\iota_{M_1,M_2,f} \, : \, B\, \mathop{\otimes}_A \hrm{Hom}_A(M_1,M_2) &\rightarrow \hrm{Hom}_A\big(M_1,B\, \otimes_A M_2) \\
		b\otimes f &\mapsto \big[m_1 \mapsto b\otimes f(m_1)\big]
	\end{align*} Moreover, the target is naturally isomorphic $\hrm{Hom}_B(B\otimes_A M_1, B \otimes_A M_2)$ and $\iota_{M_1,M_2,f}$ can be rewritten $(b\otimes f) \mapsto (b\,\hrm{Id}_B\otimes f)$.
	
	\item If $M_1$ is finite projective, the previous morphism is an isomorphism.
	
	\item If $M_1$ is of finite presentation and $f$ is flat, the previous morphism is an isomorphism.
	
	\item If $M_1$ and $M_2$ are finite projective, then so is $\hrm{Hom}_A(M_1,M_2)$.
	\end{enumerate}
\end{lemma}
\begin{proof}
	\begin{enumerate}[itemsep=0mm]
		\item Left to the reader.
		
		\item Proof can be found in \cite[\href{https://stacks.math.columbia.edu/tag/0DVB}{Tag 0DBV}]{stacks-project}.
		
		\item Take a presentation of $M_1$ by an exact sequence $A^p \rightarrow A^d \rightarrow M_1 \rightarrow 0$. By left exactness of $\hrm{Hom}_A(-, M_2)$, left exactness of $\hrm{Hom}_A(-, B\otimes_A M_2)$ and flatness of $f$ we get a commutative diagram for which the second point and of this proposition and the five lemma concludes.
	
		\item The second point applied to localisations proves that the $\hcal{O}_{\hrm{Spec}(A)}$-modules  $\underline{\hrm{Hom}}_{\hcal{O}_{\hrm{Spec}(A)}}(\widetilde{M_1}, \widetilde{M_2})$ and $\widetilde{\hrm{Hom}_A(M_1,M_2)}$ are isomorphic. Since the modules $M_1$ and $M_2$ are locally finite free, we deduce through $\hcal{O}_{\hrm{Spec}(A)}$-modules that $\hrm{Hom}_R(M_1,M_2)$ is locally finite free.
		\end{enumerate}
\end{proof}

\begin{coro}\label{iso-special-R}
	Let $D_1,D_2$ be two objects of $\detaleproj{\hcal{S}}{R}$, then we have an isomorphism of $R$-modules:
	
	$$\forall s \in \hcal{S},\,\,\,\iota_{D_1,D_2,\varphi_s} \, : \, \varphi_s^* \hrm{Hom}_R(D_1, D_2) \rightarrow \hrm{Hom}_R(\varphi_s^*D_1, \varphi_s^* D_2), \, \, \,\, 1 \otimes f \mapsto \hrm{Id_R} \otimes f.$$ 
	
	The same result holds if each $\varphi_s$ is flat, for $D_1$ in $\detale{\hcal{S}}{R}$ and $D_2$ in $\dmod{\hcal{S}}{R}$.
\end{coro}

We are ready to properly define the internal $\hrm{Hom}$.

\begin{defiprop}\label{hom_modules_étales}
	Let $D_1$ belong to $\detale{\hcal{S}}{R}$ and $D_2$ to $\dmod{\hcal{S}}{R}$. 
	
	The $R$-module $\hrm{Hom}_R(D_1,D_2)$ endowed with the linearisations $$\varphi_{s,\underline{\hrm{Hom}}_R(D_1,D_2)}^* \, : \,\varphi_s^*(\hrm{Hom}_R(D_1,D_2))\xrightarrow[]{\iota_{D_1,D_2, \varphi_s}} \hrm{Hom}_R(\varphi_s^*D_1,\varphi_s^*D_2) \xrightarrow[]{\varphi_{s,D_2}^* \circ \,\, - \,\, \circ (\varphi_{s,D_1}^*)^{-1}} \hrm{Hom}_R(D_1,D_2)$$ is an $\hcal{S}$-module over $R$. We call it the \textit{internal $\hrm{Hom}$} and write it $\underline{\hrm{Hom}}_R(D_1,D_2)$.
	
	If both $D_1$ and $D_2$ belong $\detaleproj{\hcal{S}}{R}$, then so do $\underline{\hrm{Hom}}_R(D_1,D_2)$.
		
	If the ring $R$ is noetherian and each $\varphi_s$ is flat, this holds in $\detale{\hcal{S}}{R}$.
\end{defiprop}
\begin{proof}
For the correct definition, it only remains to prove for all $(s,s')\in \hcal{S}^2$ that: $$\varphi_{ss',\underline{\hrm{Hom}}_R(D_1,D_2)}=\varphi_{s,\underline{\hrm{Hom}}_R(D_1,D_2)}\circ \varphi_{s',\underline{\hrm{Hom}}_R(D_1,D_2)}. $$ For $f\in \hrm{Hom}_R(D_1,D_2)$, we compute $\varphi_{s,\underline{\hrm{Hom}}_R(D_1,D_2)}(f)$. By delinearisation and description of $\iota_{D_1,D_2,\varphi_s}$, it is equal to $\varphi_{s,D_2}^*\circ (\hrm{Id}\otimes f)\circ (\varphi_{s,D_1}^*)^{-1}$. Explicitly we get
	\begin{align*}
		\varphi_{s,\underline{\hrm{Hom}}_R(D_1,D_2)}(f)\bigg(\sum r_i\, \varphi_{s,D_1}(d_i)\bigg) &=\big(\varphi_{s,D_2}^*\circ (\hrm{Id}\otimes f)\big)\bigg(\sum r_i \otimes d_i\bigg) \\
		&=\varphi_{s,D_2}^*\bigg(\sum r_i \otimes f(d_i)\bigg) \\
		&= \sum r_i \,\varphi_{s,D_2}\big(f(d_i)\big)
	\end{align*} Any element of $D_1$ can be written as $\sum r_i \varphi_{ss',D_1}(d_i)$ thanks to the étaleness of $D_1$. Using the above equality, we obtain that $\varphi_{ss',\underline{\hrm{Hom}}_R(D_1, D_2)}(f)$ and $\left[\varphi_{s,\underline{\hrm{Hom}}_R(D_1,D_2)}\circ \varphi_{s',\underline{\hrm{Hom}}_R(D_1,D_2)}(f)\right]$ coincide on such expressions, hence on $D_1$.

Corollary \ref{iso-special-R} for étale projective modules shows that $\iota_{D_1,D_2,\varphi_s}$ is an isomorphism, implying that the linearisations $\varphi_{s,\underline{\hrm{Hom}}_R(D_1,D_2)}^*$ also are. The fourht point of Lemma \ref{hom_fin_proj} shows that $\hrm{Hom}_R(D_1,D_2)$ is finite projective. 
	
	Consider the second case. The étale case of Corollary \ref{iso-special-R} proves again that the $\varphi^*_{s,\underline{\hrm{Hom}}_R(D_1,D_2)}$ are isomorphisms. Moreover, if we take an epimorphism $R^k\twoheadrightarrow D_1$, the deduced map $$\hrm{Hom}_R(D_1,D_2)\rightarrow D_2^k$$ is injective. Then, the noetherianity of $R$ implies that $\hrm{Hom}_R(D_1,D_2)$ is of finite presentation.
\end{proof}

\begin{rem}
	First note that we have crucially used the étaleness of $D_1$ to define the structural endomorphisms.
\end{rem}

We can express our construction and its properties in a more appropriate language.

\begin{prop}\label{monoidalstructure}
	\begin{enumerate}[itemsep=0mm]
		\item Consider the bifunctor $\,-\otimes_R -\,$ on $\dmod{\hcal{S}}{R}$, the object $R$, coherence and swap maps coming from the tensor product on $R\text{-}\hrm{Mod}$. They endow $\dmod{\hcal{S}}{R}$ with the structure of a symmetric monoidal category. The same holds for the full subcategory $\detale{\hcal{S}}{R}$.
		
		\item  The full subcategory $\detaleproj{\hcal{S}}{R}$ with the same structure is closed symmetric monoidal. The right adjoint to $-\otimes_R D$ is  $\underline{\hrm{Hom}}_R(D,-)$.
		
		\item The previous point holds for the full subcategory $\detale{\hcal{S}}{R}$ if $R$ is noetherian and each $\varphi_s$ is flat.
	\end{enumerate}
\end{prop}
\begin{proof}
	Tensor product is a symmetric monoidal structure on $R\text{-}\hrm{Mod}$. Most of the proposition is obtained from previous results and checking three facts. First, associators, unitors, and swap maps for the underlying modules of objects in $\dmod{\hcal{S}}{R}$ are actually maps in $\dmod{\hcal{S}}{R}$. Second, the bifunctoriality morphisms for internal $\hrm{Hom}$ on $R\text{-}\hrm{Mod}$ produces morphisms in $\dmod{\hcal{S}}{R}$ when applied to morphisms in $\dmod{\hcal{S}}{R}$. Finally, the adjunction bijection in $R\text{-}\hrm{Mod}$ restricts-corestricts to morphisms in $\dmod{\hcal{S}}{R}$.
\end{proof}

We conclude this study of internal $\hrm{Hom}$ by recovering morphisms in $\dmod{\hcal{S}}{R}$.

\begin{prop}\label{inv_hom}
	Let $D_1$ belongs $\detale{\hcal{S}}{R}$ and $D_2$ to $\dmod{\hcal{S}}{R}$. The $\hcal{S}$-module $\underline{\hrm{Hom}}_R(D_1,D_2)$ verifies that $$\bigcap_{s\in \hcal{S}} \underline{\hrm{Hom}}_R(D_1,D_2)^{\varphi_s=\hrm{Id}}=\hrm{Hom}_{\dmod{\hcal{S}}{R}}(D_1,D_2).$$
\end{prop}
\begin{proof}
	Let $s\in \hcal{S}$. We use the expression of $\varphi_{s,\underline{\hrm{Hom}}_R(D_1,D_2)}(f)$ that we got while proving Proposition \ref{hom_modules_étales}. Applied to $\varphi_{s,D_1}(d)$, it proves that if $\varphi_{s,\underline{\hrm{Hom}}_R(D_1,D_2)}(f)=f$ then $f\circ \varphi_{s,D_1}=\varphi_{s,D_2}\circ f$. Conversely, if the second equality holds, for every $d=\sum r_i \varphi_{s,D_1}(d_i)$, we have $$	\varphi_{s,\underline{\hrm{Hom}}_R(D_1,D_2)}(f)(d) = \sum r_i\, \varphi_{s,D_2}(f(d_i)) = \sum r_i\, f(\varphi_{s,D_1}(d_i)) = f(d).$$
\end{proof}

	\vspace{1 cm}
	
	\section{Operations on $\hcal{S}$-modules over $R$}\label{section_operations}
	
	A Fontaine-type functor is decomposed as an extension of scalars followed by a taking of invariants. This section introduces such operations on $\hcal{S}$-modules over $R$ and its full subcategories.
	
	\hspace{1 cm}
	
	\subsection{Extension of scalars}\label{ex_subsection}
	
		First, we study the change of base ring.
		
		\begin{defi}
		The category $\hcal{S}\text{-}\hrm{Ring}$ has for objects the $\hcal{S}$-rings introduced in Definition \ref{def_sring} and for morphisms the $\hcal{S}$-equivariant ring morphisms $a\, : \, R\rightarrow T$.
	\end{defi}

When considering two $\hcal{S}$-rings $R$ and $T$, we note $s \mapsto \varphi_s$ the structural monoid morphism for $R$ and $s\mapsto \varphi'_s$ the structural monoid morphism for $T$ to avoid ambiguity.

\begin{defiprop}
	Let $a\, : \, R\rightarrow T$ be a morphism of $\hcal{S}$-rings. Let $D$ belongs $\dmod{\hcal{S}}{R}$. We define
	
	$$\forall s \in \hcal{S}, \,\,\,\varphi'^*_{s,\hrm{Ex}(D)} \, : \,\varphi'^*_s\left(T\, \mathop{\otimes}_R D\right) \xrightarrow[]{\sim} T\, \mathop{\otimes}_R (\varphi_s^*D) \xrightarrow[]{\hrm{Id}_T \otimes \varphi^*_{s,D}} T\, \mathop{\otimes}_R D.$$ With these data, the module  $\left(T\otimes_R D\right)$ belongs to $\dmod{\hcal{S}}{T}$.
	
	 We define the functor
	 $$\hrm{Ex} \, : \, \dmod{\hcal{S}}{R}\rightarrow \dmod{\hcal{S}}{T}, \,\,\, D \mapsto \left(T\otimes_R D\right), \,\, f\mapsto \hrm{Id}_T \otimes f.$$
\end{defiprop}
\begin{proof}
We let the reader check that our construction is correct. It is nice to remark that the delinearisations verify

$$\varphi'_{s,\hrm{Ex}(D)}(t\otimes d)=\varphi'_s(t)\otimes \varphi_{s,D}(d).$$
\end{proof}

\begin{prop}\label{constr_ext_etale}
	For any $\hcal{S}$-rings morphism $a\, : \, R\rightarrow T$, the functor $\hrm{Ex}$ has the following interactions with the previous section.
	
	\begin{enumerate}[itemsep=0mm]
		\item The functor $\hrm{Ex}$ restricts-corestricts to étale (resp. étale projective) $\hcal{S}$-modules as follows
		
			$$\hrm{Ex} \, : \, \detale{\hcal{S}}{R} \rightarrow \detale{\hcal{S}}{T}$$
		
		and 
		
		$$\hrm{Ex} \, : \, \detaleproj{\hcal{S}}{R} \rightarrow \detaleproj{\hcal{S}}{T}.$$
		
		\item The functor $\hrm{Ex}$ is strong symmetric monoidal. Thus, its restrictions to étale (resp. étale projective) modules is too.
		
		\item For all objects $D_1$ of $\detale{\hcal{S}}{R}$ and $D_2$ of $\dmod{\hcal{S}}{R}$, there is a morphism in $\dmod{\hcal{S}}{T}$
		
		$$\hrm{Ex}(\underline{\hrm{Hom}}_R(D_1,D_2)) \rightarrow \underline{\hrm{Hom}}_T(\hrm{Ex}(D_1),\hrm{Ex}(D_2)),$$ given by the setup of a lax monoidal functor between two closed monoidal categories. The underlying $T$-modules morphism coincides with $\iota_{D_1,D_2,a}$ from Lemma \ref{hom_fin_proj}.
		
		If $D_1$ belongs to $\detaleproj{\hcal{S}}{R}$, the morphism is an isomorphism\footnote{In other terms, $\hrm{Ex}$ is closed monoidal on $\detaleproj{\hcal{S}}{R}$.}.
		
		\item It is also an isomorphism if $a$ is flat.
	\end{enumerate}
\end{prop}

\begin{proof}
	\begin{enumerate}[itemsep=0mm]
		\item Base change preserves both the finite presentation property, the finite projectiveness and the constant rank. Moreover, the definition of $\varphi'^*_{s,\hrm{Ex}(D)}$ makes it an isomorphism as soon as $\varphi^*_{s,D}$ is one.
		
		\item The base change on categories of modules is strong symmetric monoidal. We only need to check that coherence maps for the underlying modules live in $\dmod{\hcal{S}}{R}$.
		
		\item Recall how we get this morphism. By adjonction, we have a morphism in $\dmod{\hcal{S}}{R}$
		
		$$h\, : \, \underline{\hrm{Hom}}_R(D_1,D_2)\mathop{\otimes}_R D_1\rightarrow D_2.$$ The composition
		$$ \hrm{Ex}\left(\underline{\hrm{Hom}}_R(D_1,D_2)\right) \, \mathop{\otimes}_R \hrm{Ex}(D_1) \rightarrow \hrm{Ex}\left(\underline{\hrm{Hom}}_R(D_1,D_2)\, \mathop{\otimes}_R D_1\right) \xrightarrow[]{\hrm{Ex}(h)} \hrm{Ex}(D_2),$$ whose first term is the coherence map for $\hrm{Ex}$, gives us the desired morphism by adjunction. In our context, we know that $h$ is set-theoretically given by $f\otimes d_1 \mapsto f(d_1)$ so we can compute to identify the obtained morphism with $\iota_{D_1,D_2,a}$. 
	
		If both modules are projective, the second point of Lemma \ref{hom_fin_proj} tells us that $\iota_{D_1,D_2,a}$ is an isomorphism.
	
	\item Identical using rather the third point of Lemma \ref{hom_fin_proj}.
	\end{enumerate}
\end{proof}

\begin{rem}\label{rem_oubli_dmod}
	The forgetful functor from $\dmod{\hcal{S}}{T}$ to $\dmod{\hcal{S}}{R}$ is a right adjoint to $\hrm{Ex}$. In fact, let $D$ belongs to $\dmod{\hcal{S}}{R}$ and $\Delta$ to $\dmod{\hcal{S}}{T}$. The adjunction bijection on underlying modules $$\hrm{Hom}_R(D,\Delta) \xrightarrow[]{\sim} \hrm{Hom}_T\left(T\, \mathop{\otimes}_R D,\Delta\right), \,\,\, f\mapsto [t\otimes d \mapsto t\,f(d)]$$ restricts-corestricts to the maps in $\dmod{\hcal{S}}{R}$ and $\dmod{\hcal{S}}{T}$. 
	
	For $\Delta=\hrm{Ex}(D)$, the identity gives rise to a morphism $$D\rightarrow \hrm{Ex}(D), \,\, d\mapsto 1\otimes d$$ in $\dmod{\hcal{S}}{R}$. The definition of $\varphi'_{s,\hrm{Ex}}$ allows to take invariants and obtain a map $$\bigcap_{s\in \hcal{S}} D^{\varphi_s=\hrm{Id}} \rightarrow \bigcap_{s\in\hcal{S}} \hrm{Ex}(D)^{\varphi'_s=\hrm{Id}}.$$
\end{rem}

\begin{prop}\label{constr_ext_etale_hom}
	Let $D_1$ be an object of $\detale{\hcal{S}}{R}$ and $D_2$ of $\dmod{\hcal{S}}{R}$. The previous remark applied to internal $\hrm{Hom}$ and the third point of Proposition \ref{constr_ext_etale} gives a morphism in $\dmod{\hcal{S}}{R}$
	
	$$\underline{\hrm{Hom}}_R(D_1,D_2)\rightarrow \hrm{Ex}(\underline{\hrm{Hom}}_R(D_1,D_2)) \xrightarrow[]{\iota_{D_1,D_2,a}} \underline{\hrm{Hom}}_T(\hrm{Ex}(D_1),\hrm{Ex}(D_2)).$$ After taking invariants, Proposition \ref{inv_hom} identifies it with a map
	
	$$\hrm{Hom}_{\dmod{\hcal{S}}{R}}(D_1,D_2) \rightarrow \hrm{Hom}_{\dmod{\hcal{S}}{T}}(\hrm{Ex}(D_1),\hrm{Ex}(D_2)).$$
	
	\begin{enumerate}[itemsep=0mm]
		\item This application is the one given by functoriality of $\hrm{Ex}$.
		
		\item Suppose that $a$ is injective. Then, the functor $\hrm{Ex}$ is faithful from $\detaleproj{\hcal{S}}{R}$.

	\end{enumerate}
\end{prop}
\begin{proof}
	\begin{enumerate}[itemsep=0mm]
		\item Follow the image of $f$. It goes to $1\otimes f$, then to $\hrm{Id}_T \otimes f$ which is exactly $\hrm{Ex}(f)$.
		
		\item From the identification of $\hrm{Ex}$ on morphisms, as soon as $D_1$ and $D_2$ are étale projective, the second point of Lemma \ref{hom_fin_proj} implies that $\iota_{D_1,D_2,a}$ is an isomorphism. We already proved that $\underline{\hrm{Hom}}_R(D_1,D_2)$ is again finite projective, hence flat. Together with injectivity of $a$, it proves that the composition's first part is injective. The composition itself is injective and remains injective after taking invariants. It is precisely the faithfulness condition.
	\end{enumerate}
\end{proof}

\hspace{1 cm}
	
	\subsection{Invariants by a normal submonoid}\label{subsection_inv}

In this section, we move towards the second step of Fontaine-type functors: taking invariants. We stick with an $\hcal{S}$-ring $R$ and add the datum of a normal submonoid\footnote{We say that $\hcal{S}'$ is normal in $\hcal{S}$, noted $\hcal{S}'\triangleleft \hcal{S}$, if $\forall s\in \hcal{S}, \,\,\, s\hcal{S}'=\hcal{S}'s$. This condition allows to define a structure of monoid on the set of left cosets and the obtained quotient satisfy the same universal property as group quotients. Beware that $\hcal{S}$ might not be normal in itself, e.g. $\hrm{M}_2(\R)$, or more generally that the kernel of a monoid map might not be normal. Hence, the normal submonoids are not the only cases where the universal property of quotients is easily understood.} $\hcal{S}'$ of $\hcal{S}$. The subring $R^{\hcal{S}'}$ is endowed with a structure of $\hcal{S}/\hcal{S}'$-ring via the restriction-corestriction of each $\varphi_s$. The inclusion $R^{\hcal{S}'}\subseteq R$ is a morphism of $\hcal{S}$-rings.

\begin{defiprop}
	Let $D$ be an object of $\dmod{\hcal{S}}{R}$. Each $\varphi_{s,D}$ restricts-corestricts to $D^{\hcal{S}'}$ and these restrictions endow it with the structure of object in $\dmod{\sfrac{\hcal{S}}{\hcal{S}'}}{R^{\hcal{S}'}}$.

	The functor $\hrm{Inv} \, : \, \dmod{\hcal{S}}{R} \rightarrow \dmod{\sfrac{\hcal{S}}{\hcal{S}'}}{R^{\hcal{S}'}}$ is defined by $D\mapsto D^{\hcal{S}'}$ and by restriction-corestriction of the maps.
	
	In the same setup, we call \textit{comparison morphism for $D$} the map $$R\otimes_{R^{\hcal{S}'}} \hrm{Inv}(D) \rightarrow D,$$ where the first map is the base change of the inclusion, and the second one is $r\otimes d \mapsto rd$.	
\end{defiprop}

\begin{prop}\label{inv_detaleproj}
	Suppose that $R^{\hcal{S}'} \subseteq R$ is faithfully flat and that the comparison morphism for $D$ $$R\, \mathop{\otimes}_{R^{\hcal{S}'}} \hrm{Inv}(D) \rightarrow D$$ is an isomorphism.
	
	If $D$ belongs to $\detale{\hcal{S}}{R}$ (resp. to $\detaleproj{\hcal{S}}{R}$), then $\hrm{Inv}(D)$ belongs to $\detale{\sfrac{\hcal{S}}{\hcal{S}'}}{R^{\hcal{S}'}}$ (resp. to $\detaleproj{\sfrac{\hcal{S}}{\hcal{S}'}}{R^{\hcal{S}'}}$).
\end{prop}
\begin{proof}
	According to \cite[\href{https://stacks.math.columbia.edu/tag/03C4}{Tag 03C4}]{stacks-project} the fact that $D^{\hcal{S}'}$ is of finite presentation (resp. finite projective) can be checked after base change to $R$, i.e. on $D$ thanks to the comparison isomorphism. The étaleness condition can also be checked after base change to $R$. Because $R^{\hcal{S}'} \subseteq R$ is $\hcal{S}$-equivariant, the base change of $\varphi_{s\hcal{S}',D^{\hcal{S}'}}$ is identified to $\varphi_{s,D}$ via the comparison morphism.
\end{proof}

\begin{coro}\label{inv_monoidal}
	Suppose that $R^{\hcal{S}'}\subseteq R$ is faithfully flat and that the comparison morphism is an isomorphism for every (resp. étale, resp. étale projective) $\hcal{S}$-module over $R$. Then $\hrm{Inv}$ is a strong symmetric monoidal functor from $\dmod{\hcal{S}}{R}$ (resp. $\detale{\hcal{S}}{R}$, resp. $\detaleproj{\hcal{S}}{R}$).
\end{coro}
\begin{proof}
	Let $D_1,D_2$ be $\hcal{S}$-modules over $R$ (resp. and that they are étale, resp. and that they are étale projective). There is a natural morphism of $\hcal{S}/\hcal{S}'$-modules over $R^{\hcal{S}'}$:
	
	$$\hrm{Inv}(D_1)\otimes_{R^{\hcal{S}'}} \hrm{Inv}(D_2) \rightarrow \hrm{Inv}\left(D_1\otimes_R D_2\right), \,\,\, d_1\otimes d_2 \mapsto (d_1 \otimes d_2).$$ We check that it is an isomorphism after base change to $R$. The module $D_1\otimes_R D_2$ is still an (resp. étale, resp. étale projective) $\hcal{S}$-module over $R$ so its comparison morphism is an isomorphism. Moreover, naturally $$R\otimes_{R^{\hcal{S}'}} \left(\hrm{Inv}(D_1)\otimes_{R^{\hcal{S}'}} \hrm{Inv}(D_2)\right) \cong \left( R\otimes_{R^{\hcal{S}'}} \hrm{Inv}(D_1)\right) \otimes_R  \left( R\otimes_{R^{\hcal{S}'}} \hrm{Inv}(D_1)\right).$$ We conclude by using the comparison morphisms for $D_1$ and $D_2$.
\end{proof}

\begin{prop}\label{inv_closed}
	Suppose that $R^{\hcal{S}'}\subseteq R$ is faithfully flat. For every étale $D_1$ and $D_2$ in $\dmod{\hcal{S}}{R}$ for which comparison morphisms are isomorphisms, there is a natural isomorphism of $\hcal{S}/\hcal{S}'$-modules over $R^{\hcal{S}'}$ whose source and target are correctly defined $$\underline{\hrm{Hom}}_{R^{\hcal{S}'}}(\hrm{Inv}(D_1),\hrm{Inv}(D_2)) \rightarrow \hrm{Inv}\left(\underline{\hrm{Hom}}_R(D_1,D_2)\right).$$
	
	In particular, if the comparison morphism is an isomorphism for every object in $\detaleproj{\hcal{S}}{R}$, then the functor $\hrm{Inv}$ from this category is closed monoidal.
\end{prop}
\begin{proof}
	Thanks to Propositions \ref{hom_modules_étales} and \ref{inv_detaleproj}, our objects are well defined. We also leave to the reader to check that the scalar extension of the morphisms gives the predicted morphism.
	
	Showing that it is an isomorphism can be checked after base change to $R$. Because the $\hrm{Inv}(D_i)$ are étale projective, the fourth point of \ref{constr_ext_etale} applied to $R^{\hcal{S}'}\subset R$ identifies $$R\otimes_{R^{\hcal{S}'}} \underline{\hrm{Hom}}_{R^{\hcal{S}'}}(\hrm{Inv}(D_1),\hrm{Inv}(D_2)) \,\,\,\text{ to } \,\, \,\underline{\hrm{Hom}}_R\left(R\otimes_{R^{\hcal{S}'}} \hrm{Inv}(D_1),R\otimes_{R^{\hcal{S}'}} \hrm{Inv}(D_1)\right)$$ and then to $\underline{\hrm{Hom}}_R(D_1,D_1)$ by the comparison isomorphisms.
\end{proof}

I don't see how to weaken the hypotheses of these Propositions. The invariant step of a Fontaine equivalence behaves like descent (when it isn't exactly descent). Hence, studying these comparison morphisms is the difficult part\footnote{Once the rings in play are constructed, which can be tricky.} (see \cite{zabradi_equiv} or \cite{zabradi_kedlaya_carter} for the multivariable case). Even in the basic case, we recall in section \ref{section_fontaine} that it relies on either Galois descent, or on counting points on an étale variety over an algebraically closed field, both of which are rather deep results. Requiring this condition as a black box seems quite reasonable for a general setting.

\vspace{0.5cm}
We finish by a proposition that will be useful in section \ref{section_dévissage}.

\begin{prop}\label{inv_detale_et_proj}
	Let $D$ be a finite projective module over $R^{\hcal{S}'}$. Suppose that  $$\forall r\in R^{\hcal{S}'}, \, \exists n\geq 1, \,\,\, R[r^{\infty}]=R[r^n]$$ Then, the map $$c\, : \, D \rightarrow \hrm{Inv}\left(R\, \mathop{\otimes}_{R^{\hcal{S}'}} D \right)$$ is an isomorphism of $R^{\hcal{S}'}$-modules. If $D$ was an $\hcal{S}/\hcal{S}'$-module over $R^{\hcal{S}'}$, then $c$ is an isomorphism of $\hcal{S}/\hcal{S}'$-modules over $R^{\hcal{S}'}$.
\end{prop}
\begin{proof}
 If $D$ is free, the isomorphism $$R\otimes_{R^{\hcal{S}'}} D \cong R^d$$ in $\dmod{\hcal{S}'}{R}$ concludes. In general, the module $D$ is locally free. Choose $r\in R^{\hcal{S}'}$ such that $D[r^{-1}]$ is free over $R^{\hcal{S}'}[r^{-1}]$. The $\hcal{S}'$-ring structure on $R$ extend to an $\hcal{S}'$-ring structure on $R[r^{-1}]$. According to the free case, the map $$D[r^{-1}]\rightarrow \hrm{Inv}\left(R[r^{-1}] \, \mathop{\otimes}_{R^{\hcal{S}'}[r^{-1}]} D[r^{-1}] \right)$$ is an isomorphism. Using that $R$ has bounded $r^{\infty}$-torsion and that the $R^{\hcal{S}'}$-module $D$ is flat, we identify the target of this isomorphism with $\hrm{Inv}\left(R\, \mathop{\otimes}_{R^{\hcal{S}'}} D\right)[r^{-1}]$ and the isomorphism itself with $c[r^{-1}]$. The map $c$ is an isomorphism locally on $\hrm{Spec}(R^{\hcal{S}'})$, hence an isomorphim.
	
We let the reader check the equivariance when we have an $\hcal{S}/\hcal{S}'$-action.
\end{proof}

\hspace{1 cm}

	\subsection{Coinduction to a bigger monoid}

	The previous subsection showed how to deal with taking of invariants and quotienting the considered monoid. This subsection considers the adjoint construction: inflating the monoid. It is a construction I used in \cite{nataniel_fontaine}. Coinduction for monoids mimics its analogue for groups. In the setting of groups, the reader might refer to \cite[\S 6.1 and \S6.3]{weibel_1994}.

Let $\hcal{S}$ be a submonoid of a bigger monoid $\hcal{T}$. The forgetful functor from $\hcal{T}\text{-}\hrm{Set}$ to $\hcal{S}\text{-}\hrm{Set}$ has a right adjoint: the coinduction denoted by $\hrm{Coind}_{\hcal{S}}^{\hcal{T}}$. Explicitely, $$\forall X\in \hcal{S}\text{-}\hrm{Set}, \,\,\,\coindu{\hcal{S}}{\hcal{T}}{X}:=\big\{f\, : \, \hcal{T} \rightarrow X \, \,\big|\,\,\forall s \in \hcal{S}, t\in \hcal{T}, \,\, f(st)=s\cdot f(t)\big\}$$ with $\hcal{T}$-action given by $t \cdot f=[s\mapsto f(st)]$. The coinduction commutes to limits. Moreover, the evaluation at the identity element of $\hcal{T}$ induces a bijection

 \begin{equation}\label{eqn:inv}
	\left[\coindu{\hcal{S}}{\hcal{T}}{X}\right]^{\hcal{T}} \xrightarrow[]{\sim} X^{\hcal{S}} \tag{*}
\end{equation}

\begin{prop}\label{induction_anneau}
	1) For any monoid $\hcal{S}$,  the category of rings objects in $\hcal{S}\text{-}\hrm{Set}$ is equivalent to $\hcal{S}\text{-}\hrm{Ring}$.
	
	The coinduction induces a functor from $\hcal{S}\text{-}\hrm{Ring}$ to $\hcal{T}\text{-}\hrm{Ring}$, for which the ring structure on $\coindu{\hcal{S}}{\hcal{T}}{R}$ is the termwise structure on functions from $\hcal{T}$ to $R$.
	
	2) For every $\hcal{S}$-ring $R$, the category of $R$-modules objects\footnote{See for instance \cite[Definition 1.3]{nataniel_nonsense_cat}.} in $\hcal{S}\text{-}\hrm{Set}$ is equivalent to $\dmod{\hcal{S}}{R}$.
	
	The coinduction induces a functor from $\dmod{\hcal{S}}{R}$ to $\dmod{\hcal{T}}{\coindu{\hcal{S}}{\hcal{T}}{R}}$: if the map $\left[\lambda \, : \, R\times D \rightarrow D\right]$ is the external multiplication on $D$, the multiplication on $\coindu{\hcal{S}'}{\hcal{S}}{D}$ is given by
	
	$$\coindu{\hcal{S}}{\hcal{T}}{R} \times \coindu{\hcal{S}}{\hcal{T}}{D} \xrightarrow[]{\sim} \coindu{\hcal{S}}{\hcal{T}}{R\times D} \xrightarrow[]{\coindu{\hcal{S}}{\hcal{T}}{\lambda}} \coindu{\hcal{S}}{\hcal{T}}{D}.$$
\end{prop}
\begin{proof}
	1) A ring object in $\hcal{S}\text{-}\hrm{Set}$ corresponds to an addition map, a neutral element, an opposite map and a multiplication map which are $\hcal{S}$-equivariant and make the suitable diagrams commute. They give a ring structure, and $\hcal{S}$ acts by ring endomorphisms thanks to the equivariance of the diagrams.
	
	Because coinduction naturally commutes with fiber products, the previous paragraph and \cite[Lemma 1.4]{nataniel_nonsense_cat} provide the desired promotion. The same lemma gives the description of the ring structure; for instance, the multiplication is given by $$\coindu{\hcal{S}}{\hcal{T}}{R}\times \coindu{\hcal{S}}{\hcal{T}}{R} \xrightarrow[]{\sim} \coindu{\hcal{S}}{\hcal{T}}{R\times R} \xrightarrow[]{\coindu{\hcal{S}}{\hcal{T}}{\mu_R}} \coindu{\hcal{S}}{\hcal{T}}{R}$$ and we check that it corresponds to the pointwise multiplication on $R^{\hcal{T}}$.
	
	2) Similar.
\end{proof}

\begin{rem} The upgraded functor $\hrm{Coind}_{\hcal{S}}^{\hcal{T}}$ from $\dmod{\hcal{S}}{\Z}$ to $\dmod{\hcal{T}}{\Z}$ is right adjoint to the forgetful functor. The identity (\ref{eqn:inv}) and the fact that the coinduction is right adjoint to a left exact functor provides natural isomorphism in $D(\hrm{Ab})$ between monoid cohomology $$\forall M \in \dmod{\hcal{S}}{\Z}, \,\,\,\hrm{R}\Gamma\left(\hcal{T}, \hrm{RCoind}_{\hcal{S}}^{\hcal{T}}(M)\right) \,\cong \, \hrm{R}\Gamma(\hcal{S},M).$$
\end{rem}

\medskip

After this rough setup, we give a certain number of useful results for groups then try to adapt them for monoids.

\begin{lemma}\label{induction_anneau_tensor}
	Let $\hcal{H}$ be subgroup of a group $\hcal{G}$.
	\begin{enumerate}[itemsep=0mm]
	\item For any $\hcal{G}$-ring $R$, the map $$R \rightarrow \coindu{\hcal{H}}{\hcal{G}}{R}, \,\,\, r\mapsto [g\mapsto \varphi_g(r)]$$ is a morphism of $\hcal{G}$-rings.
	
	\item Suppose that $\hcal{H}$ is of finite index in $\hcal{G}$. Let $R$ be a $\hcal{G}$-ring, $T$ an $\hcal{H}$-ring and $i\, : \, R\rightarrow T$ a morphism of $\hcal{H}$-rings. The first point produces a morphism of $\hcal{G}$-rings 
	$$j\, : \, R\rightarrow \coindu{\hcal{H}}{\hcal{G}}{R} \xrightarrow[]{\coindu{\hcal{H}}{\hcal{G}}{i}} \coindu{\hcal{H}}{\hcal{G}}{T}.$$ For every $D$ in $\dmod{\hcal{G}}{R}$, the following map is an isomorphism in $\dmod{\hcal{G}}{\coindu{\hcal{H}}{\hcal{G}}{T}}$:
	
	$$\coindu{\hcal{H}}{\hcal{G}}{T} \otimes_R D \rightarrow \bigcoindu{\hcal{H}}{\hcal{G}}{T\otimes_R D}, \,\,\, f\otimes d \mapsto [g \mapsto f(g) \otimes \varphi_{g,D}(d)].$$
	\end{enumerate}
\end{lemma}
\begin{proof}
	1) Left to the reader.
	
	2) We leave it to the reader to check that it is a well-defined additive $\hcal{G}$-equivariant $\coindu{\hcal{H}}{\hcal{G}}{T}$-linear morphism and move on to the second morphism. Fix a system of distinct representatives $\hcal{R}$ of $\hcal{H}\backslash\hcal{G}$. Left cosets form a partition of $\hcal{G}$ hence for an $\hcal{H}$-ring (resp. an $\hcal{H}$-abelian groups) $M$, the map $$\hrm{ev}_M \, : \,\coindu{\hcal{H}}{\hcal{G}}{M} \xrightarrow[f\mapsto (f(g))_{g\in\hcal{R}}]{} \prod_{\hcal{R}} M$$ is an isomorphism of $\hcal{G}$-rings (resp. of $\hcal{G}$-abelian groups). The action of $g'\in \hcal{G}$ on the right is defined by sending $(m_g)_{g\in \hcal{R}}$ to $(\varphi_{h_{g',g}}(m_{k_{g',g}}))_{g\in \hcal{R}}$ where $k_{g',g}$ is the representative of $g'^{-1}g$ and $g'^{-1}g=h_{g',g} k_{g',g}$.

 For our $\hcal{H}$-ring $T$, the $R$-algebra structure given by $j$ identify to $\prod_{g\in \hcal{R}} i\circ \varphi_g$. 
	
	Because $[\hcal{G} \, : \,\hcal{H}]<+\infty$, the product over $\hcal{R}$ is also a direct sum. It is possible to express the studied map for $D$ as the composition of the following bijections:

	\begin{center}
		\begin{tikzcd}
		\coindu{\hcal{H}}{\hcal{G}}{T} \otimes_R D \ar[rr,"\hrm{ev}_T \otimes_{\hrm{Id}_D}"] & & \left(\mathop{\bigoplus} \limits_{g\in \hcal{R}, \, R} T\right) \mathop{\bigotimes} \limits_{\mathop{\oplus} \limits_{g\in \hcal{R}} i\circ \varphi_g, \,\,R} D \ar[rrr,"\left(\oplus t_g\right) \otimes d  \,\mapsto \, \oplus\, (t_g\otimes d)"] & & & \mathop{\bigoplus} \limits_{g\in \hcal{R}} \left( T\mathop{\otimes} \limits_{i\circ \varphi_g, \,R} D\right) \ar[dd,"\oplus (t_g \otimes d_g) \, \mapsto \, \oplus (t_g \otimes (1 \otimes d_g))"'] \\ \\
		\bigcoindu{\hcal{H}}{\hcal{G}}{T\otimes_R D} & & \mathop{\bigoplus} \limits_{g\in \hcal{R}} \left( T\otimes_{R} D\right) \ar[ll,"\hrm{ev}_{(T\otimes_R D)}^{-1}"]  & & & \mathop{\bigoplus} \limits_{g\in \hcal{R}} \left( T\otimes_R \varphi_g^*D\right) \ar[lll,"\oplus \left(\hrm{Id}_T \otimes \varphi_{g,D}^*\right)"] 
		\end{tikzcd}
	\end{center} This concludes.
\end{proof}

\begin{prop}\label{coindu_etproj}
	Let $\hcal{H}<\hcal{G}$ with $[\hcal{G}:\hcal{H}]<+\infty$. Let $R$ be an $\hcal{H}$-ring. The functor $\hrm{Coind}_{\hcal{H}}^{\hcal{G}}$ from $\dmod{\hcal{H}}{R}$ to $\dmod{\hcal{G}}{\coindu{\hcal{H}}{\hcal{G}}{R}}$ satisfies the following properties:
	
	\begin{enumerate}[itemsep=0mm]
		\item  It is essentially surjective.
		
		\item It sends $\detale{\hcal{H}}{R}$ to $\detale{\hcal{G}}{\coindu{\hcal{H}}{\hcal{G}}{R}}$ and its restriction-corestriction is essentially surjective.
		
		\item It sends $\detaleproj{\hcal{H}}{R}$ to $\detaleproj{\hcal{G}}{\coindu{\hcal{H}}{\hcal{G}}{R}}$ and its restriction-corestriction is essentially surjective.
	\end{enumerate}
\end{prop}
\begin{proof}
	\begin{enumerate}[itemsep=0mm]
		\item Apply the second point of the Lemma \ref{induction_anneau_tensor} to the $\hcal{H}$-ring morphism $$i\, : \, \coindu{\hcal{H}}{\hcal{G}}{R} \rightarrow R, \,\,\, f\mapsto f(1_G).$$ The corresponding morphism $j$ is the identity on $\coindu{\hcal{H}}{\hcal{G}}{R}$ and thus exhibits $D$ in $\dmod{\hcal{G}}{\coindu{\hcal{H}}{\hcal{G}}{R}}$ as the coinduction of $R\otimes_{i,\coindu{\hcal{H}}{\hcal{G}}{R}} D$.
		
		\item \'Etalness is automatic for groups. In this setup with $[\hcal{G} \,: \, \hcal{H}]<\infty$, the coinduction coincide with the induction, left adjoint to the forgetful functor. Hence, coinduction is right exact and commutes to finite products. Applying coinduction on a finite presentation of $D$ produces a finite presentation of $\coindu{\hcal{H}}{\hcal{G}}{R}$.
		
		Moreover, any $D$ is the coinduction of $R\otimes_{i,\coindu{\hcal{H}}{\hcal{G}}{R}} D$, which is finitely presentated as soon as $D$ is.
		
		\item It remains to show that being locally finite free of constant rank is preserved by coinduction. Let $D$ be an étale projective module. Let $(r_i)_{i\in I}\in R^I$ be a finite family such that each $D[r_i^{-1}]$ is free and whose non-vanishing loci cover $\hrm{Spec}(R)$. For all $g\in \hcal{R}$ and $i\in I$, call $r_{i,g}$ the function in $\coindu{\hcal{H}}{\hcal{G}}{R}$ with support on $\hcal{H}g$ such that $r_{i,g}(g)=r_i$. The localisation of $\coindu{\hcal{H}}{\hcal{G}}{R}$ at $r_{i,g}$ is isomorphic to $R[r_i^{-1}]$ and $\coindu{\hcal{H}}{\hcal{G}}{D}[r_{i,g}^{-1}]$ to $D[r_i^{-1}]$. Moreover, the non-vanishing loci of the $r_{i,g}$ cover\footnote{The original covering hypothesis provides a family $(r'_i)$ such that $\sum_i r_i r_i'=1$. Thus, $\sum_{i,g} r_{i,g} r'_{i,g}=1$.} $\hrm{Spec}\left(\coindu{\hcal{H}}{\hcal{G}}{R}\right)$.
		
	Finally, any $D$ is the coinduction of $R\otimes_{i,\coindu{\hcal{H}}{\hcal{G}}{R}} D$, which is finite projective as soon as $D$ is.
	\end{enumerate}
\end{proof}

\begin{rem}
	We should be able to prove in this setup that the coinduction is strong symmetric monoidal and commutes with internal $\hrm{Hom}$ when it is defined. For monoids, finding the right setup seems really hard.
\end{rem}
	
\vspace{0.5 cm}

The finite index is crucial in Lemma \ref{induction_anneau_tensor} to allow the commutation of tensor product with the limit defining the coinduction. We used that the induction, which has dual properties (left adjoint to the forgetful functor, defined as a colimit, etc), coincides with coinduction. 

To deal with submonoids $\hcal{S}<\hcal{T}$, we also want tensor product and coinduction to commute. Imposing the finiteness of the index, defined as the cardinality of the set of the left cosets, is far too strong: the monoid $\N$ doesn't even have a finite index in itself, because the left cosets are contained in one another. As a first attempt, replace it by the existence of a finite family of left cosets which is cofinal for the inclusion\footnote{This is equivalent to imposing that there are finitely many maximal elements and that they are cofinal.}. Some deeper difficult occurs since the left cosets can be neither disjoint nor included in one another; for instance, if we take $\Delta$ to be the diagonal of $\N^2$, the elements $(0,1)$ and $(1,0)$ correspond to two distinct maximal cosets whose intersection contains $(1,2)$. The coinduction is not the product over the maximal left cosets but a subobject. Choose $\hcal{R}_{\min}$ a system of distinct representatives\footnote{Be careful that for $t\in \hcal{S}$ being a representative of $t_0\hrm{S}$ means that $t\hcal{S}=t_0\hcal{S}$ and not merely $t\in t_0\hcal{S}$.}. For the maximal left cosets and define $$\hcal{L}(\hcal{R}_{\min}):=\{(s_1,s_2,t_1,t_2) \in \hcal{S}^2 \times \hcal{R}_{\min}^2 \, |\, s_1 t_1=s_2 t_2\}.$$ The set $\hcal{L}(\hcal{R}_{\min})$ becomes a poset by fixing that $$ \forall s\in  \hcal{S}, \, \forall (s_1,s_2,t_1,t_2)\in \hcal{L}(\hcal{R}_{\min}), \,\,\,(s s_1,s s_2,t_1,t_2)\leq (s_1,s_2,t_1,t_2).$$ Then, $$\coindu{\hcal{S}}{\hcal{T}}{X}\cong\left\{ (x_t)\in \prod_{t\in \hcal{R}_{\min}} X \, \bigg|\, \forall (s'_1,s'_2,s_1,s_2) \in \hcal{L}(\hcal{R}_{\min}), \,\,\, \varphi_{s'_1}\left(x_{s_1}\right)=\varphi_{s'_2}\left(x_{s_2}\right)\right\}.$$ The condition $\varphi_{s'_1}\left(x_{s_1}\right)=\varphi_{s'_2}\left(x_{s_2}\right)$ can be restricted to a cofinal family in $\hcal{L}(\hcal{R}_{\min})$. Note that the poset $\hcal{L}(\hcal{R}_{\min})$ doesn't depend on the chosen representatives up to isomorphism; we call it $\hcal{L}$. This leads to the following definition.

\begin{defi}
	A submonoid $\hcal{S}$ is of \textit{finite subtle index} in another monoid $\hcal{T}$ if there are finitely many maximal left cosets for the inclusion, if these are cofinal, and if the maximal quadruples of $\hcal{L}$ are finitely many and cofinal.
\end{defi}

\begin{rem}
	For Fontaine equivalences, monoids appear because of Frobenii on imperfect rings. In \cite{zabradi_kedlaya_carter}, we encounter monoids like $(f\N)^d < \N^d$, and I even stumbled upon $(f\N)^d+\Delta<\N^d$ myself. We could first prove results for coinduction in a perfect setting then try to recover an imperfect version. However, this can be more technical than introducing this monoidal setting.
\end{rem}

\begin{lemma}\label{induction_anneau_tensor_2}
	Let $\hcal{S}<\hcal{T}$.
	
	\begin{enumerate}[itemsep=0mm]
		\item  Let $R$ be a $\hcal{T}$-ring. The map $$R \rightarrow \coindu{\hcal{S}}{\hcal{T}}{R}, \,\,\, r\mapsto [t\mapsto \varphi_t(r)]$$ is a morphism of $\hcal{T}$-rings.
		
		\item Suppose that $\hcal{S}$ is of finite subtle index inside $\hcal{T}$. Let $R$ be a $\hcal{T}$-ring, let $T$ be an $\hcal{S}$-ring and $i\,: \, R\rightarrow T$ an $\hcal{S}$-ring morphism. As in Lemma \ref{induction_anneau_tensor}, we obtain a morphism of $\hcal{T}$-rings $R\rightarrow \coindu{\hcal{S}}{\hcal{T}}{T}$. For any object $D$ of $\detaleproj{\hcal{S}}{R}$, the following map is an isomorphism in $\dmod{\hcal{T}}{\coindu{\hcal{S}}{\hcal{T}}{T}}$
	
	$$\coindu{\hcal{S}}{\hcal{T}}{T} \otimes_R D \rightarrow \bigcoindu{\hcal{S}}{\hcal{T}}{T\otimes_R D}, \,\,\, f\otimes d \mapsto [s \mapsto f(s) \otimes \varphi_{s,D}(d)].$$ 
	\end{enumerate}
\end{lemma}
\begin{proof}
	Similar to Lemma \ref{induction_anneau_tensor}, with one little change: the finite subtle index implies that $\hrm{Coindu}_{\hcal{S}}^{\hcal{T}}$ is a finite limit and not anymore a direct sum, this we must use the flatness of $D$ to ensure that the tensor product by $D$ commutes with coinduction.
\end{proof}

\begin{prop}\label{coindu_im_ess}
	Let $\hcal{S}$ be a monoid of finite subtle index inside a monoid $\hcal{T}$. Let $R$ be an $\hcal{S}$-ring. The essential image of $$\hrm{Coindu}_{\hcal{S}}^{\hcal{T}}\,: \, \detaleproj{\hcal{S}}{R} \rightarrow \dmod{\hcal{T}}{\coindu{\hcal{S}}{\hcal{T}}{R}}$$ contains $\detaleproj{\hcal{T}}{\coindu{\hcal{S}}{\hcal{T}}{R}}$.
\end{prop}
\begin{proof}
Similar to Proposition \ref{coindu_etproj} using Lemma \ref{induction_anneau_tensor_2}.
\end{proof}

\begin{rem}
	It seems difficult to find a monoidal condition guaranteeing an analogous corollary for étale modules. The étalness of coinduced modules is hard to grasp; having non-empty intersections of cosets tends to make the evaluation functions from the coinduction to the base module not surjective.
\end{rem}

	\vspace{1 cm}
	
	\section{The category of $\hcal{S}$-modules over $R$ with projective $r$-dévissage}\label{section_dévissage}

	This section introduces a subcategory of étale modules suitable for $\zp$-representations. Fontaine only considers the rings $\zp$, $\Oe$ and $\Oehat$, which are discrete valuation rings. Thanks to the structure of finite type modules over a discrete valuation ring, we have a lot of equivalent description $(\varphi,\Gamma)$-modules in the essential image of $\mathbb{D}$.
	
	We wonder which formulation of finite type $\zp$-representations' properties are preserved by scalar extension and taking of invariants with great generality. Each subquotient $p^n V/p^{n+1} V$ is a finite $\mathbb{F}_p$-vector space, hence finite projective of constant rank. The same is true for the dévissage by $V[p^{n+1}]/V[p^n]$. It appears (see Theorem \ref{complet_pf}) that for reasonable rings $R$, imposing finite presentation and the projectivity of the first dévissage have many additional consequences. This allows to deal simultaneously with torsion-free representations, torsion representations and their extensions. In addition, we can highlight that such condition suits dévissage strategies (see Theorem \ref{inv_dvdetaleproj_dévissage}).
	
	We begin by a heavy theorem, which underlines that finite presentation and projectivity of the dévissage impose a rigidity on modules.
	
	\begin{defi}
		A \textit{dévissage setup} is a pair $(R,r)$ where $R$ is a ring and $r\in R$ such that $R$ is $r$-torsion-free, $r$-adically complete and separated.
	\end{defi}

	\begin{defi}
	Let $M$ be an $R$-module. We say that $M$ has \textit{finite projective $(r,\mu)$-dévissage} if each $r^n M/r^{n+1}M$ is finite projective of constant rank as an $R/r$-module.
	
	We say that $M$ has finite projective $(r,\tau)$-dévissage if each $M[r^{n+1}]/M[r^n]$ is finite projective of constant rank as an $R/r$-module.
	\end{defi}

	\begin{theo}\label{complet_pf}
		Let $(R,r)$ be a dévissage setup. For every $M\in R\text{-}\hrm{Mod}$, the following are equivalent:
		
		\begin{enumerate}[label=\roman*),itemsep=0mm]
			\item $M$ is $r$-adically complete and separated with finite projective $(r,\mu)$-dévissage.

			\item $M$ is finitely presented with finite projective $(r,\mu)$-dévissage and bounded $r^{\infty}$-torsion.

			\item There exists $N\geq 1$, a finite projective $R$-module of constant rank $M_{\infty}$ and an $r^N$-torsion $R$-module with finite projective $(r,\mu)$-dévissage $M_{\hrm{tors}}$ such that 
			
			$$M\cong M_{\infty} \oplus M_{\hrm{tors}}.$$
		\end{enumerate}

	We also have that:
		
\begin{enumerate}[itemsep=0mm]
	\item Any $M$ satisfying the previous properties verifies
	
	$$\forall P\in R\text{-}\hrm{Mod}, \,\,\, P[r]=\{0\} \,\implies \,\hrm{Tor}^R_1(M,P)=\{0\}.$$

	\item Any $M$ satisfying the previous properties has finite projective $(r,\tau)$-dévissage.
	
	\item Suppose in addition that $\hrm{K}_0(R/r)=\Z$, i.e.  every finite projective $R/r$-module is stably free. Then the three above conditions on an $R$-module $M$ are also equivalent to
	
	\begin{enumerate}[label=\roman*),itemsep=0mm, start=4]
		\item There exists $N\geq 1$ and an isomorphism $$M\cong M_{\infty} \oplus \bigoplus_{1\leq n \leq N} M_n$$ where $M_{\infty}$ is a finite projective $R$-module of constant rank and each $M_n$ is a finite projective $R/r^n$-module of constant rank.
	\end{enumerate}
\end{enumerate}
	\end{theo}
\begin{proof}
	The proof is spread in the appendix among Theorem \ref{complet_pf_1}, Proposition \ref{complet_pf_2} and Theorem \ref{complet_pf_3}.
\end{proof}

\begin{rem}
	One could also show that this is equivalent to $M$ being finitely presented with finite projective $(r,\mu)$-dévissage and $r$-adically separated.
	
	We also notice that, for $R$ noetherian, the condition on the $r^{\infty}$-torsion is automatic. I do not know in general wether the condition is automatic for finitely presented $R$ modules with finite projective $(r,\mu)$-torsion when $(R,r)$ is a dévissage setup. After quick manipulations, we realise the existence of a finitely presented module with unbounded $r^{\infty}$-torsion is equivalent to the existence of a finitely generated submodule $M\subset R^d$ such that the $r$-adic topology on $R^d$ do not induce the $r$-adic topology on $M$. We also  encountered this kind of issue in the topological section (see Proposition \ref{inv_cdvdetaleproj}).
\end{rem}

\begin{rem}
	The condition $iv)$ is not a priori preserved by a Fontaine-type functor. Even if finite type $\zp$-representations have such decomposition as $\zp$-modules as well as their base change to $\Oehat$, preserving this decomposition through taking of invariants (i.e. by Galois descent) would require such decomposition to be $\cg_{\qp}$-invariant.
	
	Fix a non trivial character $\chi \,: \, \cg_{\qp} \rightarrow \mathbb{F}_p$. Define $$V:= \left(\quot{\Z}{p^2 \Z}\right) e_1 \oplus \mathbb{F}_p e_2$$ with $\zp$-linear Galois action given by \begin{align*} \sigma \cdot e_1 &= e_1 + \chi(\sigma) e_2 \\ \sigma \cdot e_2 &= p\chi(\sigma) e_1 + e_2 \end{align*} The only stable submodule of $V[p]$ is $V[p] \cap pV$ which forbids a $\cg_{\qp}$-invariant decomposition. Worst, no submodule isomorphic to $\Z/p^2 \Z$ is stable by the Galois action. To see more similar examples (even for multivariable variants) and computation of the associated (multivariable) $(\varphi,\Gamma)$-modules, see \cite{nataniel_calcul_phig}.
\end{rem}
	
	We add monoid action. We fix a monoid $\hcal{S}$ for the rest of this section.
 	
 	\begin{defi} An \textit{$\hcal{S}$-dévissage setup} is a pair $(R,r)$ where $R$ is an $\hcal{S}$-ring, where $r\in R$, such that $(R,r)$ is a dévissage setup and that the element $r$ verifies $$\forall s \in \hcal{S}, \,\,\, \varphi_s(r)R = r R.$$
 	\end{defi} 
 	
 	For an $\hcal{S}$-dévissage setup $(R,r)$, the morphisms $(\varphi_s)_{s\in\hcal{S}}$ restrict-corestrict to an $\hcal{S}$-ring structure on $R/r$. The quotient map $R\rightarrow R/r$ is a morphism of $\hcal{S}$-rings.

	\begin{defi}\label{defi_detaledvproj}
	Let $(R,r)$ be an $\hcal{S}$-dévissage setup. The category $\detaledvproj{\hcal{S}}{R}{r}$, called the \textit{étale $\hcal{S}$-modules over $R$ with projective $r$-dévissage}, is the full subcategory of $\detale{\hcal{S}}{R}$ whose objects have an underlying $R$-module with finite projective $(r,\mu)$-dévissage and bounded $r^{\infty}$-torsion.
	\end{defi}

\begin{rem}
	The condition $\varphi_s(R)\in rR$ alone would be sufficient to define a structure of $\hcal{S}$-ring on $R/r$. Here, we impose that $\varphi_s(r)R =rR$ to transmit the $\hcal{S}$-action from $D$ to its two dévissages. Without this hypothesis, $\varphi_{s,D}$ sends $D[r]$ to $D[\varphi_s(r)]$ which might be a lot bigger. For similar reasons, even if $\varphi_s(R)\in rR$  alone implies that $\varphi_{s,D}$ restricts to $r^n D$, it might be zero modulo $r^{n+1} D$ and lose étaleness of the action.
\end{rem}

\begin{rem}
	We will use the modules with projective $r$-dévissage to find the necessary conditions on the essential image of a Fontaine-type functor for representations over $\hcal{O}_K$ with for $K$ a local $p$-adic field. In every case I can think of, the ring $R$ will be a $\hcal{O}_K$-algebra, $r$ will be a uniformiser of $\hcal{O}_K$ and the $\varphi_s$ will be at best $\hcal{O}_K$-algebra morphisms, at worst semilinear algebra morphisms with respect to Galois action on $K$. In any case, the condition $\varphi_s(r)R=rR$ is verified.
			
	Our introduction of such dévissage setting aims to automate the "dévissage and passage to limit" steps of Fontaine equivalences. The $r$-adic separation and completeness of $R$ are thus essential conditions on a dévissage setting for such strategy to make sense.
\end{rem}

\begin{lemma}\label{dev_etale}
	Let $(R,r)$ be an $\hcal{S}$-dévissage setup and $D$ be an object of $\detaledvproj{\hcal{S}}{R}{r}$.
	
	\begin{enumerate}[itemsep=0mm]
		\item For each $r$-torsion-free module $P$ we have $\hrm{Tor}^R_1(D,P)=\{0\}$.
	
		\item For each $n\geq 0$, the morphisms $(\varphi_{s,D})_{s\in \hcal{S}}$ restrict-corestrict to $r^n D$. With this $\hcal{S}$-action, $r^n D$ belongs to $\detaledvproj{\hcal{S}}{R}{r}$. Quotienting gives a structure of object of $\detaleproj{\hcal{S}}{\sfrac{R}{r}}$ on each $r^n D/r^{n+1} D$.
	
		\item For each $n\geq 0$, the morphisms $(\varphi_{s,D})_{s\in \hcal{S}}$ restrict-corestrict to $D[r^n]$. With this $\hcal{S}$-action, $D[r^n]$ belongs to $\detaledvproj{\hcal{S}}{R}{r}$. Quotienting gives a structure of object of $\detaleproj{\hcal{S}}{\sfrac{R}{r}}$ on each $D[r^{n+1}]/D[r^n]$.
	\end{enumerate}
\end{lemma}
\begin{proof}
	The vanishing of $\hrm{Tor}^R_1$ only restates a part of Theorem \ref{complet_pf}.
	
	Let $n\geq 0$. By Theorem \ref{complet_pf}, the $R$-module $D$ is $r$-adically complete and separated, thus $r^n D$ is also. Finite projectivity for $(r,\mu)$-dévissage of $r^n D$ comes from the dévissage of $D$. The Theorem \ref{complet_pf} then proves that $r^n D$ is finitely presented. In addition, the condition $\varphi_s(r)R=rR$ allows to identify $\varphi_{s,r^n D}^*$ to
	
	$$\varphi_s^* (r^n D) \cong \varphi_s(r)^n \varphi_s^*D= r^n \left(\varphi_s^*D\right) \xrightarrow[\varphi_{s,D}^*]{\sim}   r^n D.$$ Hence $r^n D$ is étale. 
	
	At the level of abelian groups, applying $\varphi_s^*$ is tensoring with $R$ viewed as $R$-module via $\varphi_s$. As $R$ is $\varphi_s(r)$-torsion-free, Lemma \ref{tor_torsion_gen} proves the vanishing of $\hrm{Tor}^R_1(R_{\varphi_s},\sfrac{r^n D}{r^{n+1} D})$; the exact sequence defining the quotient $r^n D/r^{n+1} D$ stays exact after passing to $\varphi_s^*$. The five lemma concludes that $r^n D/r^{n+1}D$ is étale (we already knew it was finite projective).
	
	For the $(r,\tau)$-dévissage, we give the argument for étaleness. Because $r^n D$ is still étale with finite projective $(r,\mu)$-dévissage and bounded $r^{\infty}$-torsion, we can apply the $\hrm{Tor}^R_1$-vanishing to $r^n D$; the exact sequence $$0 \rightarrow D[r^n] \rightarrow D \xrightarrow{r^n\times} r^n D \rightarrow 0,$$ it is still exact after passing to $\varphi_s^*$. Hence, the map $\varphi_{s,D}^*$ sends $\varphi_s^*(D[r^n])=(\varphi_s^*D)[r^n]$ to $D[r^n]$. and the five lemma concludes for the étaleness of $D[r^n]$. As above, we can transmit étaleness to the quotient.
\end{proof}

There is a reciprocal.

\begin{lemma}\label{etale_modulo_r}
	Let $(R,r)$ be an $\hcal{S}$-dévissage setup. Let $D$ be an object of $\dmod{\hcal{S}}{R}$ such that the underlying $R$ module is of finite presentation with finite projective $(r,\mu)$-dévissage and bounded $r^{\infty}$-torsion. If each $r^n D/r^{n+1} D$ belongs to $\detaleproj{\hcal{S}}{\sfrac{R}{r}}$, then $D$ belongs to $\detaledvproj{\hcal{S}}{R}{r}$.
\end{lemma}
\begin{proof}
	The first point of Theorem \ref{complet_pf} makes $\varphi_s^* -$ and the formation of $(r,\mu)$-dévissage commute. It implies that the $R$-module $\varphi_s^* D$ is finitely presented with finite projective $(r,\mu)$-dévissage. Hence, the source and target of $\varphi_{s,D}^*$ are both $r$-adically complete and separated; the fact that $\varphi_{s,D}^*$ is an isomorphism will be deduced from the fact that it is an isomorphism on each term of the $(r,\mu)$-dévissage.
\end{proof}

\begin{prop}\label{monoidal_on_detaledv}
	Let $(R,r)$ be an $\hcal{S}$-dévissage setup. The closed symmetric monoidal structure on $\detale{\hcal{S}}{R}$ verifies that
	\begin{enumerate}[itemsep=0mm]
		\item The full subcategory $\detaledvproj{\hcal{S}}{R}{r}$ of $\dmod{\hcal{S}}{R}$ is closed under the symmetric monoidal structure.
		
		\item If $\hrm{K}_0(\sfrac{R}{r})=\Z$, this full subcategory is also closed under internal $\hrm{Hom}$.
	\end{enumerate}
\end{prop}
\begin{enumerate}[itemsep=0mm]
	\item Let $D_1$ and $D_2$ be objects of $\detaledvproj{\hcal{S}}{R}{r}$. Let's analyse the $(r,\mu)$-dévissage of $D_1\otimes_R D_2$. Tensoring the exact sequence corresponding to $rD_2 \subset D_2$ by $D_1$, we obtain an exact sequence \begin{equation}\label{eqn:exactplusconnect}\hrm{Tor}_1^R(\quot{D_1}{rD_1},D_2) \rightarrow rD_1 \otimes_R D_2 \rightarrow D_1\otimes_R D_2 \tag{**}\end{equation} where the image of the last morphism is $r(D_1\otimes D_2)$. As $D_1/rD_1$ is a finite projective $R/r$-module, pushing further the analysis of $\hrm{Tor}$ at the beginning of Lemma \ref{tor_torsion_gen} gives a identification $$\hrm{Tor}_1^R\left(\quot{D_1}{rD_1},D\right) \cong \quot{D_1}{rD_1} \otimes_{\sfrac{R}{r}} D[r]$$ natural in $D_1$ and $D$. 
	
	We will compute the connecting morphism at the left of (\ref{eqn:exactplusconnect}). Consider the exact sequence $$0\rightarrow R \xrightarrow{r\times} R \rightarrow \quot{R}{r} \rightarrow 0$$ and enhance it to an exact sequence of $R$-projective resolutions as follows 
	
	\begin{center}
	\begin{tikzcd}
		&  \arrow[d,dotted] &  \arrow[d,dotted] & \arrow[d,dotted] & \\ 0 \ar[r] & 0 \ar[r] \ar[d] & 0 \ar[r] \ar[d] & 0 \ar[r] \ar[d] & 0 \\  0 \ar[r] & 0 \ar[r] \ar[d] & R \ar[r,"\hrm{Id}"] \ar[d,"(-\hrm{Id} \text{, }r\times)"] & R \ar[r] \ar[d,"r\times"] & 0 \\ 0 \ar[r] & R \ar[r,"i_1"]\ar[d,"\hrm{Id}"',two heads] & R\oplus R \ar[r,"p_2"] \ar[d,"(r\times) + \hrm{Id}",two heads] & R \ar[r] \ar[d,two heads] & 0 \\ 0 \ar[r] & R \ar[r,"r\times"] & R \ar[r] & \quot{R}{r} \ar[r] & 0	
 	\end{tikzcd}
	\end{center} To compute the connecting morphism, we tensor with $D_2$, then apply the snake lemma from the kernel of the right middle-height map (which is $D_2[r]$) to the second-from-the-bottom left term (which is $D_2$). This morphism is merely the inclusion. Now fix $d_1 \in D_1$ and use the following morphism of exact sequences
	
	\begin{center}
		\begin{tikzcd}
			0 \ar[r] & R \ar[r,"r\times"] \ar[d,"\times rd_1"'] & R \ar[r] \ar[d,"\times d_1"']  & \quot{R}{r} \ar[r] \ar[d,"\times d_1"'] & 0 \\
			0 \ar[r] & rD_1 \ar[r] & D_1 \ar[r] & \quot{D_1}{rD_1} \ar[r] & 0
		\end{tikzcd}
	\end{center} which gives us a commutative square with horizontal maps being connecting morphisms:
	
	\begin{center}
		\begin{tikzcd}
			D_2[r] \cong \hrm{Tor}_1^R\left(\quot{R}{r},D_2\right) \ar[rr] \ar[d,"d_1 \otimes \,-"'] & & D_2 \ar[d,"(rd_1) \otimes \,-"'] \\
			\quot{D_1}{rD_1} \otimes_{\sfrac{R}{r}} D_2[r] \cong \hrm{Tor}_1^R\left(\quot{D_1}{rD_1},D_2\right) \ar[rr] & & rD_1\otimes_R D_2
		\end{tikzcd}
	\end{center} Because the upper horizontal map always identifies with the inclusion, the connecting morphism we look for is $$\quot{D_1}{rD_1} \otimes_{\sfrac{R}{r}} D_2[r] \rightarrow rD_1 \otimes_R D_2, \,\,\, d_1 \otimes d_2 \mapsto rd_1\otimes d_2.$$ Hence, we obtain that $$r(D_1\otimes_R D_2) \cong rD_1\otimes_R \quot{D_2}{D_2[r]}.$$ As we only used the $(r,\mu)$-dévissage of $D_1$, we can apply the result to $rD_1$ and $D_2/D_2[r]$ to obtain $$r^2(D_1\otimes_R D_2) \cong r^2 D_2 \otimes_R \quot{D_2}{D_2[r^2]}.$$ Applying recursivement and taking quotients shows that
	
	\begin{align*}
		\forall n\geq 0, \,\,\,\quot{r^n (D_1\otimes_R D_2)}{r^{n+1}(D_1\otimes_R D_2)} & \cong \quot{\left(r^n D_1 \otimes_R \sfrac{D_2}{D_2[r^n]}\right)}{\left(r^{n+1} D_1 \otimes_R \sfrac{D_2}{D_2[r^{n+1}]}\right)} \\ & \cong \left(\quot{r^n D_1}{r^{n+1} D_1}\right) \otimes_{\sfrac{R}{r}} \left(\quot{D_2}{\left(r D_2 + D_2[r^n]\right)}\right)
	\end{align*} The first term of this tensor product is finite projective over $R/r$ because $D_1 \in \detaledvproj{\hcal{S}}{R}{r}$. The second term is isomorphic to $r^n D_2/r^{n+1} D_2$ as $(D_2[r^n]+ r D_2)/rD_2$ is the kernel of $$\quot{D_2}{rD_2} \xrightarrowdbl{r^n\times} \quot{r^n D_2}{r^{n+1}D_2}.$$ Their tensor product is still finite projective.
	
		We finally use $iii)$ of Theorem \ref{complet_pf_1}. We have that $(D_1)_{\infty} \otimes_R (D_2)_{\infty}$ is finite projective, a fortiori $r$-torsion free. The other terms $(D_1)_{\hrm{tors}} \otimes_R (D_2)_{\infty}$, $(D_1)_{\infty} \otimes_R (D_2)_{\hrm{tors}}$ and $(D_1)_{\hrm{tors}} \otimes_R (D_2)_{\hrm{tors}}$ are of $r^{N_1+N_2}$-torsion, with $D_1$ (resp. $D_2$) being of $r^{N_1}$- (resp. $r^{N_1}$-)torsion.
	
	\item Let $D_1$ and $D_2$ be objects of $\detaledvproj{\hcal{S}}{R}{r}$. Theorem \ref{complet_pf_3} allows to fix two decompositions $$D_1=D_{1,\infty}\oplus \bigoplus_{1\leq n \leq N} D_{1,n} \, \text{  and  } \, D_2=D_{2,\infty}\oplus \bigoplus_{1\leq n \leq N} D_{2,n}$$ where the $D_{i,\infty}$ are finite projective $R$-modules and the $D_{i,n}$ are finite projective $R/r^n$-modules. We obtain that	
	\begin{align*}
		\hrm{Hom}_R(D_1,D_2) &=  \hrm{Hom}_R(D_{1,\infty},D_{2,\infty}) \oplus \bigoplus_{1\leq n\leq N}\hrm{Hom}_R\left(D_{1,\infty},D_{2,n}\right) \oplus  \bigoplus_{1\leq i,j \leq N} \hrm{Hom}_R\left(D_{1,i},D_{2,j}\right) \\
		&=\hrm{Hom}_R(D_{1,\infty},D_{2,\infty}) \oplus \bigoplus_{1\leq n\leq N}\hrm{Hom}_{R/r^n}\left(\quot{D_{1,\infty}}{r^n D_{1,\infty}},D_{2,n}\right) \\ 
		& \phantom{=} \phantom{aaaa} \oplus  \bigoplus_{1\leq i,j \leq N} \hrm{Hom}_{R/r^{\min(i,j)}}\left(\quot{D_{1,i}}{r^{\min(i,j)} D_{1,j}},D_{2,j}[r^{\min(i,j)}]\right) \\
		&\cong  \hrm{Hom}_R(D_{1,\infty},D_{2,\infty}) \oplus \bigoplus_{1\leq n\leq N}\hrm{Hom}_{R/r^n}\left(\quot{D_{1,\infty}}{r^n D_{1,\infty}},D_{2,n}\right) \\
		&\phantom{=}  \phantom{aaaa}\oplus  \bigoplus_{1\leq i,j \leq N} \hrm{Hom}_{R/r^{\min(i,j)}}\left(\quot{D_{1,i}}{r^{\min(i,j)} D_{1,j}},\quot{D_{2,j}}{r^{\min(i,j)}D_{2,j}}\right)
	\end{align*} where the last isomorphism uses the two first points of Lemma \ref{lemma_pour_ko_1}. Each module in the $\hrm{Hom}$ sets is finite projective over the corresponding ring. First, this implies that $\underline{\hrm{Hom}}_R(D_1,D_2)$ is finitely presented with finite projective $(r,\mu)$-dévissage (apply the fourth point of Lemma \ref{hom_fin_proj} multiple times then Theorem \ref{complet_pf_3}). Then, we remark that the previous isomorphism upgrades to an isomorphism of $\hcal{S}$-modules for $\underline{\hrm{Hom}}$. Each term on the right is étale (use the first point of Proposition \ref{constr_ext_etale} on each $R\rightarrow R/r^n$) then the second point of Proposition \ref{monoidalstructure} applied for all the rings $R/r^n$ concludes that $\underline{\hrm{Hom}}_R(D_1,D_2)$ is étale.
\end{enumerate}

	\smallskip
	
	We now study the preservation of such conditions by base change. Let $(R,r)$ be an $\hcal{S}$-dévissage setup. Let $a$ be a morphism of $\hcal{S}$-rings $a\,: \,R\rightarrow T$ such that the pair $(T,a(r))$ is a dévissage setup. For all $s$, the condition $\varphi_s(r) R = rR$ implies that we have $\varphi_s(a(r)) T= a(r)T$. This way, both $\hcal{S}$-rings $R$ and $T$ are suitable for the dévissage strategy.

	\begin{prop}\label{dvproj_ext}
		In the setup above, the functor $\hrm{Ex}$ sends $\detaledvproj{\hcal{S}}{R}{r}$ to $\detaledvproj{\hcal{S}}{T}{a(r)}$.
	\end{prop}
	\begin{proof}
	The second point of Proposition \ref{constr_ext_etale} already tells that $\hrm{Ex}(D)$ belongs to $\detale{\hcal{S}}{T}$.
		
		Let $n\geq 0$. The $R$-module $D/r^n D$ is of $r^n$-torsion with finite projective $(r,\mu)$-dévissage and the ring $T$ is $a(r)$-torsion-free; Lemma \ref{tor_torsion_gen} implies that $\hrm{Tor}^R_1\left(\sfrac{D}{r^{n}D},T\right)=\{0\}$, which translates into the injectivity of $$T\otimes_R r^n D \xrightarrow{\hrm{Id}_T\,\otimes\, \hrm{inclusion}} T\otimes_R D.$$ Its image is $a(r)^n \left(T\otimes_R D\right)$. The obtained isomorphism $$T\otimes_R r^n D \xrightarrow {\sim} a(r)^n \left(T\otimes_R D\right)$$ being compatible with inclusions as $n$ varies, we deduce that $$\quot{a(r)^n \left(T\otimes_R D\right)}{a(r)^{n+1} \left(T\otimes_R D\right)} \cong T\otimes_R \quot{r^n D}{r^{n+1}D}=\quot{T}{a(r)}\otimes_{\sfrac{R}{r}}\quot{r^n D}{r^{n+1}D}.$$ By Proposition \ref{constr_ext_etale}, we conclude that $\hrm{Ex}(D)$ has finite projective $(a(r),\mu)$-dévissage.
		
		Using that $\hrm{Tor}_1^R(r^n D,T)=\{0\}$, we obtain that $(T\otimes_R D)[a(r)^n]=(T\otimes_R D[r^n])$ hence $\hrm{Ex}(D)$ has bounded $a(r)^{\infty}$-torsion.
	\end{proof}

	\begin{prop}
		In the previous setup, suppose that $\hrm{K}_0(\sfrac{R}{r})=\Z$. Then, the functor $\hrm{Ex}$ is closed monoidal from $\cdetaledvproj{\hcal{S}}{R}{r}$.
	\end{prop}
	\begin{proof}
		By imitating the proof of the third point of Proposition \ref{constr_ext_etale}, we only need to prove that $i_{D,D',a}$ from Lemma \ref{hom_fin_proj} is an isomorphism as soon as $D$ is of finite projective $(r,\mu)$-dévissage and bounded $r^{\infty}$-torsion. As usual, we decompose $$D\cong D_{\infty} \oplus \bigoplus_{1\leq i \leq n} D_i.$$ Hence, we obtain that $$\hrm{Hom}_R(D,D') = \hrm{Hom}_R(D_{\infty},D') \oplus \bigoplus_{1\leq i \leq n} \hrm{Hom}_{R/r^n}(D_i, D'[r^n]).$$ Applying Lemma \ref{hom_fin_proj} for each term varying the ring gives the desired isomorphism.
	\end{proof}

\hspace{1 cm}

We move towards the preservation through invariants.
Let $(R,r)$ be an $\hcal{S}$-dévissage setup. As in section \ref{subsection_inv}, we fix a normal submonoid $\hcal{S}'\triangleleft\hcal{S}$ and impose that $r$ belongs to $R^{\hcal{S}'}$. The pair $(R^{\hcal{S}'},r)$ is automatically an $\hcal{S}/\hcal{S}'$-dévissage setup. 

\begin{prop}\label{inv_dvdetaleproj}
	In this setup suppose that:
	
	\begin{itemize}[itemsep=0mm]
		\item The inclusion $R^{\hcal{S}'}\subset R$ is faithfully flat.
		\item We have\footnote{This happens as soon as $R^{\hcal{S}'}/r$ is reduced. All applications will satisfy this stronger condition.}
	
	$$\forall t\in \left(\quot{R^{\hcal{S}'}}{r}\right), \,\exists n\geq 1, \,\,\, \left(\quot{R^{\hcal{S}'}}{r}\right)[t^{\infty}]= \left(\quot{R^{\hcal{S}'}}{r}\right)[t^n].$$
	
	\item The map $R^{\hcal{S}'}/r \rightarrow \left(R/r\right)^{\hcal{S}'}$ is an isomorphism\footnote{Because $R$ is $r$-torsion-free, this is equivalent to $\hrm{H}^1(\hcal{S}',R)$ being $r$-torsion-free.}.
	\end{itemize}
	
	Let $D$ be in $\detaledvproj{\hcal{S}}{R}{r}$ such that the comparison morphism $$ R\, \mathop{\otimes}_{R^{\hcal{S}'}} \hrm{Inv}(D)\rightarrow D$$ is an isomorphism. Then $\hrm{Inv}(D)$ belongs to $\detaledvproj{\sfrac{\hcal{S}}{\hcal{S}'}}{R^{\hcal{S}'}}{r}$.
\end{prop}
\begin{proof}
	The Proposition \ref{inv_detaleproj} already tells us that $\hrm{Inv}(D)$ belongs to $\detale{\sfrac{\hcal{S}}{\hcal{S}'}}{R^{\hcal{S}'}}$. 
	
	Let $n\geq 0$. Consider the commutative diagram:
	
	\begin{center}
		\begin{tikzcd}
			r^n\left(R\, \mathop{\otimes}_{R^{\hcal{S}'}} \hrm{Inv}(D)\right) \ar[r,"\sim"] \ar[d,two heads] & r^n D \\
			R\, \mathop{\otimes}_{R^{\hcal{S}'}} r^n \hrm{Inv}(D) \ar[ur]
		\end{tikzcd}
	\end{center} where the horizontal morphism is obtained by multiplying by $r^n$ the isomorphism of comparison. The diagonal morphism is an isomorphism. Compatibility with inclusions as $n$ varies gives an isomorphism in $\dmod{\hcal{S}}{R}$ as follows:
	
	$$R\, \mathop{\otimes}_{R^{\hcal{S}'}} \quot{r^n \hrm{Inv}(D)}{r^{n+1}\hrm{Inv}(D)} \xrightarrow[]{\sim } \quot{r^n D}{r^{n+1} D}$$  which identifies to an isomorphism in $\dmod{\hcal{S}}{\sfrac{R}{r}}$: $$j\,: \, \quot{R}{r} \, \mathop{\otimes}_{R^{\hcal{S}'}/r} \quot{r^n \hrm{Inv}(D)}{r^{n+1}\hrm{Inv}(D)} \xrightarrow[]{\sim } \quot{r^n D}{r^{n+1}D}.$$ 
	
	The comparison morphism for $r^n D/r^{n+1} D$ and $R/r$ appears horizontally in the commutative diagram:
	
	\begin{center}
		\begin{tikzcd}
			\quot{R}{r} \, \mathop{\otimes}_{R^{\hcal{S}'}/r} \hrm{Inv}\left(\quot{r^n D}{r^{n+1} D}\right) \ar[rr] & & \quot{r^n D}{r^{n+1}D} \\
			\quot{R}{r} \, \mathop{\otimes}_{\sfrac{R^{\hcal{S}'}}{(r))}} \hrm{Inv}\left(  \quot{R}{r} \, \mathop{\otimes}_{R^{\hcal{S}'}/r} \quot{r^n \hrm{Inv}(D)}{r^{n+1} \hrm{Inv}(D)}\right) \ar[u,"\hrm{Id}_R\otimes \hrm{Inv}(j)"',"\sim" labl] \\
			\quot{R}{r} \, \mathop{\otimes}_{R^{\hcal{S}'}/r} \quot{r^n \hrm{Inv}(D)}{r^{n+1} \hrm{Inv}(D)}  \ar[uurr,bend right,"j"',"\sim"] \ar[u,"\hrm{Id}_R\otimes c"']
		\end{tikzcd}
	\end{center} where $c$ is the morphism given at Proposition  \ref{inv_detale_et_proj}. We will apply this proposition after quick remarks.
	
	 The morphism $R^{\hcal{S}'}/r \hookrightarrow R/r$ is faithfully flat because $R^{\hcal{S}'} \subset R$ is also. Hence, the condition on $R^{\hcal{S}'}/r$ can be lifted to $R/r$. Finally, the isomorphism $R^{\hcal{S}'}/r \rightarrow \left(R/r\right)^{\hcal{S}'}$ finishes to prove that the ring $R/r$ verifies the conditions of Proposition \ref{inv_detale_et_proj}.
	
	
	Thanks to faithful flatness, the isomorphism $j$ underlines that $r^n \hrm{Inv}(D)/r^{n+1} \hrm{Inv}(D)$ is finite projective over $R^{\hcal{S}'}/r$. We apply Proposition \ref{inv_detale_et_proj} to $r^n \hrm{Inv}(D)/r^{n+1} \hrm{Inv}(D)$ and $R/r$, obtaining that the morphism $c$ is an isomorphism. Now, the comparison morphism for $r^n D/r^{n+1} D$ is an isomorphism and we use Proposition \ref{inv_detaleproj} with $R/r$ and $r^n D/r^{n+1} D$ to conclude that $\hrm{Inv}(D)$ is of finite projective $(r,\mu)$-dévissage. It is of bounded $r^{\infty}$-torsion as any subobject of $D$ is.
\end{proof}

This proposition uses the comparison isomorphism for $R$-modules to deduce them for all terms of the dévissage. The following corollary explains how to lift the comparison isomorphisms.

\begin{theo}\label{inv_dvdetaleproj_dévissage}
	In the invariant dévissage setup, suppose that:
	
	\begin{itemize}[itemsep=0mm]
		\item The inclusion\footnote{It is an injection because $R$ is $r$-torsion-free.} $R^{\hcal{S}'}/r\subset R/r$ is faithfully flat.
		
		\item We have $\hrm{H}^1(\hcal{S}',\sfrac{R}{r})=\{0\}$.
		
		\item For every object $D$ of $\detaleproj{\hcal{S}}{\sfrac{R}{r}}$ the comparison morphism $$ R\, \mathop{\otimes}_{R^{\hcal{S}'}} \hrm{Inv}(D)\rightarrow D$$ is an isomorphism. 
	\end{itemize}
	
	\noindent Then, the comparison morphism is an isomorphism for every object of $\detaledvproj{\hcal{S}}{R}{r}$ and the functor $\hrm{Inv}$ sends $\detaledvproj{\hcal{S}}{R}{r}$ to $\detaledvproj{\sfrac{\hcal{S}}{\hcal{S}'}}{R^{\hcal{S}'}}{r}$ and is closed strong symmetric monoidal.
\end{theo}
\begin{proof}
	Our first reflex is to say roughly "by dévissage". If Theorem \ref{complet_pf_1} proved that objects in $\detaledvproj{\hcal{S}}{R}{r}$ are $r$-adically complete and separated, nothing garantee that the left term of their comparison morphism is. We walk a tight path to prove that comparison morphisms are isomorphisms for every term of the $(r,\mu)$-dévissage then concludes for $D$ itself.
	
	Remark that we have dropped two hypothesis compared to Propositon \ref{inv_dvdetaleproj}. The fact that $R^{\hcal{S}'}/r \rightarrow \left(R/r\right)^{\hcal{S}'}$ is an isomorphism is proved on the way. The second hypothesis only aimed to recover the isomorphism $$\quot{r^n \hrm{Inv}(D)}{r^{n+1}\hrm{Inv}(D)} \xrightarrow{\sim} \hrm{Inv}\left(\quot{r^n D}{r^{n+1} D}\right)$$ from the comparison isomorphism for $D$. Here, it is recovered differently.

	\underline{Step 1:} we prove that for an $r$-adically complete and separated $\hcal{S}'$-module $D$ over $R$, we have $$\hrm{H}^1(\hcal{S}',D) \cong \lim \limits_{\longleftarrow} \hrm{H}^1\left(\hcal{S}',\quot{D}{r^nD}\right).$$
	Monoid cohomology is computed by cochain complexes; so we have exact sequences $$0 \rightarrow \sfrac{D}{\{d \, |\, \forall s, \,\, \varphi_{s,D}(d)-d \in r^n D\}} \cong \quot{\left(\sfrac{D}{r^n}\right)}{\left(\sfrac{D}{r^n D}\right)^{\hcal{S}'}} \xrightarrow{\hrm{d}^0} \hrm{Z}^1\left(\hcal{S}',\quot{D}{r^n D} \right) \rightarrow \hrm{H}^1\left(\hcal{S}',\quot{D}{r^n D} \right)\rightarrow 0.$$
	Passing to the limit, we obtain an exact sequence	$$ 0\rightarrow \quot{D}{D^{\hcal{S}'}} \xrightarrow{\hrm{d}^0} \hrm{Z}^1(\hcal{S}',D) \rightarrow \lim \limits_{\substack{\longleftarrow\\ n}} \hrm{H}^1\left(\hcal{S}', \quot{D}{r^n D} \right) \rightarrow \hrm{R}^1 \lim \limits_{\substack{\longleftarrow\\ n}} \,\,\quot{D}{\{d \, |\, \forall s, \,\, \varphi_{s,D}(d)-d \in r^n D\}}$$ The transition maps for the rightside system are surjective: it is Mittag-Leffler and the $\hrm{R}^1 \lim$ vanishes.

	\underline{Step 2:} we prove that every finite projective étale $\hcal{S}$-module over $R/r$ is $\hcal{S}'$-acyclic.
	
	There is a dévissage of each $\sfrac{R}{r^n}$ with subquotients isomorphic to $R/r$ as $\hcal{S}'$-abelian groups. With the cohomological hypothesis, dévissage tells that each $\hrm{H}^1(\hcal{S}',\sfrac{R}{r^n})$ vanishes. Thanks to step 1, we deduce that $\hrm{H}^1(\hcal{S}',R)$ vanishes.

	At this point, remark at this point that the hypothesis "$R$ is $r$-torsion-free" and the cohomological one give that $R^{\hcal{S}'}/r\cong \left(R/r\right)^{\hcal{S}'}$. Hence, the inclusion $\left(R/r\right)^{\hcal{S}'} \subset R/r$ is faithfully flat.
	
	Now take a finite projective étale $\hcal{S}$-module $D$ over $R/r$. Via the comparison morphism, it is isomorphic to $$\quot{R}{r} \otimes_{(\sfrac{R}{r})^{\hcal{S}'}} \hrm{Inv}(D)$$ and Proposition \ref{inv_detaleproj} shows that $\hrm{Inv}(D)$ is finite projective over $(R/r)^{\hcal{S}'}$. Fix a presentation $$\hrm{Inv}(D)\oplus P = (\quot{R^{\hcal{S}'}}{r})^k$$ giving a presentation $$D \oplus \left(\quot{R}{r}\otimes_{R^{\hcal{S}'}/r} P\right) \cong\left( \quot{R}{r}\otimes_{R^{\hcal{S}'}/r} \hrm{Inv}(D)\right) \oplus \left(\quot{R}{r}\otimes_{R^{\hcal{S}'}/r} P\right)\cong \left(\quot{R}{r}\right)^k$$ as $\hcal{S}'$-module over $R/r$. Monoid cohomology commutes to direct sums, hence the vanishing of $\hrm{H}^1(\hcal{S}',R)$ implies the vanishing of $\hrm{H}^1(\hcal{S}',D)$.

	\underline{Step 3:} go back to $D$ belonging to $\detaledvproj{\hcal{S}}{R}{r}$. Recall that by Lemma \ref{dev_etale}, the terms of the two dévissages belong to $\detaleproj{\hcal{S}}{\sfrac{R}{r}}$, so their comparison morphisms are isomorphisms. Step 2 proves that every $r^n D/r^{n+1} D$ is $\hcal{S}'$-acyclic. Moreover $D$ is of finite presentation hence complete and separated by Theorem \ref{complet_pf}. Dévissage and step 1 implies the vanishing of $\hrm{H}^1(\hcal{S}',D)$. Using Lemma \ref{dev_etale}, we obtain this vanishing for each $r^n D$ and each $D[r^n]$.

	\underline{Step 4:} the $\hcal{S}'$-acyclicity of $D[r^n]$ used on the exact sequence $$0\rightarrow D[r^n] \rightarrow D \xrightarrow{r\times} r^n D \rightarrow 0$$ implies that $r^n \hrm{Inv}(D) \rightarrow \hrm{Inv}(r^n D)$ is an isomorphism. The $\hcal{S}'$-acyclicity of $r^{n+1} D$, used on the exact sequence $$0\rightarrow r^{n+1} D \rightarrow r^n D \rightarrow \quot{r^n D}{r^{n+1} D} \rightarrow 0$$ implies that $\hrm{Inv}(r^n D)/\hrm{Inv}(r^{n+1} D) \rightarrow \hrm{Inv}(r^n D/r^{n+1}D)$ is an isomorphism. Combined, this exhibits an isomorphism of $\hcal{S}'$-modules over $R^{\hcal{S}'}/r$ between $r^n \hrm{Inv}(D)/r^{n+1}\hrm{Inv}(D)$ and $\hrm{Inv}(r^n D/r^{n+1}D)$.

	The comparison isomorphism for $r^n D/r^{n+1} D$ descends finite projectivity to the $\left(R/r\right)^{\hcal{S}'}$-module $\hrm{Inv}(r^n D/r^{n+1} D)$. The isomorphism with $r^n \hrm{Inv}(D)/r^{n+1} \hrm{Inv}(D)$ concludes that $\hrm{Inv}(D)$ has finite projective $(r,\mu)$-dévissage.  Moreover, $D$ being $r$-adically complete and separated implies that $\hrm{Inv}(D)$ is too. We can apply Theorem \ref{complet_pf} to show that $\hrm{Inv}(D)$ is finitely presented. Hence, $R\otimes_{R^{\hcal{S}'}} \hrm{Inv}(D)$ is finitely presented, therefore complete and separated by Theorem \ref{complet_pf}.
 	
 	\underline{Step 5:} generalising the previous step, we could obtain isomorphisms $$\forall n,k\geq 0, \,\,\,\quot{r^n (R\otimes_{R^{\hcal{S}'}} \hrm{Inv}(D))}{r^{n+k}(R\otimes_{R^{\hcal{S}'}} \hrm{Inv}(D))} \cong R\otimes_{R^{\hcal{S}'}} \hrm{Inv}\left(\quot{r^n D}{r^{n+k} D}\right)$$ compatible with reduction and multiplication. We can therefore use dévissage from the comparison isomorphisms for each $r^n D/r^{n+1} D$ to obtain the comparison morphism for $D$ is an isomorphism modulo $r^n$ for each $n\geq 1$. Step 4 proved that the left side of the comparison morphism for $D$ is $r$-adically complete and separated, and we already new that the right side $D$, is complete and separated. Thus, it is an isomorphism. 
 	
 	Thanks to Proposition \ref{inv_detaleproj}, we know that each $\hrm{Inv}(r^n D/r^{n+1} D)$ belongs to $\detaleproj{\sfrac{\hcal{S}}{\hcal{S}'}}{\sfrac{R^{\hcal{S}'}}{r}}$, hence Step 4 implies that each $r^n \hrm{Inv}(D)/r^{n+1}\hrm{Inv}(D)$ belongs to the same subcategory. Then Lemma \ref{etale_modulo_r} concludes that $\hrm{Inv}(D)$ belongs to $\detaledvproj{\sfrac{\hcal{S}}{\hcal{S}'}}{\sfrac{R^{\hcal{S}'}}{r}}{r}$.
 	
 	\underline{Step 6:} it remains to show that the functor is closed strong symmetric monoidal. The isomorphisms between $R^{\hcal{S}'}$-modules that we must prove are between complete and separated ones. Hence, it suffices to proves that the induced morphisms between their $(r,\mu)$-dévissages are isomorphisms. For this, use the identifications of the devissage of invariants proved above, the strategy of Corollary \ref{inv_monoidal} and Proposition \ref{inv_closed} and the  fully faithfulness of $R^{\hcal{S}'}/r\subset R/r$.
\end{proof}

\begin{prop}\label{inv_detale_et_dvproj}
	In the same setup, suppose that:
	
	\begin{itemize}[itemsep=0mm]
		\item We have $\hrm{H}^1(\hcal{S}',\sfrac{R}{r})=\{0\}$.
		\item We have $$\forall t\in R^{\hcal{S}'}/r, \, \exists n\geq 1, \,\,\, \left(R/r\right)[t^{\infty}]=\left(R/r\right)[t^n].$$
	\end{itemize}
	
	Let $D$ be a finitely presented $R^{\hcal{S}'}$-module with finite projective $(r,\mu)$-dévissage and bounded $r^{\infty}$-torsion	. Then, the map $$c\, : \, D \rightarrow \hrm{Inv}\left(R\, \mathop{\otimes}_{R^{\hcal{S}'}} D \right)$$ is an isomorphism of $R^{\hcal{S}'}$-modules. If $D$ was an $\hcal{S}/\hcal{S}'$-module over $R^{\hcal{S}'}$, then $c$ is an isomorphism of $\hcal{S}/\hcal{S}'$-modules over $R^{\hcal{S}'}$.
\end{prop}
\begin{proof}
	As $R$ is without $r$-torsion, the same arguments than in Lemma \ref{etale_modulo_r} shows that $R\otimes_{R^{\hcal{S}'}}-$ commutes with the formation of the $(r,\mu)$-dévissage; hence $\left(R\otimes_{R^{\hcal{S}'}} D\right)$ is a finitely presented $R$-module with finite projective $(r,\mu)$-dévissage. By Theorem \ref{complet_pf}, it is $r$-adically complete and separated and so is $\hrm{Inv}\left(R\otimes_{R^{\hcal{S}'}} D\right)$. We obtained that the source and target of $c$ are $r$-adically complete and separated hence it is an isomorphism as soon as it is on each term of the $(r,\mu)$-dévissage.
	
	Like in the proof of Theorem \ref{inv_cdvdetaleproj_dévissage}, we can show that for every finitely presented $R^{\hcal{S}'}$-module with finite projective $(r,\mu)$-dévissage $\left(R\otimes_{R^{\hcal{S}'}} D\right)$ is $\hcal{S}'$-acyclic. Hence, the functor $\hrm{Inv}\left(R\otimes_{R^{\hcal{S}'}} -\right)$ commutes to the formation of $(r,\mu)$-dévissage. The isomorphisms for the $(r,\mu)$-dévissage then follow from Proposition \ref{inv_detale_et_proj}.
\end{proof}

\vspace{1 cm}

	\section{Adding topology}\label{section_topo}

	As mentionned in the introduction, number theory considers \textit{continuous} Galois representations. For finite dimensional $\mathbb{F}_p$-linear representations of $\cg_{\qp}$ continuity traditionally means smoothness, for finite type $\zp$-representations this means continuity for the $p$-adic topology. These properties should impose topological conditions on Fontaine-type functors' essential image. In the literature, P. Schneider's book \cite{schneider_tate_mod} put aside, topological issues are treated quickly (e.g. in \cite{Fontaine_equiv}), elusively (e.g. in \cite{zabradi_equiv} where it does not show up in proofs) or incorrectly. An idea already existing in part of the literature is to equip finite type modules with a canonical topology we call the initial topology. This way, for a topological monoid $\hcal{S}$ acting continuously on a topological ring $R$, continuity of the action on an object of $\dmod{\hcal{S}}{R}$ becomes a property. Therefore, this section aims to input topological data in our categories of $\hcal{S}$-modules over $R$.

	\subsection{Initial topology and associated subcategories}

	In this section, we fix $\hcal{S}$ a topological monoid.
	
	\begin{defi}
		We define the category $\hcal{S}\text{-}\hrm{Ring}$ of \textit{topological $\hcal{S}$-rings}. Its objects are the $\hcal{S}$-rings $R$ endowed with a topological ring structure such that the underlying map $$\hcal{S}\times R \rightarrow R, \,\,\, (s,r)\mapsto \varphi_s(r)$$ is continuous. Its morphisms are the continuous morphisms of $\hcal{S}$-rings.
	\end{defi}

	For this section we fix a topological $\hcal{S}$-ring $R$. Our first idea might be to define a general category of topological $R$-modules and even go to the condensed world, but looking back to representations we notice that finite dimensional $\mathbb{F}_p$-linear (resp. $\zp$-linear) representations are always endowed with a specific topology: the discrete (resp. the $p$-adic) one. There is indeed an intrinsic topological structure on finite type modules (see \cite[Section 2.2]{schneider_tate_mod}).
	
	\begin{lemma}\label{topo_can_bien}
		Let $M$ be a finite type $R$-module and fix a quotient map $\pi \,: \,R^k \twoheadrightarrow M$. The quotient topology makes $M$ a topological $R$-module. Moreover, every $R$-linear map $f$ from $M$ to a topological $R$-module $N$ is continuous for this topology.
	\end{lemma}
	\begin{proof}
		Quotient topology from a topological $R$-module is a topological $R$-module structure; we will give a bit of the argument to familiarise the reader. We check that the external multiplication is continuous. Quotient topology allows to check this after pre-composition by $\hrm{Id}_R\times \pi$. The following diagram is commutative, where the lower map is the multiplication on $R^k$:
		
		\begin{center}
			\begin{tikzcd}
				R\times M \ar[r] & M \\
				R \times R^k \ar[u,"\hrm{Id}_R\times \pi"] \ar[r] & R^k \ar[u,"\pi"']
			\end{tikzcd}
		\end{center} The composition via the down-right corner is continuous and commutativity of the diagram concludes.
		
	Now, let $f$ be a $R$-linear map from $M$ to a topological $R$-module $N$. Universal property of quotient topology allows to check continuity on $(f\circ \pi)$. let $(e_i)$ be the canonical basis of $R^k$ and $n_i=(f\circ \pi)(e_i)$.The following diagram is commutative:
		
		\begin{center}
			\begin{tikzcd}
				M \ar[rr,"f"] & & N \\
				R^k \ar[u,"\pi"] \ar[rr,"(r_i)\mapsto (r_i n_i)"'] & & N^k \ar[u,"\sum"']
			\end{tikzcd}
		\end{center} Seing $(f\circ \pi)$ as a composition via down-right corner highlights its continuity.
	\end{proof}
	
	\begin{defiprop}
		For any finite type $R$-module $M$, the quotient topology given by a quotient map doesn't depend on the chosen quotient map; we call it the \textit{initial topology}.
		
		The functor from finite type $R$-modules to topological $R$-modules endowing them with initial topology is fully faithful.
	\end{defiprop}
	\begin{proof}
		Applying Lemma \ref{topo_can_bien} to $\hrm{Id}_N$ with two quotient topologies proves that $\hrm{Id}_N$ is an homeomorphism: the topology doesn't depend on the quotient map. The same lemma shows that any $R$-linear map between two finite type $R$-modules is continuous if we see them with their initial topology. This is exactly the fully faithfulness assertion.
	\end{proof}
	
	This topology also behaves very well with respect to quotients.
	
	\begin{lemma}\label{finest_quot_subob}
	For any surjection $p\,: \, M \twoheadrightarrow N$ of finite type $R$-modules, the quotient topology on $N$ coming from the initial topology on $M$ is the initial topology on $N$.
	\end{lemma}
	\begin{proof}
		Consider that for any quotient map $\pi \,: \, R^k \twoheadrightarrow M$, we have that $(p\circ \pi)$ is a quotient map for $N$ and endows it with initial topology. But as it factorises by $\pi$, the map $(p\circ \pi)$ induces the quotient topology associated to $p$.
			
	\end{proof}

\begin{defi}
	The category of \textit{topological étale $\hcal{S}$-modules over $R$}, denoted by $\cdetale{\hcal{S}}{R}$, is the full subcategory of $\detale{\hcal{S}}{R}$ whose objects an action map
	
	$$\hcal{S}\times D \rightarrow D, \,\,\, (s,d)\mapsto \varphi_{s,D}(d)$$ continuous for the initial topology on $D$.
	
	The category of \textit{topological étale projective $\hcal{S}$-modules over $R$}, denoted by $\cdetaleproj{\hcal{S}}{R}$, is intersection of the full subcategories of topological étale and étale projective $\hcal{S}$-modules over $R$.
\end{defi}

\begin{ex}\label{ex_s_discrete}
	If $\hcal{S}$ has discrete topology, the continuity condition is automatic. It is equivalent to saying that all $\varphi_{s,D}$ are continuous endomorphisms. Then, for $s\in \hcal{S}$ and $(d_i)_{1\leq i \leq k}$ a generating family of $D$, this is tested on the map $$R^k\rightarrow D, \,\,\, (r_i) \mapsto \sum_i \varphi_s(r_i) \varphi_{s,D}(d_i)$$ which is continuous.
\end{ex}

\begin{rem}
	The previous continuity condition is often implicitly used\footnote{Or explicitely stated in \cite{schneider_tate_mod}.} to treat continuity in literature.
\end{rem}

\vspace{0.5 cm}

It seems difficult to give a simple condition for continuity in general because the initial topology allows to test the continuity after replacement of the module by a free one only at the source. It also happens that initial topology behaves better on finite projective modules. We illustrates both phenomena.
	
	\begin{lemma}\label{continuity_etale}
		Let $D$ be an object of $\detale{\hcal{S}}{R}$ and $\pi \,: \, R^k \rightarrow D$ a quotient map. The module $D$ belongs to $\cdetale{\hcal{S}}{R}$ if and only if each $$\hcal{S} \rightarrow D, \,\,\, s\mapsto \varphi_{s,D}(\pi(e_i))$$ is continuous.
	\end{lemma}
	\begin{proof}
 		Assume each stated composition is continuous. We check continuity of the action map on the composition $$\hcal{S}\times R^k \xrightarrow{\hrm{Id}_{\hcal{S}}\times \pi} \hcal{S} \times D \rightarrow D,$$ which happens to also decompose as
 		
 		$$\hcal{S}\times R^k \xrightarrow[]{\Delta_{\hcal{S}}\times \hrm{Id}_{R^k}} \hcal{S} \times \left(\hcal{S} \times R\right)^k \xrightarrow{\hrm{Id}_{\hcal{S}} \times \left((s,r)\mapsto \varphi_s(r)\right)^k} \hcal{S}\times R^k \xrightarrow[]{(s,r_i)\mapsto \sum r_i \varphi_{s,D}(\pi(e_i))} D$$ hence is continuous.
	\end{proof}
	
	\begin{lemma}\label{topo_can_proj_bien}		
		Let $D$ be a finite projective $R$-module. Take a presentation of $D$ as $D\oplus D'=R^k$. 
		
		\begin{enumerate}[itemsep=0mm]
			\item The topology induced as subset of $R^k$ is the initial topology. In particular, it doesn't depend on the chosen presentation.
			
			\item Suppose that $D$ belongs to $\detaleproj{\hcal{S}}{R}$. For each $s\in \hcal{S}$, the semilinear map $$R^k \xrightarrow[]{\pi} D \xrightarrow[]{\varphi_{s,D}} D \hookrightarrow R^k$$ produces a matrix $\hrm{M}_{s,D}$. The module $D$ belongs to $\cdetaleproj{\hcal{S}}{R}$ if and only if the map $s\mapsto \hrm{M}_{s,D}$ is continuous for the product topology on matrices.
		\end{enumerate}
	\end{lemma}
	\begin{proof}
		\begin{enumerate}[itemsep=0mm]
			\item The Lemma \ref{topo_can_bien} already establishes that the initial topology is the finest. The projection $\pi : R^k \!\rightarrow \!D$ coming from the presentation induces the initial topology. An open for the initial topolgy is $U \subseteq D$ such that $\pi^{-1}(U)=U\oplus D'$ is open in $R^k$. For such open, we have $(U\oplus D')\cap D=U$ hence $U$ is open for the induced topology.
			
			\item Because the initial topology on $D$ is simultaneously the quotient topology from $\pi$ and the induced topolgy from the inclusion, the continuity of $\hcal{S}\times D \rightarrow D$ can be tested on: $$\hcal{S}\times R^k \xrightarrow{\hrm{Id}_{\hcal{S}}\times \pi} \hcal{S} \times D \rightarrow D \xrightarrow{i} R^k.$$ If $s\mapsto \hrm{M}_{s,D}$ is continuous, this composition also decomposes as			
			$$\hcal{S}\times R^k \xrightarrow[]{\Delta_{\hcal{S}}\times \hrm{Id}_{R^k}} \hcal{S} \times \left(\hcal{S} \times R\right)^k \xrightarrow{\hrm{Id}_{\hcal{S}} \times \big[(s,r)\mapsto \varphi_s(r)\big]^k} \hcal{S}\times R^k \xrightarrow[]{(s,v)\mapsto \hrm{M}_{s,D}(v)} R^k$$ hence is continuous.
			
			Conversely, fix $d_i=\pi(e_i)$. The map sending $s$ to the $(i,j)$-th coefficient  of $\hrm{M}_{s,D}$ is the composition $$\hcal{S} \xrightarrow{\hrm{Id}_{\hcal{S}} \times d_i} \hcal{S} \times D \rightarrow D \xrightarrow{\iota} R^k \xrightarrow{p_j} R$$ which is continuous.
		\end{enumerate}
	\end{proof}

	\begin{prop}\label{tensor_on_top}
		The full subcategory $\cdetale{\hcal{S}}{R}$ of $\detale{\hcal{S}}{R}$ is a monoidal subcategory.
		
		The same is true for $\cdetaleproj{\hcal{S}}{R}$.
	\end{prop}
	\begin{proof}
		We check the stability by tensor product. Let $D_1$ and $D_2$ be objects of $\cdetale{\hcal{S}}{R}$. The third point of Proposition \ref{monoidal_structure} already proves that their tensor product is étale. It remains to check continuity. Fix two quotient maps $\pi_l \,: \, R^{k_l} \twoheadrightarrow D_l$. Because $(\pi_1(e_i)\otimes \pi_2(e_j))_{(i,j)\in \llbracket 1, k_1 \rrbracket \times \llbracket 1 , k_2\rrbracket}$ generates $D_1\otimes_R D_2$, Lemma \ref{continuity_etale} allows to check continuity on each $$\hcal{S}\rightarrow D_1\otimes_R D_2, \,\,\, s\mapsto \varphi_{s,D_1}(\pi_1(e_i))\otimes \varphi_{s,D_2}(\pi_2(e_j)).$$ They decompose as $$\hcal{S} \xrightarrow{s\mapsto (\varphi_{s,D_1}(\pi_1(e_i)), \varphi_{s,D_2}(\pi_2(e_j))} D_1 \times D_2 \xrightarrow{(d_1,d_2)\mapsto d_1\otimes d_2} D_1\otimes_R D_2.$$ The first morphism is continuous because $D_1$ and $D_2$ both belong to $\cdetale{\hcal{S}}{R}$. Continuity of the second one can be checked after pre-composition by each $\pi_l$'s; in the following commutative diagram 
			
		\begin{center}
			\begin{tikzcd}
				R^{k_1} \times R^{k_2} \ar[rr,"(r_i\text{,}r_j) \mapsto (r_i r_j)"] \ar[d,"\pi_1\times \pi_2"'] & & R^{k_1 k_2} \ar[d,"\pi_1 \otimes \pi_2"'] \\ D_1\times D_2 \ar[rr] & & D_1\otimes_R D_2
			\end{tikzcd}
		\end{center} the composition via the upper-right corner emphasises continuity.
		
		The ring $R$ is an object of $\cdetale{\hcal{S}}{R}$ precisely because it is a topological $\hcal{S}$-ring.
		
		The projective version uses the fourth point of Proposition \ref{monoidal_structure} rather than the third.
	\end{proof}

Finding an adjoint is harder and require additional conditions on the topological ring. We begin by a general technical lemma.
	
	\begin{lemma}\label{huber_ouvert_uniforme}
		Let $A$ be a topological ring. Let $M$ be a finite projective $A$-module with its initial topology. Let $(m_1,\ldots, m_k)$ be a generating family of $M$.
		
		\begin{enumerate}[itemsep=0mm]
			\item For every neighbourhood $U$ of $0$ in $A^k$, there exists a neighbourhood $V$ of $0$ in $M$ such that $$\forall m \in V, \,\, \exists (a_i)\in U, \,\,\, m=\sum_{i=1}^k a_i m_i.$$
			
			\item Supppose that $A^{\times}$ is open in $A$. There exists a neighbourhood of $(m_1,\ldots,m_k)$ in $M^k$ such that every family in this neighbourhood is generating.
			
			\item Suppose that one of the following is verified
			
			\begin{enumerate}[label=\alph*),itemsep=0mm] 
				\item The subset $A^{\times}$ is open in $A$ and the inverse map is continuous on $A^{\times}$.
				
				\item The topological ring $A$ is a Huber ring\footnote{A weaker condition, but less pleasant to state is sufficient. For a subset $X$ in a ring $A$, call $X^{n\times}$ the subgroup generated by $\{x_1\ldots x_n\, |\, (x_i)\in X^n\}$. We only need a basis of neighbourhoods of $0$ which are individually stable by sum, product, and whose family of multiplications $X^{n\times}$ is final among neighbourhoods of $0$. Also note that if $A$ is a complete Huber ring, it verifies the first condition.}.
		\end{enumerate}
		We can improve the first result.  For every neighbourhood $U$ of $0$ in $A^k$, there exists a neighbourhood $W$ of $(m_1,\ldots,m_k)$ in $M^k$ and a neighbourhood $V$ of $0$ in $D$ such that $$\forall d\in V, \forall (m'_i) \in W \,\, \text{s.t.}\,\, (m'_i) \,\, \text{is generating}, \, \,\, \exists (a_i) \in U, \,\,\, d=\sum a_i m'_i.$$ In the first case, we can drop the condition "$(m'_i)$ is generating".
		\end{enumerate}
		
		These results can be rephrased as "small neighbourhoods of $0$ in $M$ have elements expressed uniformly with small coordinates in the family $(m_i)$ (resp. in all close enough families)".
	\end{lemma}

	\begin{proof}
		\begin{enumerate}[itemsep=0mm]
			\item We can reformulate the result by saying that $$p\,: \,A^k \rightarrow M, \,\,\, (a_i) \mapsto \sum_{i=1}^k a_i m_i$$ is open. This quotient map is indeed open because the initial topology on $M$ is the associated quotient topology.
			
			\item The group $\hrm{GL}_k(A)$ is open in $\hrm{M}_k(A)$ with product topology on coefficients as the inverse image of $A^{\times}$ by the determinant. Let $(U_1,\ldots,U_k)$ be neighbourhoods of the canonical basis of $A^k$ such that the open $\prod U_i$ of $A^{k^2}=\hrm{M}_k(A)$ is contained in $\hrm{GL}_k(A)$. Because $p$ is open, there exist opens $(V_1,\ldots,V_k)$ of $M$ such that $$ \forall j,\,\forall v\in V_j, \, \exists (a_{i,j})_i\in U_j, \,\,\, m=\sum a_{i,j} m_i.$$ The open $\prod V_j$ is suitable: all its families are image by $p$ of families in $\prod U_j$ which are all basis of $A^k$.
			
			\item First treat the case where $A^{\times}$ is open with continuous inverse map. There, invertible matrices form an open of $\hrm{M}_k(A)$ with continuous inverse map. Let $U$ be a neighbourhood of $0$ in $A^k$. By continuity of $$\hrm{GL}_k(A) \times A^k \rightarrow A^k, \,\,\, (M,v)\mapsto M^{-1} v,$$ there exists a neighbourhood of identity $W_{\hrm{mat}}\subset\hrm{GL}_k(A)$ and a neighbourhood $U'$ of $0$ in $A^k$ contained in the inverse image of $U$. Looking closer the previous paragraph, we already established that there exists a neighbourhood $W$ of $(m_1,\ldots,m_k)$ in $M^k$ such that $$\forall (m'_i)\in W, \, \exists \hrm{M}\in W_{\hrm{mat}}, \,\,\, m'_i=\sum \hrm{M}_{i,j} m_j.$$ Moreover, there exists a neighbourhood $V$ of $0$ in $M$ such that every element of $V$ can be expressed with coordinates in $U'$ on the family $(m_i)$. For $d\in V$, choose $a_i\in U'$ such that $d=\sum_i a_i m_i$. We have $$d=\sum_j \left(\sum_i a_i (\hrm{M}^{-1})_{i,j}\right) m_j=\sum_j (\hrm{M}^{-1}(a_i)_i)_j m_j$$ and the coordinates belong to $U$ by construction.
			
			\smallskip

			Move on to the case where $A$ is a Huber ring. Fix a ring of definition $A_0$ and an ideal of definition $I$. It is sufficient to prove the result for $U=(I^n)^k$. Let $V$ be a neighbourhood of $0$ in $M$ that the first point of this lemma furnishes for $(I^n)^k$ and $(m_i)$. Set $W=\prod_i (m_i+V)$ which is a neighbourhood of $(m_i)$ in $M^k$. For $(m'_i)\in W$ and $d\in V$, we begin by finding $(a_{i})\in (I^n)^k$ such that $$d=\sum_i a_{i} m_i= \sum_i a_i m'_i + \sum_i a_{i} (m_i-m'_i).$$ Because each $(m_i-m'_i)$ belongs to $V$, we find $(a_{j,i})\in (I^n)^{k^2}$ such that
			
			$$d=\sum_i a_{i} m'_i + \sum_{i} \left(\sum_j a_j a_{j,i}\right) m_i = \sum_i \left(a_i + \sum_j a_j a_{j,i} \right) m'_i +  \sum_{i} \left(\sum_j a_j a_{j,i}\right) (m_i-m'_i).$$ By repeating the operation, one finds two families $(a_{i,l})\in (I^n)^k$ and $(s_{i,l}) \in (I^{nl})^{k}$ satsifying $$d=\sum_i a_{i,l} m'_i + \sum_i s_{i,l} (m_i-m'_i).$$ The sequel $[\sum_i s_{i,l} (m_i-m'_i)]$ converges to zero; then the first point of this lemma for $(m'_i)$ tells that for $l$ big enough, there exists $(a_{i,\infty})\in (I^n)^k$ such that $$\sum_i s_{i,l} (m_i-m'_i)=\sum_i a_{i,\infty} m'_i.$$ This concludes.
		\end{enumerate}
	\end{proof}
	
	\begin{rem}
		Conditions of the third point are both verified by almost all the rings we use for $(\varphi,\Gamma)$-modules. Even if the proof using the first condition is less convoluted, the second condition seems easier to obtain because our rings are systematically constructed as Huber pairs even if they are not domains, nor complete (e.g. the ring $E_{\Delta}^{\hrm{sep}}$ in \cite{zabradi_equiv}).
	\end{rem}
	
We come back to the setup of a topological $\hcal{S}$-ring $R$.

	\begin{prop}\label{defi_hom_topology}
		When the condition of Proposition \ref{defi_hom_topology}'s third point is satisfied by $R$, the full subcategory $\cdetaleproj{\hcal{S}}{R}$ of $\detaleproj{\hcal{S}}{R}$ is stable by internal $\hrm{Hom}$.
		
		Thus, the symmetric monoidal structure of Proposition \ref{tensor_on_top} is closed.
	\end{prop}
	\begin{proof}
		To prove that $\underline{\hrm{Hom}}_R(D_1,D_2)$ is topological étale projective (étale projective is already established), we begin\footnote{It seems possible to explicitely determine $\hrm{M}_{s,\underline{\hrm{Hom}}_R(D_1,D_2)}$ but the subtleties would stay the same, however hidden under proliferation of indexes.} by proving that the initial topology on $\underline{\hrm{Hom}}_R(D_1,D_2)$ is the pointwise convergence topology for the initial topology on $D_2$ (i.e. induced by the product topology on map from $D_1$ to $D_2$). Let $D_i\oplus D'_i=R^{k_i}$ be presentations of the $D_i$'s. The $R$-module $\hrm{Hom}_R(D_1,D_2)$ is direct summand of the free module $\hrm{Hom}_R(R^{k_1},R^{k_2})$ as $$\hrm{Hom}_R(D_1,D_2) = \{f\in \hrm{Hom}_R(R^{k_1},R^{k_2}) \, |\, D'_1 \subset \hrm{Ker}(f) \,\, \text{ and } \,\, \hrm{Im}(f)\subset D_2\}.$$ The initial topology on $\hrm{Hom}_R(D_1,D_2)$ is henceforth obtained from this presentation. It happens that the topology on the free module $\hrm{Hom}_R(R^{k_1},R^{k_2})\cong R^{k_1 k_2}$ is the pointwise convergence; in addition, the initial topology on $D_2$ is induced from $R^{k_2}$ so the initial topology on $\hrm{Hom}_R(D_1,D_2)$ is the pointwise convergence topology. To show that $\underline{\hrm{Hom}}_R(D_1,D_2)$ is topological, it remains to show that $$\forall d\in D_1, \,\,\,\hcal{S}\times \hrm{Hom}_R(D_1,D_2) \rightarrow D_2, \,\,\, (s,f)\mapsto \varphi_{s,\underline{\hrm{Hom}}_R(D_1,D_2)}(f)(d) \, \text{ is continous}.$$ Denote by $(d_i)$ the components in $D_1$ of the canonical basis of $R^{k_1}$. Fixing an expression $d=\sum_i r_i \varphi_{s,D_1}(d_i)$, we already computed that $$\varphi_{s,\underline{\hrm{Hom}}_R(D_1,D_2)}(f)\left(\sum_i r_i \varphi_{s,D_1}(d_i)\right)=\sum_i r_i \varphi_{s,D_2}(f(d_i)).$$ 
		
		Let $(s,f)$ be a pair and $U_2$ an open neighbourhood of $(s,f)$'s image in $D_2$. Let $U$ be a neighbourhood of $0$ in $R^{k_1}$ and $(\hcal{W}_1\times W_2)$ a open neighbourhood $(s,f)$ such that $$\forall (t_i,s',g) \in U \times \hcal{W}\times W_2, \,\,\,\sum_i (r_i+t_i) \varphi_{s',D_2}(g(d_i)) \in
		U_2.$$ This is possible because $D_2$ belongs to $\cdetaleproj{\hcal{S}}{R}$ and the initial topology on the internal $\hrm{Hom}$ is the pointwise topology. Thanks to $D_2$'s étaleness, the family $(\varphi_{s,D_1}(d_i))_i$ is generating and the third point of Lemma \ref{huber_ouvert_uniforme} constructs special neighbourhood $V$ of $0$ in $D_2$ and a neighbourhood $W$ of $(\varphi_{s,D_1}(d_i))_i$ in $D^k$. Finally, fix $\hcal{W}'_1\subseteq \hcal{W}_1$ such that $$\forall s \in \hcal{W}'_1, \,\,\, \sum_i r_i\left[\varphi_{s,D_1}(d_i)-\varphi_{s',D_1}(d_i)\right] \in V \,\,\, \text{and} \,\,\, \left(\varphi_{s',D_1}(d_i)\right)_i \in W,$$ thanks to the continuity of $\hcal{S}$-action. For every $(s',g)\in \hcal{W}'_1\times W_2$, we have $$d=\sum_i r_i \varphi_{s',D_1}(d_i)+\sum_i r_i \left[\varphi_{s,D_1}(d_i)-\varphi_{s',D_1}(d_i)\right]$$ whose second half can be re-expressed as $\left(\sum_i t_i \varphi_{s',D_1}(d_i)\right)$ for some $(t_i)_i \in U$. Thus, $$\varphi_{s',\underline{\hrm{Hom}}_R(D_1,D_2)}(g)(d)=\sum_i (r_i+t_i)\varphi_{s',D_2}(g(d_i))$$ which belongs to $U_2$.
	\end{proof}
	
	\medskip 
	We move on to dévissage setups. Already thinking about Fontaine-type functor, we remark that preserving continuity happens to be subtle. In Fontaine's original setup, the exact sequences of $\Oehat$-modules or $\Oe$-modules like
	
	$$0\rightarrow pD \rightarrow D \rightarrow \quot{D}{pD}\rightarrow 0$$ used to apply dévissage on $(\varphi,\Gamma)$-modules are strict for both the $p$-adic topology and for the initial topology coming from the weak topoogy. However for this second topology, it relies on the structure theorem for finite type modules over a discrete valuation ring. The strategy in the general case is less obvious. Worse, as we said in remark \ref{rem_top_inv}, taking invariants doesn't respect the initial topology if we don't have a strucure theorem on modules (or specific noetherian Tate rings).
	
	Fix a dévissage setup $(R,r)$ and a structure of $\hcal{S}$-ring on $R$ such that $\forall s\in \hcal{S}, \,\,\, \varphi_s(r)R= rR$. The restrictions-corestrictions of $\varphi_s$ make $R/r$ belong to $\hcal{S}\text{-}\hrm{Ring}$.

	\begin{defi}\label{defi_detaledvproj_cont}
		The category $\cdetaledvproj{\hcal{S}}{R}{r}$ is the intersection of the full subcategories $\cdetale{\hcal{S}}{R}$ and $\detaledvproj{\hcal{S}}{R}{r}$.
	\end{defi}
	
		\begin{prop}\label{defi_hom_topo_dv}
		Suppose that the condition of Proposition \ref{defi_hom_topology}'s third point is verified by $R$. Suppose in addition that $\hrm{K}_0(\sfrac{R}{r})=\Z$. Then, the subcategory $\cdetaledvproj{\hcal{S}}{R}{r}$ is stable by internal $\hrm{Hom}$.
	\end{prop}
	\begin{proof}
		First consider that both $R^{\times}$ being open with continuous inverse map and $R$ being Huber imply the analogous condition on each $R/r^n$. Then decompose $\hrm{Hom}_R(D_1,D_2)$ as in the proof of Proposition \ref{monoidal_on_detaledv}'s second point and use Proposition \ref{defi_hom_topology} on each term.
	\end{proof}
	
	\begin{rem}
		As $R$ is $r$-adically complete, $R^{\times}$ is automatically open with continuous inverse map for the $r$-adic topology. Even for coarser topologies, completeness greately simplifies the proof of condition $(a)$.
	\end{rem}
	
	\medskip
	\subsection{Preservation by operations}
	
	This section will end by studying how these topological constructions interact with our previous operations and variations.
	
	 Let $a\,: \, R \rightarrow T$ be a morphism of topological $\hcal{S}$-rings.
	
	\begin{prop}\label{cdetaleproj_ext}
		The functor $\hrm{Ex}$ defined in section \ref{ex_subsection} sends $\cdetale{\hcal{S}}{R}$ to $\cdetale{\hcal{S}}{T}$. It also sends $\cdetaleproj{\hcal{S}}{R}$ to $\cdetaleproj{\hcal{S}}{T}$ and is strong symmetric monoidal.
		
		When the condition of Proposition \ref{defi_hom_topology}'s third point is satisfied by $R$, the image of $\hrm{Ex}$ is closed under internal $\hrm{Hom}$ and $\hrm{Ex}$ naturally commutes to internal $\hrm{Hom}$.
	\end{prop}
	\begin{proof}
		The first point of Proposition \ref{constr_ext_etale} already establishes that $\hrm{Ex}$ sends $\cdetale{\hcal{S}}{R}$ to $\detale{\hcal{S}}{T}$ and $\cdetaleproj{\hcal{S}}{R}$ to $\detaleproj{\hcal{S}}{T}$. Only topological conditions remain.
		
		Fix $D\in \cdetale{\hcal{S}}{R}$ and a quotient map $\pi \,: \, R^k \rightarrow D$. The familiy $(1\otimes \pi(e_i))$ is generating in $T\otimes_R D$, hence Lemma \ref{continuity_etale} reduces continuity to checking that each $$\hcal{S} \rightarrow T\otimes_R D, \,\,\, s\mapsto 1\otimes \varphi_{s,D}(\pi(e_i))$$ is continuous. These maps decompose as $$\hcal{S} \xrightarrow{s\mapsto \varphi_{s,D}(\pi(e_i))} D \xrightarrow{d\mapsto 1\otimes d} T\otimes_R D$$ where the first map is continuous because $D$ is topological. Continuity of the second can be checked after pre-composition by $\pi$; we look at the following commutative diagram:
		\begin{center}
			\begin{tikzcd}
				R^k \ar[r,"\prod a"] \ar[d,"\pi"'] & T^k \ar[d,"\hrm{Id}_T \otimes \pi"'] \\
				D \ar[r] & T\otimes_R D
			\end{tikzcd}
		\end{center} whose path via the upper-right corner emphasises the continuity.
		
		Results on the (closed) symmetric monoidal structure are deduced from the results on $\hrm{Mod}^{\hrm{\acute{e}t}}_{\hrm{proj}}$.
	\end{proof}
	
	\medskip
	
	Add the datum of a normal submonoid $\hcal{S}'$ of $\hcal{S}$ and endow the quotient monoid $\hcal{S}/\hcal{S}'$ with the quotient topology. It becomes a topological monoid. The ring $R^{\hcal{S'}}$ endowed with the induced topology from $R$ belongs to $\hcal{S}/\hcal{S'}\text{-}\hrm{Ring}$. The inclusion $R^{\hcal{S}'}\subset R$ is a morphism of topological $\hcal{S}$-rings.
	
	\begin{prop}\label{inv_cdetale}
		Let $D$ belongs to $\cdetaleproj{\hcal{S}}{R}$. Suppose that $R^{\hcal{S'}}\subset R$ is faithfully flat and that the comparison morphism $$R\otimes_{R^{\hcal{S'}}} \hrm{Inv}(D) \rightarrow D$$ is an isomorphism. Then $\hrm{Inv}(D)$ is an object of $\cdetaleproj{\sfrac{\hcal{S}}{\hcal{S}'}}{R^{\hcal{S}'}}$.
		
		Suppose in addition that $R$ is complete Haussdorf with a countable basis of neighborhoods of zero. Suppose that $R^{\hcal{S}'}$ is complete noetherian, has a sequence of units converging to zero, verifies that the topological nilpotent elements form a neighborhood of zero. Then, the same result is correct for $\cdetale{\hcal{S}}{R}$.
	\end{prop}
	\begin{proof}
		Thanks to Proposition \ref{inv_detaleproj}, we already know that $\hrm{Inv}(D)$ belongs to $\detaleproj{\sfrac{\hcal{S}}{\hcal{S}'}}{R^{\hcal{S}'}}$. Fix a presentation of $\hrm{Inv}(D)$ as direct summand of a finite free $R^{\hcal{S}'}$-module. If we base change this presentation to $R$, the comparison isomorphism identifies it as a presentation of $D$ as direct summand of a finite free $R$-module. Therefore, the induced topology on $\hrm{Inv}(D)$ from the initial topology on $D$ is the initial topology.  The continuity condition follows directly from the continuity on $D$.
		
			Let's consider the other case. The same strategy reduces the claim to showing  that the initial topology $\hrm{Inv}(D)$ is induced from the initial topology on $D$. As the topology on $R$ is Haussdorf with countable basis of neighborhoods of zero, the same is true for $R^{\hcal{S}'}$ which is equipped with the induced topology. We can therefore apply \cite[Theorem 2.12]{henkel_open_map} to $R^{\hcal{S}'}$-module and we only need to show that the induced topology is complete Haussdorf with a countable basis of neighborhoods of zero to conclude. As the initial topology on $D$ is Haussdorf with a countable basis of neighborhoods of zero, the same is true for the induced topology on $\hrm{Inv}(D)$. If $(d_k)_{k\geq 0}$ is a Cauchy sequence in $\hrm{Inv}(D)$ for the induced topology, then it is in $D$. This sequence converges into $D$. We concludes by noticing that the continuity of the action of $\hcal{S}'$ on $D$ implies that $\hrm{Inv}(D)$ is closed in $D$.
	\end{proof}

	\begin{rem}\label{rem_top_inv}
		It seems difficult to obtain such proposition in general for $\cdetale{\hcal{S}}{R}$ because, for a finitely presented $R^{\hcal{S}'}$-module $D$, nothing garantuee that the initial topology on $\hrm{Inv}(D)$ coincides with the induced topology from the initial topology on $D$. Hence checking continuity after base change fails.
	\end{rem}
	\medskip

Studying the interaction between topology and coinduction seems wild\footnote{We could define a continous version of coinduction, which would only contain continuous function from $\hcal{T}$ to the topological space (resp. ring, resp. module). This would have the desired adjunction property and need the space (resp. the ring) to have a continous action of $\hcal{S}$ so as to have a chance to get more elements that $\Z$-valued functions. In suitable settings, we could endow the coinduction with compact-open topology and unravel our propositions. However, such a general setting didn't occur in my work.}. To keep things easy, we impose conditions that are verified in some concrete setup, though I don't claim that they are optimal. Suppose that $\hcal{S}$ is a submonoid of a bigger topological monoid $\hcal{T}$, with induced topology. The coinduction of a ring will be endowed with the limit topology seeing coinduction as a big equaliser.

\begin{lemma}
	\begin{enumerate}[itemsep=0mm]
		\item The limit topology on the coinduction $\coindu{\hcal{S}}{\hcal{T}}{R}$ is the pointwise convergence on functions.
		
		\item Suppose that for every $t\in \hcal{T}$ and $\hcal{U}$ open in $\hcal{S}$, the set $(\hcal{U}t)$ is open\footnote{This is not implied by $\hcal{S}$ being open: it is wrong for $\hrm{M}_n(\R)$ seen as a submonoid of itself.} (in particular $\hcal{S}$ is open in $\hcal{T}$), then $\coindu{\hcal{S}}{\hcal{T}}{R}$ with limit topology is a topological $\hcal{T}$-ring.
	\end{enumerate}
\end{lemma}
\begin{proof}
	\begin{enumerate}[itemsep=0mm]
		\item The limit is indexed by a poset with minimal vertices indexed by $\hcal{T}$. Because the transition maps are expressed a some of the continuous maps $\varphi_s$, the limit topology is the induced topology from the product over $\hcal{T}$. It is the pointwise convergence of functions.
		
		\item Thanks to the first point, we only need to prove that for every $t_0 \in \hcal{T}$, the map $$\hcal{T}\times \coindu{\hcal{S}}{\hcal{T}}{R} \rightarrow R, \,\,\, (t,f)\mapsto (t\cdot f)(t_0)=f(t_0 t)$$ is continous. Let $(t_1,f_1)$ belongs to the source and $W$ be a neighbourhood of $f_1(t_0 t_1)$ in $R$. By continuity of the $\hcal{S}$-action on $R$, there exists an neighbourhood $(\hcal{U}\times V)$ of $(e_{\hcal{S}},f_1(t_0t_1))$ in $(\hcal{S}\times R)$ whose image lies in $W$. Thanks to hypothesis $\hcal{U}t_0t_1$ is a neighbourdhood of $t_0t_1$ in $\hcal{T}$, thus continuity of $\hcal{T}$'s law produces a neighbourhood $\hcal{V}$ of $t_1$ in $\hcal{T}$ such that $t_0 \hcal{V} \subset \hcal{U}t_0t_1$. Then, $$\forall (t,f) \in \hcal{V}\times \{f\,\, |\,\, f(t_0t_1) \in V\}, \, \exists s\in \hcal{S}, \,\,\,f(t_0 t')=\varphi_{s}(f(t_0t_1))$$ hence belongs to $W$.
	\end{enumerate}
\end{proof}

\begin{prop}
	Suppose that for every $t\in \hcal{T}$ and $\hcal{U}$ open in $\hcal{S}$, the set $(\hcal{U}t)$ is open and that $\hcal{S}$ is of finite subtle index in $\hcal{T}$. Then, the essential image of $$\hrm{Coindu}_{\hcal{S}}^{\hcal{T}}\,: \, \cdetaleproj{\hcal{S}}{R} \rightarrow \dmod{\hcal{T}}{\coindu{\hcal{S}}{\hcal{T}}{R}}$$ contains $\cdetaleproj{\hcal{T}}{\coindu{\hcal{S}}{\hcal{T}}{R}}$.
\end{prop}
\begin{proof}
	The topology on $\coindu{\hcal{S}}{\hcal{T}}{R}$ being the pointwise convergence topology, the evaluation at the identity of $\hcal{T}$ is a continuous $\hcal{S}$-ring morphism. Then the proof is the same as in \ref{coindu_im_ess}, using the isomorphism of Lemma \ref{induction_anneau_tensor_2} and the fact that $\hrm{Ex}$ preserves continuity which is Proposition \ref{cdetaleproj_ext}. 
\end{proof}
	
	\medskip
	
	We move on to dévissage setups. Fix again a dévissage setup $(R,r)$ and a structure of $\hcal{S}$-ring on $R$ such that $\forall s\in \hcal{S},\, \varphi_s(r) R=rR$. Fix  a morphism $a\,: \, R\rightarrow T$ of topological $\hcal{S}$-rings and suppose that $(T,a(r))$ is a dévissage setup.

	\begin{prop}\label{ex_cdetaledvproj}
	The functor $\hrm{Ex}$ sends $\cdetaledvproj{\hcal{S}}{R}{r}$ to the full subcategory $\cdetaledvproj{\hcal{S}}{T}{a(r)}$, it strong symmetric monoidal.
	
	If the conditions of Proposition \ref{defi_hom_topo_dv} are verified by $R$, its image is stable by internal $\hrm{Hom}$ and naturally commutes to the formation of internal $\hrm{Hom}$.
	\end{prop}
	\begin{proof}
	Combine Propositions \ref{dvproj_ext} and \ref{cdetaleproj_ext}.
	\end{proof}
	
	\vspace{0.5 cm}
	
	We add the datum of a normal submonoid $\hcal{S}'$ and suppose that $r\in R^{\hcal{S}'}$. If we are willing to adapt Corollary \ref{inv_cdvdetaleproj_dévissage} to our topological setting, it turns out that we need to be a little more subtle about the categories we consider\footnote{see remark \ref{rem_h1} to understand how $\hrm{H}^1_{\hrm{cont}}(\hcal{S}',\sfrac{R}{r})$ might not vanish for the only reasonable topology making the $\hcal{S}$-action continuous.}. Precisely, we want to allow the action of the normal submonoid $\hcal{S}'$ to be continuous for another better behaved topology, so as to analyse the comparison morphism with respect to this topology, then transfer continuity with respect to another. In addition, as we did not obtain a preservation of continuity for étale modules by $\hrm{Inv}$, we put ourselves in a setup where we can use Theorem \ref{complet_pf_3}.
	
	\begin{defi}
		Let $(R,r)$ be a dévissage setup and $\mathscr{T}'$ be a ring topology on $R$. We say that $\mathscr{T}'$ has \textit{good $r$-dévissages properties} if for all finitely presented $R$-module $D$ with finite projective $(r,\mu)$-dévissage, the initial topology on the $R$-module $D$ induces the initial topology on the $R$-module $rD$ and on the $R/r$-module $D[r]$.
	\end{defi}
	
	\begin{defi}\label{def_gooddvprop}
		Let $(R,r)$ be a dévissage setup. Let $\hcal{S}$ be a topological monoid and add a structure of $\hcal{S}$-ring on $R$. Fix a topology $\mathscr{T}$ on $R$ making the $\hcal{S}$-ring structure a topological $\hcal{S}$-ring structure. Let $\hcal{S}'$ be a normal submonoid (with induced topology) and $\mathscr{T}'$ be a topology on $R$ enhancing it to a topological $\hcal{S}'$-ring. We suppose that $r\in R^{\hcal{S}'}$ as before Theorem \ref{inv_cdvdetaleproj}.
		
		We define $\cdetaledvproj{\hcal{S},\, \hcal{S}'}{R}{r}$ as the full subcategory of $\cdetaledvproj{\hcal{S}}{R}{r}$ for the topology $\mathscr{T}$ on $R$ formed by objects $D$ such that the forgetful functor to $\dmod{\hcal{S}'}{R}$ sends $D$ into $\cdetale{\hcal{S}'}{R}$ for the topology $\mathscr{T}'$ on $R$.
	\end{defi}

	\begin{prop}\label{inv_cdvdetaleproj}
		In the setup above, suppose in addition that:
		
		\begin{itemize}[itemsep=0mm]
			\item The map $R^{\hcal{S}'}/r \rightarrow R/r$ is faithfully flat
			\item The map $R^{\hcal{S}'}/r \hookrightarrow \left(R/r\right)^{\hcal{S}'}$ is an isomorphism\footnote{This is true for instance as soon as $\hrm{H}^1_{\hrm{cont}}(\hcal{S}',R)$ is $r$-torsion-free for any topology coarser than the $r$-adic. Here we already see that the properties for the topology $\mathscr{T}$ comes without any cost.}.
			
			\item We have $\hrm{K}_0(\sfrac{R^{\hcal{S}'}}{r})=\Z$.
			
		\end{itemize}let $D$ be an object of $\cdetaledvproj{\hcal{S},\, \hcal{S}'}{R}{r}$. If the comparison morphism $$ R\otimes_{R^{\hcal{S}'}} \hrm{Inv}(D)\rightarrow D$$ is an isomorphism, then $\hrm{Inv}(D)$ belongs to $\cdetaledvproj{\sfrac{\hcal{S}}{\hcal{S}'}}{R^{\hcal{S}'}}{r}$.
	\end{prop}
	\begin{proof}
		Proposition \ref{inv_dvdetaleproj} tells that $\hrm{Inv}(D)$ belongs to $\detaledvproj{\sfrac{\hcal{S}}{\hcal{S}'}}{R^{\hcal{S}'}}{r}$. Continuity remains to prove. 
		
		Suppose proved that the initial topology on $\hrm{Inv}(D)$ is induced by the inclusion into $D$ with initial topology. Take $(d_i)$ a generating family of $\hrm{Inv}(D)$. Combining Lemma \ref{continuity_etale} and this assumption, continuity of the action can be checked on each $$\quot{\hcal{S}}{\hcal{S}'} \xrightarrow{s\hcal{S}'\mapsto \varphi_s(d_i)} \hrm{Inv}(D) \rightarrow D,$$ which are continuous thanks to continuity on $D$.
		
		It happens that $\hrm{K}_0(\sfrac{R^{\hcal{S}'}}{r})=\Z$ allows to prove the topological assumption. Thanks to Theorem \ref{complet_pf_3}, we have a decomposition $$\hrm{Inv}(D)=D'_{\infty}\oplus \bigoplus_{1\leq n \leq N} D'_n$$ which translates thanks to the isomorphism of comparison into the identification of $\hrm{Inv}(D) \subset D$ as the direct sum of $D'_{\infty} \subset R\otimes_{R^{\hcal{S}'}} D'_{\infty}$ and each $D'_n \subset (R/r^n)\otimes_{R^{\hcal{S}'}/r^n} D'_n$. The inclusion $R^{\hcal{S}'}\subseteq R$ induces by definition the topology on $R^{\hcal{S}'}$. Moreover, on $D'_n$ (resp. $R\otimes_{R^{\hcal{S}'}} D'_n$) the initial topology as finite type $R^{\hcal{S}'}$-module (resp $R$-module) and finite type $R^{\hcal{S}'}/r^n$-module (resp $R/r^n$-module) coincide\footnote{Topology on $R/r^n$ is the quotient topology and a generating family as an $R/r^n$-module is also generating as an $R$-module.}. This show that the induced topology on $\hrm{Inv}(D)$ is the initial one.
	\end{proof}

\begin{theo}\label{inv_cdvdetaleproj_dévissage}
	In the setup of Definition \ref{def_gooddvprop}, suppose that:
	
	\begin{itemize}[itemsep=0mm] 
		\item The map $R^{\hcal{S}'}/r \rightarrow R/r$ is faithfully flat.
		
		\item We have $\hrm{K}_0(\sfrac{R^{\hcal{S}'}}{r})=\Z$. 
		
		\item The topology $\mathscr{T}'$ is coarser that the $r$-adic topology and has good $r$-dévissages properties.
		
		\item We have $\hrm{H}^1_{\hrm{cont}}(\hcal{S}',\sfrac{R}{r})=\{0\}$ for $\mathscr{T}'$.
		
		\item For every $D$ in $\cdetaleproj{\hcal{S},\, \hcal{S}'}{\sfrac{R}{r}}$ the comparison morphism $$ R\, \mathop{\otimes}_{R^{\hcal{S}'}} \hrm{Inv}(D)\rightarrow D$$ is an isomorphism. 
	\end{itemize}
	
	\noindent Then, the comparison morphism is an isomorphism for every object of $\cdetaledvproj{\hcal{S},\,\hcal{S}'}{R}{r}$ and the functor $\hrm{Inv}$ sends $\cdetaledvproj{\hcal{S},\,\hcal{S}'}{R}{r}$ to $\cdetaledvproj{\sfrac{\hcal{S}}{\hcal{S}'}}{R^{\hcal{S}'}}{r}$ and is strong symmetric monoidal. It is closed monoidal as soon as $R^{\times}$ is open with continuous inverse map or if $R$ is a Huber ring with ideal of definition generated by elements of $R^{\hcal{S}'}$.
\end{theo}
\begin{proof}
	Follow closely Theorem \ref{inv_dvdetaleproj_dévissage} to prove that for every $D$ in $\cdetaledvproj{\hcal{S},\,\hcal{S}'}{R}{r}$, the comparison morphism is an isomorphism. Remark that the step 1 works for continuous cohomology: because $\mathscr{T}'$ is a ring topology, coarser than the $r$-adic topology, each $r$-adically complete and separated finite type module $D$ satisfy $$D\cong \lim \limits_{\leftarrow} \quot{D}{r^n D}$$ in the category of topological groups, equipping $D$ with the initial topology and each term of the limit with quotient topology. Also remark that only continuous $\hcal{S}'$-cohomology is needed for step 2, and that the good $r$-dévissages properties implies that all considered exact sequences in steps 3 and 4 are strict exact sequence of topological abelian groups. Then use Proposition \ref{inv_cdvdetaleproj} for the topological $\hcal{S}$-ring $R$ with $\mathscr{T}$. The fact that it is strong symmetric monoidal, closed in some case, doesn't require any new idea.
\end{proof}

\begin{ex}
	Although we stated things for an abstract topology $\mathscr{T}'$, one major example is to use the $r$-adic topology as $\mathscr{T}'$ for which the initial topology on any finitely generated module is the $r$-adic one. For this topology, any finitely generated $R$-module $D$ induces the $r$-adic topology on $rD$. For this topology, for any $R$-module $D$ with bounded $r$-torsion, the $p$-adic topology induces the discrete topology on $D[r]$. Thus, the $r$-adic topology has good $r$-dévissages properties thanks to Theorem \ref{complet_pf}.
	
	Another setting would make $\mathscr{T}'$ have good $r$-dévissages properties: if $\hrm{K}_0(\sfrac{R}{r})=\Z$ and if the induced topology on $rR$ from $R$ is the one given by the isomorphism $R\xrightarrow{r\times} rR$. Even though this sometimes shows that $\mathscr{T}$ also has good $r$-dévissages properties, considering $\mathscr{T}'$ could still be crucial to guarantee the cohomological condition.
\end{ex}

	\vspace{1 cm}
	
	\section{Fontaine's equivalence for $\qp$ revisited}\label{section_fontaine}

In this section, we aim to obtain develop the occurences of "Le cas des représentations de $p$-torsion s'en déduit par dévissage
et le cas général par passage à la limite" in \cite[Proof of 1.2.6]{Fontaine_equiv} using our vocabulary. This becomes quite a long list of overall easy properties pour verified, but more serious applications are to be found in \cite{nataniel_fontaine}.

We introduce our tool rings. Consider the valued field $\mathbb{C}_p:=\widehat{\overline{\qp}}$ with continuous action of $\cg_{\qp}$ for the valuation topology. On the perfectoid field $\widehat{\qp(\mu_{p^{\infty}})}$, this action factorises by $\ch_{\qp}:=\hrm{Gal}(\overline{\qp}|\qp(\mu_{p^{\infty}}))$ and the quotient $\sfrac{\cg_{\qp}}{\ch_{\qp}}=\hrm{Gal}(\qp(\mu_{p^{\infty}})|\qp)$ is topologically isomorphic to $\zp^{\times}$. The tilt $\widehat{\qp(\mu_{p^{\infty}})}^{\flat}$ has absolute Galois group\footnote{See \cite[Theorem 2.3]{scholze_perf}.} isomorphic to $\ch_{\qp}$ and $\mathbb{C}_p^{\flat}$ is the completion of its algebraic closure. The action of $\cg_{\qp}$ on $\mathbb{C}_p^{\flat}$ is continuous for the valuation topology and its restriction to $\ch_{\qp}$ identifies with the action of the absolute Galois group of $\widehat{\qp(\mu_{p^{\infty}})}^{\flat}$.

Define $E^+:=\mathbb{F}_p \llbracket X \rrbracket$,  $E:=\mathbb{F}_p (\!( X )\!)$ and call $X$-adic topology on $E$ the ring topology for which $X^n E^+$ is a fundamental system of neighbourhoods of $0$. Let $(\zeta_{p^n})_{n\geq 0}$ be a compatible sequence of $(p^n)$-th roots of unity in $\overline{\qp}$. The $\mathbb{F}_p$-algebra morphism $$ \iota \,: \, E \rightarrow \widehat{\qp(\mu_{p^{\infty}})}^{\flat}, \,\,\, X\mapsto (\zeta_{p^n}-1)_{n\geq 0},$$ continuous for $X$-adic topology on $E$, is injective and an homeomorphism on its image. Fix an extension $$\iota \,: \, E^{\sep}\rightarrow \mathbb{C}_p^{\flat}$$ of this injection. Its image is stable by the $\cg_{\qp}$-action; we can equip $E^{\sep}$ with its subspace topology from $\mathbb{C}_p^{\flat}$ and transfer a $\cg_{\qp}$-action continuous for the $X$-adic topology. The deduced $\ch_{\qp}$-action is continuous for the discrete topology and identifies $\ch_{\qp}$ to the absolute Galois group of $E$. The image of $E^{\sep}$ is also stable by the $p^{\hrm{th}}$-power Frobenius $\varphi$ which is continous for both $X$-adic and discrete topologies and stabilises $E$.

In order to lift to characteristic zero, consider the Witt vector ring $\hrm{W}(\mathbb{C}_p^{\flat})$ equipped with the product topology of the valuation topology on tilts, and with $\zp$-linear continous action of $\cg_{\qp}$. It is $p$-adically complete and separated, strict henselian. Define $\Oeplus:=\zp\llbracket X \rrbracket$, $\Oe:=(\Oeplus[X^{-1}])^{\wedge p}$ and call natural topology the ring topology having $$\left\{p^n \Oe + X^m \Oeplus \, \bigg|\, n,m\geq 0\right\}$$ as basis of neighbourhoods of $0$. The $\zp$-algebra morphism $$j \,: \,\Oe \rightarrow \hrm{W}(\mathbb{C}_p^{\flat}), \,\,\,X \mapsto [(\zeta_{p^n})]-1$$ continuous for the natural topology, is injective and an homeomorphism on its image. Define $\Oehat$ to be the $p$-adic completion of a strict henselization of $\Oe$ and extend the previous morphism to $\Oehat$. Its image is stable under $\varphi$ and the $\cg_{\qp}$ action on the Witt vectors; we call natural topology on $\Oehat$ its subspace topology from $\hrm{W}(\mathbb{C}_p^{\flat})$ and transfer a continuous $\cg_{\qp}$-action, and a continous Frobenius. The $\ch_{\qp}$-action and $\varphi$ are continuous for the $p$-adic topology. Both $\varphi$ and the $\cg_{\qp}$-action stabilise $\Oe$ and the second one even factorises through $\ch_{\qp}$.

To summarise, these heavy constructions produce three rings. 
\begin{itemize}[itemsep=0mm]
	\item The ring $\zp$ with $p$-adic topology and trivial action of $\cg_{\qp}$. It can also be seen as a topological $(\varphi^{\N}\times \cg_{\qp})$-ring with trivial action.
	
	\item The ring $\Oe$ with natural topology and structure of topological $(\varphi^{\N}\times \cg_{\qp})$-ring. As the action of $\ch_{\qp}$ is trivial, the action factorises through the quotient $(\varphi^{\N}\times \zptimes)$ giving a topological $(\varphi^{\N}\times \zptimes)$-ring.
	
	\item The ring $\Oehat$ with induced topology from $\hrm{W}(\mathbb{C}_p^{\flat})$ and structure of topological $(\varphi^{\N}\times \cg_{\qp})$-ring. The $p$-adic topology also equips it with a structure of topological $(\varphi^{\mathbb{N}}\times \ch_{\qp})$-ring.
\end{itemize}

\medskip

First, we construct Fontaine's functor in \cite[\S 1.2.2]{Fontaine_equiv} in our own language. It goes from continuous finite type $\zp$-representations of $\cg_{\qp}$ to $\dmod{\varphi^{\N}\times \zp^{\times}}{\Oe}$. The category of $\cg_{\qp}$-representations over $\zp$ is $\dmod{\cg_{\qp}}{\zp}$ with $\cg_{\qp}$ acting trivially. First, consider the following facts where all $\hcal{S}$-ring structures are the one constructed above:

\begin{enumerate}[itemsep=0mm]
	\item  The inclusion $\zp\subset \Oehat$ is a morphism of $(\varphi^{\N}\times \cg_{\qp})$-rings. 
	
	\item The monoid $\ch_{\qp}$ is a normal submonoid of $(\varphi^{\N}\times \cg_{\qp})$ and the quotient is isomorphic to $(\varphi^{\N}\times \zp^{\times})$.
	
	\item The inclusion $\Oe\subseteq \Oehat^{\ch_{\qp}}$ is an equality.
\end{enumerate}
\noindent These properties suffice to apply our formalism and define following composition of functors:

$$\mathbb{D}\,: \,\dmod{\cg_{\qp}}{\zp} \xrightarrow[]{\hrm{triv}} \dmod{\varphi^{\N}\times \cg_{\qp}}{\zp} \xrightarrow[]{\hrm{Ex}} \dmod{\varphi^{\N}\times \cg_{\qp}}{\Oehat} \xrightarrow[]{\hrm{Inv}} \dmod{\varphi^{\N}\times \zp^{\times}}{\Oe},$$ where $\hrm{triv}$ extend the action by seing $\cg_{\qp}$ as quotient of $(\varphi^{\mathbb{N}}\times \cg_{\qp})$. 
	
However, Fontaine considers only continous finite type representations. Because $\cg_{\qp}$ is a group and $\zp$ is noetherian, these automatically lie in $\detale{\cg_{\qp}}{\zp}$. Because $p\zp$ is maximal, the dévissage subquotients of such representations are automatically projective over $\mathbb{F}_p$ so the considered representations\footnote{For this paragraph to make sense, we first need the two first of the upcoming conditions to be proved.} lie in $\detaledvproj{\cg_{\qp}}{\zp}{p}$; Fontaine's category of representations is equivalent to $\cdetaledvproj{\cg_{\qp}}{\zp}{p}$. We have the following list of properties:

\begin{enumerate}[itemsep=0mm]
	\item All three rings are $p$-torsion-free and $p$-adically complete and separated.

	\item The element $p$ is invariant for all considered $\hcal{S}$-ring structures.

	\item The functor $\hrm{triv}$ send $\cdetaledvproj{\cg_{\qp}}{\zp}{p}$ to $\cdetaledvproj{\varphi^{\N}\times \cg_{\qp}}{\zp}{p}$.

	\item The inclusion $\zp \subset \Oehat$ is continuous for induced topology on $\Oehat$ and $(\varphi^{\N}\times \cg_{\qp})$-equivariant.

	\item The inclusion $\zp \subset \Oehat$ is continuous for the $p$-adic topology on $\Oehat$ and $\ch_{\qp}$-equivariant.

	\item The induced topology on $\Oe$ from $\Oehat$ is the one we constructed.

	\item The inclusion $\Oe \subset \Oehat$ is faithfully flat and $p$ is irreducible in $\Oe$.
	
	\item We have $\hrm{K}_0(E)=\Z$.

	\item The $p$-adic topology on $\Oehat$ is a linear topology (coarser thant the $p$-adic one), and has good $p$-dévissages property.

	\item The group $\hrm{H}^1_{\hrm{cont}}(\ch_{\qp},E^{\sep})$ vanishes for the discrete topology on $E^{\sep}$, which is the quotient topology on $\Oehat/p\Oehat$ coming from $p$-adic topology.

	\item For every $D$ in $\cdetaleproj{\varphi^{\N}\times \cg_{\qp},\,\ch_{\qp}}{E^{\sep}}$, the comparison morphism $$E^{\sep} \otimes_E \hrm{Inv}(D) \rightarrow D$$ is an isomorphism.
\end{enumerate} All of them are consequences of the constructions of the rings, except for the two last ones that are implied by Hilbert 90. With these properties, the Proposition \ref{ex_cdetaledvproj} for both topologies on $\Oehat$ (thanks to points 4 and 5) and corollary \ref{inv_cdvdetaleproj_dévissage} for the submonoid $\ch_{\qp}$ acting on $\Oehat$ allows to restrict $\mathbb{D}$ as

\begin{center}
	\begin{tikzcd}
	\mathbb{D} \,: \, &\cdetaledvproj{\cg_{\qp}}{\zp}{p} \ar[r,"\hrm{triv}"] & \cdetaledvproj{\varphi^{\N}\times \cg_{\qp}}{\zp}{p} \ar[d,"\hrm{Ex}"] \\
	& \cdetaledvproj{\varphi^{\N}\times \zp^{\times}}{\Oe}{p} & \cdetaledvproj{\varphi^{\N}\times \cg_{\qp},\, \ch_{\qp}}{\Oehat}{p} \ar[l,"\hrm{Inv}"]
	\end{tikzcd}
\end{center} where the topology $\mathscr{T}'$ is always the $p$-adic topology.

\bigskip

In the other direction, we first use example \ref{ex_s_discrete} to see that the full subcategories $\cdetaledvproj{\varphi^{\N}\times \zp^{\times}}{\Oe}{p}$ and $\cdetaledvproj{\varphi^{\N}\times \zp^{\times},\, \varphi^{\N}}{\Oe}{p}$ coincide. We have the following list of properties

\begin{enumerate}[itemsep=0mm]
	\item  The inclusion $\Oe\subset\Oehat$ is a continuous morphism of $(\varphi^{\N}\times \cg_{\qp})$-rings for the induced topology on both rings.
	
	\item  The inclusion $\Oe\subset\Oehat$ is a continuous morphism of $\varphi^{\N}$-rings for the $p$-adic topology on both rings. 
	
	\item The monoid $\varphi^{\N}$ is a normal submonoid of $(\varphi^{\N}\times \cg_{\qp})$ and the quotient identifies to $\cg_{\qp}$.
	
	\item The inclusion $\zp\subseteq \Oehat^{\varphi=\hrm{Id}}$ is an equality.
		
	\item The functor $\hrm{triv}$ send $\cdetaledvproj{\varphi^{\N}\times\zp^{\times},\, \varphi^{\N}}{\Oe}{p}$ to $\cdetaledvproj{\varphi^{\N}\times \cg_{\qp},\,\varphi^{\N}}{\Oe}{p}$.
	
	\item The inclusion $\Oe \subset \Oehat$ is continuous for induced topology on $\Oehat$ and $(\varphi^{\N}\times \cg_{\qp})$-equivariant.

	\item The induced topology on $\zp$ from $\Oehat$ is the $p$-adic.
	
	\item The inclusion $\zp \subset \Oehat$ is faithfully flat and $p$ is irreducible in $\zp$.
	
	\item We have $\hrm{K}_0(\mathbb{F}_p)=\Z$.
		
	\item The group $\hrm{H}^1_{\hrm{cont}}(\varphi^{\N},E^{\sep})$ vanishes for the discrete topology on $E^{\sep}$ which is the quotient topology $\sfrac{\Oehat}{p\Oehat}$ coming from the $p$-adic topology.
	
	\item For every $D$ in $\cdetaleproj{\varphi^{\N}\times \cg_{\qp},\,\varphi^{\N}}{E^{\sep}}$, the comparison morphism $$E^{\sep} \otimes_{\mathbb{F}_p} \hrm{Inv}(D) \rightarrow D$$ is an isomorphism.
\end{enumerate} Again, all but the last two properties are derived from the rings' constructions : they are the case Fontaine calls "$M$ est tué par $p$" in \cite[Proposition 1.2.6]{Fontaine_equiv} that we cannot bypass. The Proposition \ref{ex_cdetaledvproj} for both topologies on $\Oe$ and $\Oehat$ (thanks to points 1 and 2) and Corollary \ref{inv_cdvdetaleproj_dévissage} for the submonoid $\varphi^{\N}$ acting on $\Oehat$ allows to define and restrict

\begin{center}
	\begin{tikzcd}
		\mathbb{V}\,: \, &  \cdetaledvproj{\varphi^{\N}\times \zptimes}{\Oe}{p} \ar[r,"\hrm{triv}"] &  \cdetaledvproj{\varphi^{\N}\times \cg_{\qp}, \,\varphi^{\N}}{\Oe}{p} \ar[d,"\hrm{Ex}"] \\
		& \cdetaledvproj{\cg_{\qp}}{\zp}{p} &  \cdetaledvproj{\varphi^{\N}\times \cg_{\qp}, \,\varphi^{\N}}{\Oehat}{p} \ar[l,"\hrm{Inv}"]
	\end{tikzcd}
\end{center}


\vspace{0.2 cm}
On the way, we gathered enough properties to highlight that the functors $\mathbb{D}$ and $\mathbb{V}$ are quasi-inverse. Remember that corollary \ref{inv_cdvdetaleproj_dévissage} does not only imply properties of $\hrm{Inv}(D)$ but also that for any considered module with finite projective $(p,\mu)$-dévissage, the comparison morphism is an isomorphism. For any representation $V$ in $\cdetaledvproj{\cg_{\qp}}{\zp}{p}$, we use the natural comparison isomorphism for $\Oehat\otimes_{\Oe}\mathbb{D}(V)$ then for $\Oehat\otimes_{\zp} V$ to obtain a natural isomorphism

$$\Oehat \otimes_{\zp} \mathbb{V}(\mathbb{D}(V)) \cong \Oehat \otimes_{\Oe}\mathbb{D}(V) \cong \Oehat\otimes_{\zp} V.$$ A similar use of such isomorphisms establishes that $\mathbb{D}$ and $\mathbb{V}$ are quasi-inverse. We obtain Fontaine equivalence.

\begin{rem}\label{rem_h1}
	Using the $p$-adic topology is crucial. Indeed, the cohomological condition $\hrm{H}^1_{\hrm{cont}}(\ch_{\qp},E^{\sep})=\{0\}$ is not verified for the $X$-adic topology on $E^{\sep}$. The Artin-Schreier theory applied to the operator $\wp(x)=x^p-x$ on $E^{\sep}$ produces\footnote{Look at \cite[Exercices 2 and 3, Chapter IV, \S 3]{neukirch}.} an abelian extension $F|E$ of exponent $p$ and an isomorphism $$\quot{E}{\wp(E)} \rightarrow \hrm{Hom}_{\hrm{cont}}\left(\gal{F}{E},\mathbb{F}_p\right).$$ One can check that the inclusion $$\mathbb{F}_p \oplus \bigoplus_{n\geq 1, \,p \,\nmid \,n} \mathbb{F}_p X^{-n} \subseteq E$$ is a section of the projection $E\rightarrow \sfrac{E}{\wp(E)}$. Fix $\ell$ a prime different from $p$ and fix $$\chi_n \,: \, \ch_{\qp} \cong \cg_E \rightarrow \gal{F}{E} \rightarrow \mathbb{F}_p$$ the continous morphism corresponding to $X^{-\ell^n}$. The map $$\ch_{\qp} \rightarrow E, \,\,\, \sigma \mapsto \sum_{n\geq 0} \chi_n(\sigma) X^n$$ is a group morphism, continuous for the $X$-adic topology. It furnishes a $\ch_{\qp}$-continuous cocycle in $E^{\sep}$, which is not a coboundary because it doesn't factor through a finite quotient.
\end{rem}

\clearpage

\appendix

\section{Study of modules with finite projective dévissage}

This appendix is devoted to establish two strong theorems about the structure of $R$-modules with finite projective $(r,\mu)$-dévissage, prior to the additional monoid actions this article studies. As proofs are well-ordered successions of homological algebra arguments, we choose to encapsulate them into this appendix so that only a wishful reader might look into them.

\begin{defi}
	Let $M$ be an $R$-module and $r\in R$. We say that $M$ has \textit{finite projective $(r,\mu)$-dévissage} if each subquotient $r^n M/r^{n+1}M$ is finite projective of constant rank as an $R/r$-module.
	
	We say that $M$ has finite projective $(r,\tau)$-dévissage if each $M[r^{n+1}]/M[r^n]$ is finite projective of constant rank as an $R/r$-module.
\end{defi}

	\begin{theo}\label{complet_pf_1}
		Let $(R,r)$ be a dévissage setup. Then for every $R$-module $M$ the following are equivalent:	
		\begin{enumerate}[label=\roman*),itemsep=0mm]
			\item $M$ is $r$-adically complete and separated with finite projective $(r,\mu)$-dévissage.

			\item $M$ is finitely presented with finite projective $(r,\mu)$-dévissage and bounded $r^{\infty}$-torsion.
			
			\item There exists $N\geq 1$, a finite projective $R$-module of constant rank $M_{\infty}$ and an $r^N$-torsion $R$-module with finite projective $(r,\mu)$-dévissage $M_{\hrm{tors}}$ such that 		
			$$M\cong M_{\infty} \oplus M_{\hrm{tors}}.$$
		\end{enumerate}
	\end{theo}
	\begin{proof}
		\underline{$iii)\implies ii)$:} any finitely presented $R/r$-module is finitely presented as an $R$-module. Any extension of finitely presented $R$-modules is finitely presented. Because each $r^n M_{\hrm{tors}}/r^{n+1}M_{\hrm{tors}}$ is finite projective over $R/r$, hence finitely presented, this implies that $M_{\hrm{tors}}$ is finitely presented. Both terms are finitely presented, so $M$ is also. 
		
		The $(r,\mu)$-dévissage commutes with direct sums so it only remains to prove that $M_{\infty}$ has finite projective $(r,\mu)$-dévissage. The first term $M_{\infty}/r M_{\infty}\cong R/r\otimes_R M_{\infty}$ is finite projective. Because $M_{\infty}$ is flat and $R$ is $r$-torsion-free, the module $M_{\infty}$ is also $r$-torsion-free. We obtain that each $r^n M_{\infty}$ is isomorphic to $M_{\infty}$ and conclude for the other terms of the $(r,\mu)$-dévissage.
		
		The boundedness of torsion follows from the $r^N$-torsion property of $M_{\hrm{tors}}$.
		
		\underline{$ii) \implies i)$:} we first consider a finite type module $M$ with a surjection $f:R^a \twoheadrightarrow M$. We obtain an exact sequence of projective systems $$0 \rightarrow (\hrm{Ker}(f \modulo{r^n}))_{n\geq 1} \rightarrow ((R/r^n)^a)_{n\geq 1} \rightarrow (M/r^n M)_{n\geq 1} \rightarrow 0.$$ Passing to the limit and using the completeness and separatedness of $R$, we obtain an exact sequence $$R^a \rightarrow \lim \limits_{\longleftarrow} M/r^n M \rightarrow \hrm{R}^1 \lim \limits_{\longleftarrow} \, \hrm{Ker}(f\modulo{r^n}).$$  As $f$ is surjective, $$\hrm{Ker}(f\modulo{r^n})=\quot{(\hrm{Ker}(f)+r^n R^a)}{r^n R^a} \cong \quot{\hrm{Ker}(f)}{\hrm{Ker}(f)\cap r^n R^a}.$$ The transition maps between these kernels identify to the quotient maps. The projective system of kernels is therefore Mittag-Leffler, which implies the vanishing of $\hrm{R}^1 \lim$. 
		
		Now fix a presentation $R^b\xrightarrow{g} R^a \xrightarrow{f} M \rightarrow 0$. Using the previous paragraph on both $R^a \xrightarrow{f} M$ and $R^b \xrightarrow{g} \hrm{Ker}(f)$, we obtain two exact sequences $$\lim \limits_{\longleftarrow} \, \hrm{Ker}(f\modulo{r^n}) \rightarrow R^a \rightarrow \lim \limits_{\longleftarrow} M/r^n M \rightarrow 0$$ $$R^b \rightarrow \lim \limits_{\longleftarrow} \hrm{Ker}(f)/r^n \hrm{Ker}(f) \rightarrow 0.$$ If we had $\lim \limits_{\longleftarrow} \,\hrm{Ker}(f)/r^n \hrm{Ker}(f)  \,= \,\lim \limits_{\longleftarrow}\, \hrm{Ker}(f \modulo{r^n})$, it would show that $\lim \limits_{\longleftarrow} M/r^n M$ have the same presentation than $M$, hence that the latter is $r$-adically complete and separated.
		
		Using the long $\hrm{Tor}$-exact sequence for $R/r^n$ and the short exact sequence given by $f$, we have $$0 \rightarrow \hrm{Tor}_1(R/r^n, M) \rightarrow R/r^n \otimes_R \hrm{Ker}(f) \rightarrow \hrm{Ker}(f \modulo{r^n})\rightarrow 0.$$ As $R$ is $r$-torsion free, $R/r^n$ has a free resolution by $\cdots 0 \rightarrow R \xrightarrow{r^n \times} R$ which gives $\hrm{Tor}_1(R/r^n,M)\cong M[r^n]$. We also have maps of resolutions \begin{center}\begin{tikzcd} \cdots \ar[r] & 0 \ar[r] & R \ar[r,"r^{n+1}\times"] \ar[d,"r\times"'] & R \ar[r,two heads] \ar[d,equal] & R/r^{n+1} \ar[d,two heads] \\ \cdots \ar[r] & 0 \ar[r] & R \ar[r,"r^{n}\times"] & R \ar[r,two heads] & R/r^{n} \end{tikzcd}\end{center} which, associated with the naturality of long $\hrm{Tor}$-exact sequences, provide an exact sequence of projective systems $$0 \rightarrow \lim \limits_{\substack{\longleftarrow\\ m\, \mapsto\, r m}} M[r^n] \rightarrow \lim \limits_{\longleftarrow} \hrm{Ker}(f)/r^n \hrm{Ker}(f) \rightarrow \lim \limits_{\longleftarrow} \hrm{Ker}(f\modulo{r^n}) \rightarrow \hrm{R}^1 \lim \limits_{\substack{\longleftarrow\\ m\, \mapsto\, r m}} M[r^n].$$ Say that $M$ has $r^{\infty}$-torsion bounded by $r^N$. All computations of the projective system  $(M[r^n])_{n\geq 1, \, m\, \mapsto \, rm}$ can be done with $ (M[r^{nN}])_{n\geq 1, \, m\, \mapsto \, r^N m}$ hence with $ (M[r^N])_{n\geq 1, \, m\, \mapsto \, 0}$. The $\hrm{R}^1 \lim \limits_{\longleftarrow}$ of this last system vanishes and we obtain $\lim \limits_{\longleftarrow} \hrm{Ker}(f)/r^n \hrm{Ker}(f)  = \lim \limits_{\longleftarrow} \hrm{Ker}(f \modulo{r^n})$ as required.

		\underline{$i)\implies iii)$:} because all maps $r^n M/r^{n+1} M \xrightarrow{r\times} r^{n+1}M/r^{n+2}M$ are surjective, the rank of $r^n M/r^{n+1} M$ is decreasing thus stabilises for $n\gg 0$. Take $N$ such that it stabilises. First, suppose it has been proven that $r^N M$ is finite projective of constant rank over $R$. The exact sequence $$0\rightarrow M[r^N] \rightarrow M \xrightarrow[]{r^N \times} r^N M \rightarrow 0 $$ splits, which first proves that $M[r^N]$ has finite projective dévissage, then the required decomposition.
		
		
		We only need to prove that $r^N M$ is finite projective of constant rank. It verifies the hypothesis of $i)$ with dévissage subquotients being of the same rank. To lighten the notations, we denote by $M$ such module and devote the end of the proof to showing that it is finite projective of constant rank.

		We first prove by induction that $M/r^n M$ is finite projective of constant rank over $R/r^n$. Because our dévissage subquotients are finite projective with same rank functions, every surjective arrow $$r^n M/r^{n+1} M \xrightarrow{r\times} r^{n+1}M/r^{n+2}M$$ is an isomorphism. We obtain the following exact sequence: $$0 \rightarrow \quot{M}{rM} \xrightarrow{f_n:= (r\times) \, \modulo{r^n}} \quot{M}{r^{n+1} M} \rightarrow \quot{M}{r^n M} \rightarrow 0.$$ Suppose it is proved that $M/r^n M$ is finite projective as an $\sfrac{R}{r^n}$-module. Take $t \in R/r^{n+1}$ such that the localisation of $M/r^n M$ at $(t\modulo{r^n})$ is free over $\left(R/r^n\right)_{(t\modulo{r^n})}$ and a basis $(e_i)_{1\leq i\leq b}$. Fix $(\tilde{e}_i)_{1\leq i\leq b}$ any lifting to $M/r^{n+1} M$. Because $M$ is finitely generated, we can apply Nakayama lemma to the ring $\left(R/r^n\right)_{t}$, the module $\left(M/r^{n+1} M\right)_t$ and the element $r$ in the Jacobson radical to prove that $(\tilde{e}_i)$ is a generating family. Take a relation $\sum t_i \tilde{e}_i\!=\!0$. By reducing modulo $r^n$, we can prove that each $t_i$ is a multiple of $r^n$; write $t_i\!=\!r^n t'_i$. The map $$\left(\quot{M}{rM}\right)_t \xrightarrow{(f_n)_t} \left(\quot{M}{r^{n+1}M}\right)_t, \,\,\, [m]\mapsto [r^n m]$$ is an injection onto $\left(r^n M/r^{n+1} \right)_t$. Hence, we obtain that $$(f_n)_t^{-1}\bigg(\bigg[\sum_i r^n t'_i \widetilde{e_i}\bigg]\bigg) = \bigg[\sum_i t'_i \widetilde{e_i}\bigg]$$ is zero thanks to injectivity. Each $(t'_i\modulo{r})$ is zero, hence so are the $t_i$'s. We proved that $M/r^{n+1}M$ is locally free with same local rank than $M/r^n M$, hence of constant rank.
		
		Now, fix a an expression as direct summand $\iota_1:(R/r)^q \xrightarrow{\sim} M/rM \oplus M'_1$. We construct by induction a sequence of $R/r^n$-modules $M'_n$ and of isomorphisms $\iota_n :  (R/r^n )^q \xrightarrow{\sim} M/r^n M \oplus M'_n$ such that each composition $$  M/r^{n+1} M \oplus M'_{n+1} \xrightarrow{\iota_{n+1}^{-1}} \left(\quot{R}{r^{n+1}}\right)^q \rightarrow \left(\quot{R}{r^n}\right)^q \xrightarrow{\iota_n}  \quot{M}{r^n M} \oplus M'_n$$ sends $M'_{n+1}$ to $M'_n$ and coïncide with the reduction on $M/r^{n+1} M$. 
		
		Suppose that $\iota_n$ is constructed. Fix a lift $a_{n+1} \,: \,(R/r^{n+1})^q \longrightarrow M/r^{n+1}M$ of the projection $$a_n \,: \,\left(\quot{R}{r^n}\right)^q \xrightarrow{\iota_n} \quot{M}{r^n M} \oplus M'_n \twoheadrightarrow \quot{M}{r^n M},$$  which is still surjective by Nakayama. By splitting this surjection thanks to the projectivity of $\left(M/r^{n+1} M\right)$, we obtain an isomorphism $$j\, :\,  (R/r^{n+1})^q \cong (M/r^{n+1}M) \oplus M'_{n+1}$$ where $M'_{n+1}=\hrm{Ker}(a_{n+1})$. The module $M'_{n+1}$ is sent to $M'_n$ by construction, and is even sent unto. Take $(m,m') \in (M/r^{n+1} M) \oplus M'_{n+1}$ sent to a fixed $m'_0\in M'_n$. Because $a_{n+1}$ lifts $a_n$, the element $m$ belongs to $r^n M/r^{n+1}M$. Hence, its image in $(M/r^n M) \oplus M'_n$ is zero and $m'$ is also sent to $m'_0$.	
		
		By construction, we write the can write the composition $$ \quot{M}{r^{n+1}M} \subset \quot{M}{r^{n+1}M} \oplus M'_{n+1} \xrightarrow{j^{-1}} \left(\quot{R}{(r^{n+1})}\right)^q \rightarrow \left(\quot{R}{(r^n )}\right)^q \! \xrightarrow[\sim]{\iota_n}  \quot{M}{r^n M} \oplus M'_n$$ as $(\modulo{r^n}, h)$. 
		
		By projectivity of $M/r^{n+1}M$, we are able to lift $h$ to a map $H :  M/r^{n+1} M \rightarrow M'_{n+1}$. The morphism $$\iota_{n+1} \,: \, \left(\quot{R}{r^{n+1}}\right)^q \xrightarrow{j} \quot{M}{r^{n+1}M} \oplus M'_{n+1} \xrightarrow{\begin{psmallmatrix} \hrm{Id} & 0 \\ -H & \hrm{Id} \end{psmallmatrix} } \quot{M}{r^{n+1}M} \oplus M'_{n+1}$$ is one lift we looked for.
		
		Taking the limit of the $\iota_n$, we obtain an isomorphism $$\left(\lim \limits_{\longleftarrow} \quot{R}{r^n R} \right)^q \xrightarrow{\sim} \left(\lim \limits_{\longleftarrow} \quot{M}{r^n M} \right) \oplus \left(\lim \limits_{\longleftarrow} M'_n\right)$$ which is an expression of $M$ as direct summand of a free $R$-module because $R$ and $M$ are both $r$-adically complete and separated. Thanks to $r$-adic completion of $R$, we have $r\in \hrm{Jac}(R)$. Hence, for each closed point $x\in \hrm{Spec}(R)$, $$\kappa(x)\otimes_R M \cong \kappa(x)\otimes_{\sfrac{R}{r}} \quot{M}{rM}.$$ As we knew that $M/rM$ is of constant rank and because $M$ is locally free, $M$ is also of constant rank.
\end{proof}

These finitely presented modules with finite projective $(r,\mu)$-dévissage also have nice behavior with respect to $\hrm{Tor}$ functor and nice $(r,\tau)$-dévissage.

\begin{lemma}\label{tor_torsion_gen}
	Let $(R,r)$ be a dévissage setup. Let $N\geq 1$ and $M$ be a $r^N$-torsion $R$-module with finite projective $(r,\mu)$-dévissage. Let $P$ a $r$-torsion-free $R$-module, we have $$\hrm{Tor}^{R}_1(M,P)=\{0\}.$$
\end{lemma}
\begin{proof}
	\underline{The $r$-torsion case:} in this case, $M$ is in particular a flat $R/r$-module. Use the $\hrm{Tor}$ spectral sequence (cf. \cite[\href{https://stacks.math.columbia.edu/tag/068F}{Tag 068F}]{stacks-project}) to obtain a convergent spectral sequence whose second page is $$\hrm{Tor}^{R/r}_i\left(\hrm{M,Tor}^{R}_j\left(\quot{R}{r},P\right)\right) \Rightarrow \hrm{Tor}^{R}_{i+j}(M,P).$$ Because $M$ is flat, it degenerates with non zero terms lying at $i=0$, which means that $$\hrm{Tor}^{R}_{\bullet}(M,P) = \hrm{H}_{\bullet}\left(M \otimes_{\sfrac{R}{r}}\hrm{Tor}^{R}_{\bullet}\left(\sfrac{R}{r},P\right)\right).$$ Use the projective resolution of $R/r$ $$\cdots \rightarrow 0 \ \rightarrow R \xrightarrow[]{r\times} R$$ to compute that $\hrm{Tor}^{R}_{1}\left(\sfrac{R}{r},P\right)=P[r]$, which is zero because $P$ is $r$-torsion-free.
	
	\underline{General case:} it is obtained by induction on $N$, making a dévissage with $rM$.
\end{proof}

\begin{prop}\label{complet_pf_2}
	Let $(R,r)$ be a dévissage setup.
\begin{enumerate}[itemsep=0mm]
	\item For any finitely presented $R$-module $M$ with with finite projective $(r,\mu)$-dévissage and bounded $r^{\infty}$-torsion, for any $r$-torsion-free $R$-module $P$, the group $\hrm{Tor}^R_1(M,P)$ vanishes.

	\item Any finitely presented $R$-module with finite projective $(r,\mu)$-dévissage and bounded $r^{\infty}$-torsion has finite projective $(r,\tau)$-dévissage.
\end{enumerate}
\end{prop}
\begin{proof}
	\begin{enumerate}[itemsep=0mm]
	\item  Fix an expression $M\!=\!M_{\infty}\oplus M_{\hrm{tors}}$ such that $M_{\infty}$ is finite projective of constant rank over $R$ and $M_{\hrm{tors}}$ is of $r^N$-torsion with finite projective $(r,\mu)$-dévissage. Let $P$ be a $r$-torsion-free $R$-module. Thanks to Lemma \ref{tor_torsion_gen}, we have that $\hrm{Tor}^R_1(M_{\hrm{tors}},P)$ vanishes. Moreover $\hrm{Tor}^R_1(M_{\infty},P)$ vanishes thanks to $M_{\infty}$'s flatness.
	
	\vspace*{-.1cm}
	
	\item We first prove that for any $R$-module $Q$ such that $Q/rQ$ is finite projective as an $R/r$-module and any $R/r$-module $L$, we have $$\hrm{Ext}^2_R(Q,L)=\hrm{Ext}^1_{R/r}(Q[r],L).$$ Remark that $\hrm{Hom}_R(-,L)=\hrm{Hom}_{R/r}(\sfrac{R}{r}\otimes_R -, L)$ and that the tensor product sends projective modules to projective modules. The Grothendieck spectral sequence (see \cite[\href{https://stacks.math.columbia.edu/tag/015N}{Tag 015N}]{stacks-project}) produces a spectral sequence converging to $\hrm{Ext}^{\bullet}_R(Q,L)$ whose second page is $$E_2^{i,j}:=\hrm{Ext}^i_{R/r}\left(\hrm{Tor}^R_j\left(\quot{R}{r},Q\right),L\right).$$ Because $R$ is $r$-torsion-free, $\left(\cdots \rightarrow 0  \rightarrow R \xrightarrow{r\times} R\right)$ is a projective resolution of $R/r$; it implies that $\hrm{Tor}^R_j(\sfrac{R}{r},Q)=0$ for $j\geq 2$ and that $$\hrm{Tor}^R_0\left(\quot{R}{r},Q\right)= \quot{Q}{rQ} \,\text{ and } \hrm{Tor}^R_1\left(\quot{R}{r},Q\right)=Q[r].$$ At this point, we use that $Q/rQ$ is finite projective to deduce that $E_2^{i,j}$ is concentrated in degrees $(0,0)$ and $\Z\times\{1\}$. Thus, the spectral sequence degenerates on the second page and we get our result. 
	\vspace{-0.05cm}
	
	Applying this with $Q$ ranging over all subquotients in the $(r,\mu)$-dévissage of $M$, we obtain $$\forall n\geq 1, \, \forall L \in R/r\,\text{-}\hrm{Mod}, \,\,\,\hrm{Ext}^{2}_R(\sfrac{r^n M}{r^{n+1} M},L)=\hrm{Ext}^{1}_{R/r}(\sfrac{r^n M}{r^{n+1} M},L),$$ which vanishes thanks to the projectivity of the $(r,\mu)$-dévissage. By long exact sequences of $\hrm{Ext}_R^{\bullet}$, we obtain $$\forall n\geq 1, \, \forall L \in R/r\,\text{-} \hrm{Mod},\,\,\,\hrm{Ext}^2_R(\sfrac{M}{r^n M},L)=\{0\}.$$ Using again our result for $Q=M/r^n M$, we obtain $$\forall L\in R/r\,\text{-}\hrm{Mod}, \,\,\, \hrm{Ext}^{1}_{R/r}\left(\left(\quot{M}{r^n M}\right)[r],L\right)=\{0\},$$ i.e that $(M/r^n M)[r]$ is a projective $R/r$-module. Thanks to the exact sequence $$0 \rightarrow \left(\quot{M}{r^n M}\right)[r] \rightarrow \quot{M}{r^n M} \xrightarrow{r\times} \quot{rM}{r^n M} \rightarrow 0$$ where the last two terms are finitely presented $R$-modules, we deduce that $(M/r^n M)[r]$ is a finite type $R$-module, then finite projective over $R/r$. 
	
	Pick $N$ such that $r^N M/ r^{N+1} M \xrightarrow[r\times]{\sim} r^{N+1} M/r^{N+2} M$ as in $i) \Rightarrow iii)$ of Theorem \ref{complet_pf_1}. Using that $M$ is $r$-adically separated, we deduce that $r^N M$ is torsion-free. It implies that $M[r] \rightarrow \left(M/r^{N+1} M\right)[r]$ is an injection, for which $r^N M/r^{N+1} M$ is a complement of the image. We showed that $M[r]$ is direct summand of a finite projective $R/r$-module of constant rank, with a complement of constant rank. Hence $M[r]$ is also finite projective of constant rank. 
	
	From $M$ being complete and separated with finite projective $(r,\mu)$-dévissage, we deduce the same properties for $r^n M$. Because the map $$M[r^{n+1}] \xrightarrow{r^n \times} r^n M[r^{n+1}]=(r^n M)[r]$$ has $M[r^n]$ for kernel, the projectivity of $M[r^{n+1}]/M[r^n]$ is obtained from the projectivity of $r^n M$'s $r$-torsion.
	
	\qedhere\end{enumerate}
\end{proof}

Our second theorem concerns only certain rings. For a discrete valuation ring $A$ with uniformiser $a$, the structure theorem for finitely generated modules over principal ideal domains decomposes such modules as a finite sum of a free module over $A/a$, a free module over $A/a^2$, etc, and a free module over $A$. Fontaine's rings for classical $(\varphi,\Gamma)$-modules are indeed discrete valuation rings, but not their multivariable variants. As evoked in \cite[Lemma 2.3]{zabradi_equiv}, for his ring $\Oed$ with residual ring $E_{\Delta}$ at $p$, it is not clear wether all finite projective module over $E_{\Delta}$ are free\footnote{Browse \cite{cesnavicius_torsors}, \cite{rao_bhatwadekar} and \cite{rao_projective} for more details about these problems. I thank K. Česnavičius for these references.}. However, the $0^{\hrm{th}}$ algebraic $K$-group of $E_{\Delta}$ vanishes and I found out that it is sufficient to garantuee a similar decomposition for our modules with finite projective dévissage.

\begin{lemma}\label{lemma_pour_ko_1}
	Let $(R,r)$ be a dévissage setup.
	
	\begin{enumerate}[itemsep=0mm]
		\item Let $n \in  \llbracket 1, \infty \rrbracket$ and $M$ be a projective $R/r^n$-module\footnote{by convention $R/r^{\infty}:=R$.}. For each integer $k<n$, the epimorphism $$\quot{M}{rM}\xrightarrowdbl{r^k\times} \quot{r^k M}{r^{k+1} M}$$ is an isomorphism.
		
		\item Let $n \in  \llbracket 1, \infty \rrbracket$ and $M$ be a projective $R/r^n$-module. For each integer $k<n$, we have $$M[r^k]=r^{n-k}M.$$
		
		\item For every stably free finite projective $R/r$-module $M$ and $n\in \llbracket 1, \infty \rrbracket$, there exists a unique up to isomorphism finite projective $R/r^n$-module $M_{(n)}$ such that $M_{(n)}/r M_{(n)} \cong M$. If $M$ is of constant rank, then so is $M_{(n)}$.
\end{enumerate}
\end{lemma}
\begin{proof}
	\begin{enumerate}[itemsep=0mm]
		\item Because $R$ is $r$-torsion-free, this is true for $M=R$, hence for any free $R/r^nR$-module. Moreover,  each functor $M\mapsto r^k M/r^{k+1} M$ commutes to direct sums, which concludes for general projective modules.
		
		\item Same reasoning using that both functors $M\mapsto M[r^k]$ and $M\mapsto r^{n-k}M$ commute to direct sums.
		
		\item Let $M$ be a finite projective $R/r$-module, which is stably free. Fix a presentation $M\oplus (R/r)^k=(R/r)^d$. The module $M$ can be expressed as the kernel of the projection $\pi \,: \, (R/r)^d \twoheadrightarrow (R/r)^k$, which can easily be lifted. Choose a lift $\pi_{(n)} \,: \,(R/r^n)^d \twoheadrightarrow (R/r^n)^k$. The restriction-corestriction of $\pi_{(n)}$ fits in the following commutative diagram with exact rows:
		
		\begin{center}
			\begin{tikzcd}
				0 \ar[r] & \left(\quot{rR}{r^n R}\right)^d \ar[r] \ar[d,"\text{restriction of }\pi_{(n)}"',two heads] & \left(\quot{R}{r^n}\right)^d \ar[r,"\modulo{r}"] \ar[d,"\pi_{(n)}"'] & \left(\quot{R}{r}\right)^d \ar[r] \ar[d,"\pi"'] & 0 \\
				0 \ar[r] & \left(\quot{rR}{r^n R}\right)^k \ar[r] & \left(\quot{R}{r^n}\right)^k \ar[r,"\modulo{r}"] & \left(\quot{R}{r}\right)^k \ar[r] & 0
			\end{tikzcd}
		\end{center} The snake lemma then proves that $M_{(n)}:=\hrm{Ker}(\pi_{(n)})$ surjects on $M$ by reduction modulo $r$. Moreover it is finite projective as the kernel of an epimorphism between finite projective modules. For any $n\in \llbracket 1,\infty\rrbracket$, the rank of finite projective $R/r^n$ can be checked on closed points in $\hrm{V}(r)$, i.e. on the reduction modulo $r$. Hence, we obtain that $M_{(n)}$ is of constant rank if $M$ is also.
		
		Take $M_{(n)}$ and $M'_{(n)}$ two lifts of $M$. By projectivity of $M_{(n)}$ applied to $$M_{(n)} \rightarrow M_{(n)}/r M_{(n)} \xrightarrow[\iota]{\sim} M'_{(n)}/r M'_{(n)},$$ we lift it to a morphism $f: M_{(n)}\! \rightarrow\! M'_{(n)}$ such that $\left(f\modulo{r}\right)\!=\!\iota$. For each $k<n$, let $f_k$ be the restriction-coretriction to $r^k M_{(n)}$ and $r^k M'_{(n)}$. Consider the diagram 
		
		\begin{center}
			\begin{tikzcd}
				M_{(n)} \ar[rrr,"f"]\ar[ddd,"r^k \times",two heads] \ar[dr] & & & M'_{(n)} \ar[ddd,"r^k \times"] \ar[dl]  \\
			 & \quot{M_{(n)}}{rM_{(n)}} \ar[r,"\iota","\sim"'] \ar[d,"r^k \times ","\sim" labl]& \quot{M'_{(n)}}{rM'_{(n)}} \ar[d,"r^k \times ","\sim" labl] & \\
			 & \quot{r^k M_{(n)}}{r^{k+1} M_{(n)}}  & \quot{r^k M'_{(n)}}{r^{k+1}M'_{(n)}} & \\
				r^k M_{(n)} \ar[rrr,"f_k"] \ar[ur] & & & r^k M'_{(n)} \ar[ul] 
			\end{tikzcd}
		\end{center} where the envelopping square commutes (check it, this is not by definition), and the three trapezes commute.  Consider the remaining component. We invert its vertical arrows and want to check that it commutes; this can be done after precomposition by left vertical arrow (labelled as epimorphism) and then it's diagram chase using the commutations highlighted justf before.

		It proves that $\left(f_k \modulo{r}\right)$ identifies to $\iota$. All $\left(f_k\modulo{r}\right)$ are isomorphisms so $f$ is an isomorphism.
	\end{enumerate}
\end{proof}

\begin{lemma}\label{lemma_pour_ko_2}
	Let $(R,r)$ be a dévissage setup. Let $n\geq k \geq 1$ be integers, let $N$ be a finite projective $R/r^n$-module and $M$ a finite projective $R/r^k$-module. The group $\hrm{Ext}^1_{R/r^n}(M,N)$ vanishes.
\end{lemma}
\begin{proof}	
	We note that $$\hrm{Hom}_{R/r^n}(M,-)=\hrm{Hom}_{R/r^k}(M,-[r^k])= \hrm{Hom}_{R/r^k}\left(M,\hrm{Hom}_{R/r^n}\left(\quot{R}{r^k},-\right) \right).$$
	
	\noindent Morover, the functor $\hrm{Hom}_{R/r^n}(\sfrac{R}{r^k},-)$ sends injectives to injectives: we have bijections $$\hrm{Hom}_{R/r^k}\left(-,\hrm{Hom}_{\sfrac{R}{r^n}}\left(\quot{R}{r^k},I\right)\right) =\hrm{Hom}_{R/r^n} \left( \quot{R}{r^k}\otimes_{R/r^n} - , I\right)$$ and the second expression underlines the vanishing on $r^k$-torsion complexes.
	
	By Grothendieck spectral sequence, we obtain a converging spectral sequence whose second term is $$E_2^{i,j}=\hrm{Ext}^i_{R/r^k}\left( M, \, \hrm{Ext}^j_{R/r^n}\left(\quot{R}{r^k},N\right)\right) \Rightarrow \hrm{Ext}^{i+j}_{R/r^n}(M,N).$$ Using that $R$ is $r$-torsion-free, we obtain a suitable projective résolution of $R/r^k$ expressed as $$ \cdots \xrightarrow{r^k \times}\quot{R}{r^n} \xrightarrow{r^{n-k} \times}\quot{R}{r^n}\xrightarrow{r^k \times}\quot{R}{r^n}$$ which computes that $$\hrm{Ext}^j_{R/r^n}\left(\quot{R}{r^k},N\right) = \left\{ \begin{array}{ll} N[r^k] \,\,\, \text {if} \,\,\,j=0 \\ \quot{N[r^{n-k}]}{r^k N} \,\,\, \text{if} \,\,\, j=2l+1 \\ \quot{N[r^{k}]}{r^{n-k} N} \,\,\, \text{if} \,\,\, j=2l+2. \end{array}\right.$$ The second point of Lemma \ref{lemma_pour_ko_1} says that these groups vanish for $N$ projective and $j>1$. The spectral sequence degenerates and we obtain that $$\hrm{Ext}^1_{R/r^n}(M,N)=\hrm{Ext}^1_{R/r^k} (M,N[r^k])$$ which vanishes thanks to the projectivity of $M$.
\end{proof}

\begin{theo}\label{complet_pf_3}
	Let $(R,r)$ be a dévissage setup such that $\hrm{K}_0(\sfrac{R}{r})=\Z$. Then the three conditions of Theorem \ref{complet_pf_1} on a $R$-module $M$  are also equivalent to
	
	\begin{enumerate}[label=\roman*),itemsep=0mm]
		\setcounter{enumi}{3}
		\item There exists $N\geq 1$ and an isomorphism $$M\cong M_{\infty} \oplus \bigoplus_{1\leq n \leq N} M_n$$ where $M_{\infty}$ is a finite projective $R$-module of constant rank and each $M_n$ is a finite projective $R/r^n$-module of constant rank.
	\end{enumerate}
\end{theo}
\begin{proof}
	\underline{$iv)\rightarrow iii)$:} fix a decomposition given by $(iv)$. Because $\left(\bigoplus \, M_n\right)$ is $r^N$-torsion and the $(r,\mu)$-dévissage is computed term by term, we only need to prove that each $M_n$ has finite projective $(r,\mu)$-dévissage. Because $M_n$ is finite projective of constant rank over $R/r^n$, so is $M_n/r M_n$ of $R/r$ . Moreover, the second point of Lemma \ref{lemma_pour_ko_1} tells that $M_n/r M_n \xrightarrow{r^n \times} r^k M_n/r^{k+1} M_n$ are isomorphisms which concludes for the whole $(r,\mu)$-dévissage.
	
	\underline{$iii)\implies iv)$:} as this direction is more difficult, we begin by an example. Consider the dévissage setting $(\zp,p)$ and the $\zp$-module $M=\zp/p^2 \zp \oplus \mathbb{F}_p$. It has finite projective $p$-dévissage with $$\quot{M}{pM}=\quot{\zp}{p\zp} \oplus \mathbb{F}_p\cong \mathbb{F}_p^2 \,\, \text{ and } \,\, \quot{pM}{p^2 M}=\quot{p\zp}{p^2 \zp}\cong \mathbb{F}_p.$$ Moreover, the second term of $\sfrac{M}{pM}$ is naturally seen as the kernel of $M/pM \xrightarrow{p\times} \sfrac{pM}{p^2 M}$ or as  $(M[p] + pM)/pM$. Nonetheless, there is no natural lift of $(M[p]+pM)/pM\subset M/pM$ to $M$, nor is there a natural embedding of $\zp/p^2 \zp$ into $M$. In our case $R=\zp$, such lift and embedding always exist thanks to the structure of finite type $\zp$-modules. In general, we must work harder as the lifts we look for won't be given naturally from the $(r,\mu)$-dévissages. 
	
	This example gives a direction: in general each $$\quot{r^n M}{r^{n+1} M} \xrightarrow{r\times} \quot{r^{n+1} M}{r^{n+2} M}$$ splits, giving us a (non canonical) decomposition of $M/rM$ that we might want to lift. Moreover, if $M$ is exactly $r^{N+1}$-torsion, the splitting of $r^N M/r^{N+1} M$ should lift to the finite projective $R/r^{N+1} M$-module (this can also be obtained by analysis of the desired result).
	
	Let's begin the proof. First reduce to decompose the torsion part and only prove that the predicted decomposition exists for $r^N$-torsion modules $M$ with finite projective $(r,\mu)$-dévissage by induction on $N$. For $N=1$, the module $M$ is its first dévissage subquotient hence is finite projective of constant rank. Supoose that we have the result for some $N$ and fix $M$ of $r^{N+1}$-torsion. The last (possibly) non zero term of the $(r,\mu)$-dévissage is $r^N M$ which is finite projective of constant rank over $R/r$. We fix $M_{N+1}$ to be $(r^N M)_{(N+1)}$ given by the third point of Lemma \ref{lemma_pour_ko_1}. We also fix splittings of each arrow in $$\quot{M}{rM} \xrightarrowdbl{r\times} \quot{rM}{r^2 M} \xrightarrowdbl{r\times} \cdots \xrightarrowdbl{r\times} r^N M.$$
	
	By the first point of Lemma \ref{lemma_pour_ko_1}, we have $r^N M_{N+1} \cong M_{N+1}/r^N M_{N+1} \cong r^N M$. By projectivity we can complete the following diagram with an arrow $f$ 
	
	\begin{center}
		\begin{tikzcd}
			M_{N+1} \ar[dd,"r^N \times"'] \ar[drr,"f",dotted] & \\
			& & M \ar[d,two heads] \\
			r^N M_{N+1} \cong r^N M \ar[rr,"\text{chosen splitting}"'] & & \sfrac{M}{r M}
		\end{tikzcd}
	\end{center} Because the kernel of $r^N\times \text{-}$ on $M_{N+1}$ is exactly $r M_{N+1}$ (see the second point of Lemma \ref{lemma_pour_ko_1}), the map $(f \modulo{r})$ identifies to the chosen splitting, hence is injective. Moreover, let $k\leq N$. The restriction-corestriction of $f$ to $r^k M_{N+1}$ and $r^k M$, denoted by $f_k$, fits into the following diagram.
	
	
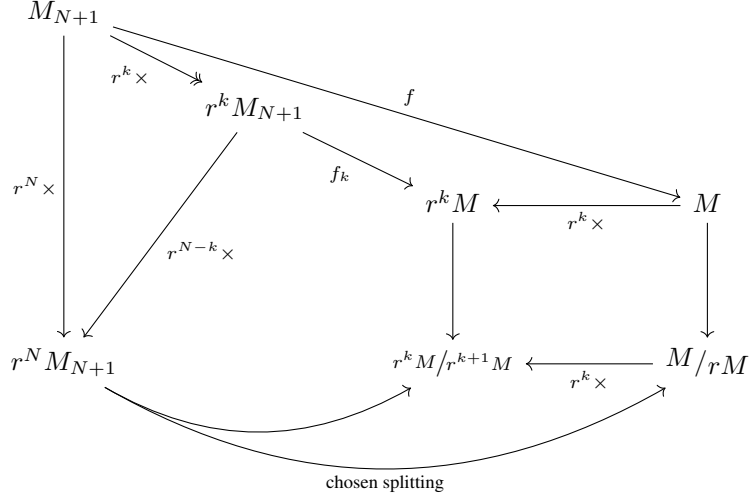
\begin{center}
		\begin{tikzcd}
			M_{N+1} \ar[dddd,"r^N \times"'] \ar[dr,two heads,"r^k \times"'] \ar[ddrrrr,"f"] \\
			& r^k M_{N+1} \ar[dr,"f_k"'] \ar[dddl,"r^{N-k}\times"] \\
			& & r^k M \ar[dd] & & M \ar[ll,"r^k \times"] \ar[dd] \\
			
			& \\
			r^N M_{N+1} \ar[rr,bend right] \ar[rrrr,bend right,"\text{chosen splitting}"'] & & \sfrac{r^k M}{r^{k+1} M} & & \quot{M}{rM} \ar[ll,"r^k\times"]
		\end{tikzcd}
	\end{center} where the small bended arrow is the composition making the bottom triangle commutes; it is therefore a splitting of $r^k M/r^{k+1} M \rightarrow r^N M$. The square, the left and the bottom triangles commute for obvious reasons, the upper quadrilateral commutes (check it) and the envelop of the diagram commutes by construction of $f$. Determining whether the central quadrilateral commutes can be checked after precomposition by the two-headed arrow. Then, use carefully the previous commutations. The restriction-corestriction $f_k$ is therefore obtained using the same construction as $f$, but for $r^k M$ rather than $M$. For the similar reasons, the map $(f_k \modulo{r})$ identifies to the chosen splitting of $r^k M/r^{k+1} M \rightarrow r^N M$ and is injective. All these injectivity properties imply that $f$ is injective. We have obtained an exact sequence $$0 \rightarrow M_{N+1} \rightarrow M \rightarrow \hrm{Coker}(f)\rightarrow 0$$ which invites to look closer to this cokernel. First, the snake lemma applied to
	\begin{center}
		\begin{tikzcd}
			0 \ar[r] & M_{N+1} \ar[r] \ar[d,"r\times"'] & M \ar[r,"f"] \ar[d,"r\times"']  & \hrm{Coker}(f) \ar[r] \ar[d,"r\times"'] & 0\\
			0 \ar[r] & M_{N+1} \ar[r]& M \ar[r,"f"]  & \hrm{Coker}(f) \ar[r]  & 0
		\end{tikzcd}
	\end{center} illustrates that $\hrm{Coker}(f)/r \hrm{Coker}(f)$ is isomorphic to $\hrm{Coker} (f\modulo{r})$ which is the cokernel of the chosen splitting. Hence this quotient is finite projective of constant rank over $R/r$. Similarly for $k\geq 1$, the quotient $r^k \hrm{Coker}(f)/r^{k+1} \hrm{Coker}(f)$ identifies to $\hrm{Coker}(f_k)/r \hrm{Coker}(f_k)$, then to the cokernel of the chosen splitting of $r^k M/r^{k+1}M \xrightarrow{r^{N-k}\times} r^N M$. It is also finite projective of constant rank and the module $\hrm{Coker}(f)$ is of $r^N$-torsion with finite projective $(r,\mu)$-dévissage. We apply the heredity hypothesis to obtain an isomorphism $$\hrm{Coker}(f)\cong \bigoplus_{1\leq n \leq N} M_n$$ with each $M_n$ being finite projective of constant rank over $R/r^n$. Finally, the exact sequence
	
	$$0 \rightarrow M_{N+1} \rightarrow M \rightarrow \bigoplus_{1\leq n\leq N} M_n\rightarrow 0$$ splits because $$\hrm{Ext}^1_{\sfrac{R}{(r^{N+1})}}\left(\bigoplus_{1\leq n \leq N} M_n ,M_{N+1}\right)= \bigoplus_{1\leq n\leq N} \hrm{Ext}^1_{\sfrac{R}{(r^{N+1})}}(M_n ,M_{N+1})$$ which vanishes thanks to Lemma \ref{lemma_pour_ko_2}.
\end{proof}

\begin{coro}\label{Zab_coro}
	Let $\Delta$ be a finite set. Consider the ring $\Oed$ as in \cite{zabradi_equiv} for $K=\qp$ with action of the monoid $\Phi\Gamma_{\Delta}:=\prod_{\alpha \in \Delta} \left(\varphi_{\alpha}^{\N} \times \Gamma_{\alpha} \right)$. For every object of $\detale{\Phi\Gamma_{\Delta}}{\Oed}$, the underlying $\Oed$-module is isomorphic to some
	
	 $$D_{\infty} \oplus \bigoplus_{1\leq n \leq N} D_n$$ where $D_{\infty}$ is a finite projective $\Oed$-module and each $M_n$ is a finite projective $\Oed/p^n$-module.
\end{coro}
\begin{proof}
	Zàbràdi defines $$\Oed=\lim \limits_{\substack{\leftarrow \\ h}} \left(\quot{\zp}{p^h\zp}\right) \llbracket X_{\alpha} \, |\, \alpha \in \Delta\rrbracket [X_{\Delta}^{-1}]$$ where $X_{\Delta}=\prod X_{\alpha}$. From this description, $\Oed$ is $p$-adically separated complete and a domain with $p\neq 0$. From \cite[Proposition 2.2]{zabradi_equiv} and $\Oed$ being a domain, we deduce that $\detale{\Phi\Gamma_{\Delta}}{\Oed}$ and $\detaledvproj{\Phi\Gamma_{\Delta}}{\Oed}{p}$ coincide. Finally, the hypothesis on $K$-theory is verified thanks to \cite[Lemma 2.3]{zabradi_equiv}. Apply Theorem \ref{complet_pf_3}. 
\end{proof}
\vspace{3 cm}

\printbibliography

\end{document}